\documentclass[aap]{imsart}

\usepackage{enumitem}
\usepackage{}
\RequirePackage{amsthm,amsmath,amsfonts,amssymb,mathrsfs,url}
\RequirePackage[numbers]{natbib}
\RequirePackage[colorlinks,citecolor=blue,urlcolor=blue]{hyperref}
\RequirePackage{graphicx}

\startlocaldefs
\theoremstyle{plain}
\newtheorem{theorem}{Theorem}[section]

\newtheorem{asump}{Assumption \normalfont H--\!\!}
\newtheorem{lem}[theorem]{Lemma}
\newtheorem{rem}[theorem]{Remark}
\newtheorem{prop}[theorem]{Proposition}
\newtheorem{cor}[theorem]{Corollary}

\theoremstyle{remark}
\newtheorem{definition}[theorem]{Definition}

\usepackage{pifont}
\endlocaldefs

\begin{document}

\begin{frontmatter}
\title{An explicit Milstein-type scheme for interacting particle systems and McKean--Vlasov SDEs with common noise and non-differentiable drift coefficients}
\runtitle{Milstein-type scheme with non-differentiable drift coefficients}

\begin{aug}
\author[A]{\fnms{Sani}~\snm{Biswas}\ead[label=e1]{sbiswas2@ma.iitr.ac.in}},
\author[B]{\fnms{Chaman}~\snm{Kumar}\ead[label=e2]{chaman.kumar@ma.iitr.ac.in}}
\author[C]{\fnms{Neelima}~\snm{}\ead[label=e3]{neelima\_maths@ramjas.du.ac.in}}
\author[D]{\fnms{Gon\c calo dos}~\snm{Reis}\ead[label=e4]{G.dosReis@ed.ac.uk}\orcid{0000-0002-4993-2672}} 
\and
\author[E]{\fnms{Christoph}~\snm{Reisinger}\ead[label=e5]{christoph.reisinger@maths.ox.ac.uk}}
\address[A]{India Institute of Technology Roorkee, India \printead[presep={,\ }]{e1}}

\address[B]{India Institute of Technology Roorkee, India\printead[presep={,\ }]{e2}}

\address[C]{Delhi University, India\printead[presep={,\ }]{e3}}

\address[D]{Edinburgh University, United Kingdom \printead[presep={,\ }]{e4}}

\address[E]{Oxford University, United Kingdom \printead[presep={,\ }]{e5}}
\end{aug}

\begin{abstract}
We propose an explicit drift-randomised Milstein scheme for both McKean--Vlasov stochastic differential equations and associated high dimensional interacting particle systems with common noise. By using a drift randomisation step in space and measure, we establish the scheme's strong convergence rate of $1$ under reduced regularity assumptions on the drift coefficient: no classical (Euclidean) derivatives in space or measure derivatives (e.g., Lions/Fr\'echet) are required. The main result is established by enriching the concepts of bistability and consistency of numerical schemes used previously for standard SDE.  We introduce certain Spijker-type norms (and associated Banach spaces) to deal with the interaction of particles present in the stochastic systems being analysed. A discussion of the scheme's complexity is provided.
\end{abstract}
\begin{keyword}[class=MSC]
\kwd[Primary ]{65C30, 60H35}
\kwd[; secondary ]{65C05, 65C35}
\end{keyword}

\begin{keyword}
\kwd{stochastic interacting particle systems}
\kwd{McKean--Vlasov equations}
\kwd{common noise}
\kwd{Milstein scheme}
\kwd{non-differentiable drift}
\kwd{drift randomisation}
\kwd{bistability}
\end{keyword}

\end{frontmatter}

	\section{Introduction}
	For a given $T>0$, consider the following  stochastic differential equation (SDE) of McKean--Vlasov type and with common noise,
	\begin{align}
		X_t = X_0 + &\int_0^t b(s, X_s, \mathcal L^1(X_s))ds +  \sum_{\ell=1}^{m_1}  \int_0^t \sigma_1^{\ell} (s, X_s, \mathcal L^1(X_s))dW_s^\ell \notag
		\\
		&+\sum_{\ell=1}^{m_0}\int_0^t \sigma_0^{\ell} (s, X_s,\mathcal L^1(X_s))dW_s^{0,\ell} ,
		\label{eq:mc:intro}
	\end{align}
	almost surely for all $t\in [0, T]$, where $W:=\{W_t\}_{t\in [0,T]}$ and   $W^0:=\{W_t^0\}_{t\in [0,T]}$ are, respectively, $m_1$ and $m_0$ dimensional independent Wiener processes  and $\{\mathcal L^1(X_t)\}_{t\in[0,T]}$ denotes the  stochastic flow of conditional marginal laws of $X:=\{X_t\}_{t\in [0,T]}$ given $W^0$. 
	The initial value  $X_0$ is an $\mathscr{F}^0$- measurable random variable, independent of $W$ and $W^0$. 
	The McKean--Vlasov SDE \eqref{eq:mc:intro} can be viewed as an infinite-dimensional system of particles with $W$ representing the randomness inherent in the individual particle and $W^0$ the randomness common to all the particles. 
	When $\sigma_0 \equiv 0$, the particles are governed by only one source of randomness, $W$, and the stochastic flow $\{\mathcal L^1(X_t)\}_{t\in[0,T]}$ becomes a deterministic one.
	Notice that McKean--Vlasov SDEs are different from standard SDEs due to the dependence of the coefficients on the (conditional) marginal law $\mathcal L^1(X_t)$  of  $X_t$ given $W^0$, which brings additional difficulties.
	
	Due to their wide applications in areas such as Finance, mathematical neuroscience and biology, machine learning and physics --- animal swarming, cell movement induced by chemotaxis, opinion dynamics, particle movement in porous media and electrical battery modelling, self-assembly of particles and dynamical density functional theory 	(see for example \cite{Holm2006,Carrillo2010,Bolley2011,Dreyer2011,Baladron2012,Goddard2012,Kolokolnikov2013,  Bossy2015, Carmona2018-I, Carmona2018-II,  Guhlke2018,Jin2020})--- McKean--Vlasov equations and associated interacting particle systems, with or without common noise, addressed via stochastic systems or associated Fokker Plank equations (\cite{kurtz1999particle,coghi2019stochastic,CoghiNilssen2021}) have gained immense popularity. 	
	
	As in the case of SDEs, explicit solutions of McKean--Vlasov SDEs are typically not available, which necessitates the  development  of numerical schemes to approximate them. 
	The numerical approximation of McKean--Vlasov SDEs can be carried out in two steps, as explained below. 
	\begin{itemize}
		\item 
		As a first step one builds the so-called \textit{interacting particle system}, $\{X^{i,N}\}_{i\in \{1,\ldots,N\}}$, where one replaces the (conditional) marginal law appearing in the coefficients of \eqref{eq:mc:intro} by the empirical law obtained from the particles. Concretely, taking $N$ i.i.d.~copies $\{W^i\}_{i\in \{1,\ldots,N\}}$ of  $W$ and $\{X^i_0\}_{i\in \{1,\ldots,N\}}$ of  $X_0$, 

		one defines the interacting particle system associated with the above McKean--Vlasov SDE by 
			\begin{align}
				X_t^{i, N}=X_0^{i} &+\int_0^tb(s,X_s^{i, N},\mu_s^{X, N}) \, ds+ \sum_{\ell=1}^{m_1} \int_0^t\sigma_1^{\ell} (s,X_s^{i, 
					N},\mu_s^{X, N}) \, dW_s^{i,\ell}\notag
				\\
				&+\sum_{\ell=1}^{m_0}\int_0^t\sigma_0^{\ell}(s,X^{i,N}_s,\mu_s^{X,N}) \, dW_s^{0,\ell},   
				\label{eq:mcparticles:intro}
			\end{align}
			almost surely for any $t\in [0, T]$ and $i\in\{1,\ldots,N\}$,  where  
			\begin{align*}
				\mu_t^{X, N}(\cdot):=\frac{1}{N}\sum_{i=1}^{N}\delta_{X_t^{i, N}}(\cdot)
			\end{align*}
			is the  empirical measure of $N$ particles. 			
			Subsequently, one needs to show, roughly put, that $\textrm{Law}(X^{i,N}_\cdot)$ for some $i$ (fixed) convergences to $\textrm{Law}(X_\cdot)$ of \eqref{eq:mc:intro} as $N\to \infty$; see the seminal work by Sznitman \cite{sznitman1991} {(and Proposition \ref{prop:poc} below)}. 	
				\color{black}

		\item In the second step of approximation, the temporal discretization of the interacting particle system is performed to obtain fully implementable numerical schemes for the McKean--Vlasov SDEs such as Euler-type schemes and Milstein-type schemes. 	
		The main difficultly is to show that all estimates are independent of the number of particles $N$ (see Theorem \ref{thm:mainresult} and Corollary \ref{thm:mainresultCOROLLARY} below). 
		The direct application of results from classic SDE theory do not deliver this independence. 
		
	\end{itemize}
	
	\emph{It is critical to note that the results of this manuscript double for either the numerical approximation of McKean--Vlasov SDEs if one's starting point is \eqref{eq:mc:intro}, or for stand-alone systems of interacting particle SDE systems if one's starting point is \eqref{eq:mcparticles:intro}.} 
	\smallskip

	\color{black}
	Results on the \textcolor{black}{strong} well-posedness and propagation of chaos for McKean--Vlasov SDEs are being extensively researched and we cannot possibly do justice to that growing body of literature. Nonetheless,  we mention some of the milestones and more recent results fitting thematically with our manuscript: 
	for PoC results, one starts from Sznitman's seminal work \cite{sznitman1991} to the monographs \cite{Carmona2018-I, Carmona2018-II} and the recent developments by \cite{DelarueLackerRamanan2019,Lacker2023,LackerLeFlem2023} to mention a few --  overall, there are still gaps in the existing PoC results, especially across dimensions. Concretely, in \cite[Section 3.5]{ChenStockinger2023} the PoC rates across the dimension $d$ is estimated  numerically for a diffusion $\sigma$ of polynomial growth and a drift $b$ of polynomial growth (a setting outside the scope of this work) and the rates estimated are better than the rates found in any theoretical results at present. 	
	For a focus on well-posedness, one starts from the McKean's seminal work  \cite{McKean1966}, to again \cite{Carmona2018-I, Carmona2018-II} and the recent developments \cite{Bauer2018, chaudru2020, C.Kumar2021, Mehri2020, Mishura2020, Salkeld2019,ChenStockinger2023} (and references therein).   
	\color{black}
	
	The  second step of numerical approximation of McKean--Vlasov SDEs saw relatively little development after the first Euler-type scheme was proposed and analysed (in a weak sense) in \cite{Bossy1997}, but experience rapid advances in a string of recent papers,
	 \cite{bao2021-I, bao2021-II, bao2021-III, Chen2021, Chen2023, Kumar2021, C.Kumar2021, Li2022, Reis2019, Reisinger2022}.  
	In particular, \cite{Reis2019} proposed  an Euler-type numerical scheme for the interacting particle system associated with the McKean--Vlasov SDE, and its strong convergence is investigated when the drift coefficient grows super-linearly in the state variable. 
	More precisely, the drift is assumed to be one-sided Lipschitz continuous in the state variable and Lipschitz continuous in the measure variable, while the diffusion coefficient is assumed to be Lipschitz continuous in both the state and measure variables. 
	The rate of strong convergence of the scheme is shown to be equal to $1/2$.   
	In \cite{bao2021-II, Kumar2021}, a Milstein-type scheme for the interacting particle system associated with the McKean--Vlasov SDEs is proposed using the notion of Lions derivatives, introduced by  P.-L.~Lions in his lectures at the Coll\`ege de France and presented in \cite{Cardaliaquet2013}, and its strong convergence is shown with rate $1$.  
	The drift coefficient is assumed to satisfy a one-sided Lipschitz condition in the state variable and a polynomial Lipschitz condition,  and the diffusion coefficient to satisfy Lipschitz condition;
	 both are assumed to be Lipschitz continuous in the measure variable.
	Furthermore, only the diffusion coefficient is required to be once differentiable (in both state and measure variables). 
	In \cite{bao2021-II}, the authors additionally require second order differentiability of the coefficients.    
	In \cite{C.Kumar2021}, an Euler-type scheme and a Milstein-type scheme are developed for the interacting particle system connected with McKean--Vlasov SDEs with common noise, where all the coefficients are allowed to grow super-linearly in the state variable. 
\smallskip

	In this article, we develop a Milstein-type scheme for the interacting particle system corresponding to the McKean--Vlasov SDE without assuming the differentiability of the drift coefficient (in space or measure component), which is therefore more relaxed than the corresponding results in \cite{bao2021-II, Kumar2021, C.Kumar2021}. It is out of the scope of this work to lift the differentiability conditions on the diffusion coefficient and hence those assumptions match those already existing in the most recent literature. 
	The relaxation of the regularity requirement of the drift coefficient is achieved by a certain randomisation strategy that needs to be applied to \emph{both state and measure components}: the technical developments necessary to deal with this difficulty are the second contribution of this manuscript. 
			In the case of SDEs (when the coefficients do not depend on the law of the solution process), the technique of randomisation has been studied  in \cite{Kruse2019} (also \cite{Beyn2010,Kruse2012} and  and more recently in \cite{morkisz2021randomized,przybylowicz2022randomized}) to construct a  Milstein-type scheme without  assuming the first order differentiability of the drift coefficient. 	
					As the coefficients in our settings depend on the law of the solution process as well,  we require a two-fold randomisation -- one with respect to the state variable  and the other with respect to the measure variable.  
		For this, we use a uniform random variable to generate a random point in each sub-interval of the time mesh and the Euler scheme is used to obtain values of the particles' states at these random points, which are then used in the drift coefficient of the Milstein scheme, both for the state and empirical measure. 
		The precise details of this randomisation can be found in Section \ref{sec:main}. 
		Critically, the technique developed  in \cite{Kruse2019} for the analysis cannot be used directly in our settings and a novel approach is required. We observe the following.
		\begin{itemize}
			\item The interacting particle system associated with the McKean--Vlasov SDE can be treated as an $(\mathbb R^{d})^N$-dimensional SDE and thus the results of  \cite{Kruse2019} could be applied directly. However, all estimates would depend on $N$ and hence ``explode'' as $N$ tends to infinity. 
			This implies that to establish our results a new tool must be developed in order to show the independence on $N$. 
			\item We propose a new notion of bistability and consistency of the numerical scheme that is appropriate for the context of high-dimensional interacting particle systems. 
			Inspired by \cite{Kruse2019} we propose suitable stochastic Spijker norms capable of dealing with the interaction component of the particles. \textcolor{black}{Further details can be found  in Section \ref{sec:main_proof}}.
			\item A discussion on the practicalities of implementing our scheme is given in Section \ref{sec:practicalImplementation} which includes a critical view on complexity and the consequence of having common-noise. 
			
			\item In the simplest version possible of \eqref{eq:mc:intro}, three drift functions are well within the scope of our work are (linear interaction, convolution-functionals and linear interaction kernels)
\begin{align*}
	b(s, X_s, \mathcal L^1(X_s)) &=  f_1(X_s)+f_2\big ( \tilde{\mathbb E}^1(g(X_s) \big), 
		\\
	\widehat b(s, X_s, \mathcal L^1(X_s)) &=  f_1(X_s)+\int_{\mathbb{R^d}} g(X_s-y)\mathcal{L}^1(X_s)(dy),
			\\
	\tilde b(s, X_s, \mathcal L^1(X_s)) &=  f_1(X_s)+\int_{\mathbb{R^d}} \tilde g(X_s, y)\mathcal{L}^1(X_s)(dy),
\end{align*}	
for any real-valued functions $f_1,f_2,g,\tilde g$ that satisfy a standard Lipschitz condition in space (but are not differentiable), e.g., $g(x)=-|x|$ and more complex examples for $g,\tilde g$ can be found in \cite{Carrillo2010,Carmona2018-I,Carmona2018-II,Harang2020,Jin2020} \textcolor{black}{and $\tilde{\mathbb{E}}^1$ represents the conditional expectation given the common noise $W^0$}.  The 2nd drift example, $\hat b$, corresponds to the usual convolution operator fairly common in modelling with McKean--Vlasov SDE and associated interacting particle systems (with or without common noise) \cite{Carrillo2019,Harang2020,Jin2020,Adams2022}.
Lastly, we point the reader to Example 3.15 in \cite{Adams2022} that intuitively highlights why $\hat b$ is  Lipschitz in the  Wasserstein metric but is not Lions differentiable. 
		\end{itemize}

		\subsection*{Organization} The main framework of the McKean--Vlasov equation and the interacting particle system including \textcolor{black}{well-posedness} and propagation of chaos is given in Section \ref{sec:preliminaries}. The numerical scheme focusing on the approximation of the interacting particle system is found in Section \ref{sec:main} as are the main convergence results. All proofs are given in Section \ref{sec:main_proof}.

	\subsection{Notations}
	
	Both the  Euclidean norm on $ \mathbb R^d $ and the  standard  matrix norm on $ \mathbb R^{d\times m} $ are denoted by  $ |\cdot| $. 
	The notation $ \delta_x $ stands for the Dirac measure centred at $ x\in\mathbb{R}^d $. 
	We use the same notation $a^\ell$  to denote the $\ell$-th column of a matrix $ a\in\mathbb{R}^{d\times m} $ and the $ \ell $-th element of a vector $ a \in \mathbb{R}^d $. 
	$ \mathscr{B}(\chi) $ stands for the Borel  $\sigma$-algebra on a topological space $ \chi $. 
	Further,  $ \mathscr P_2(\mathbb R^d) $ denotes the space of all probability measures  on   $ (\mathbb R^d, \mathscr B(\mathbb R^d)) $ having finite second moment and equipped with the $ \mathscr{L}^2 $-Wasserstein metric given by
	\begin{align*}
		\mathcal W_2(\mu, \nu)
		:=\inf_{\pi\in\Pi(\mu,\nu)}\big[\int_{\mathbb R^d\times \mathbb R^d}|x-y|^2\pi(dx, dy)\big]^{1/2},
	\end{align*}
	for any $ \mu $, $ \nu\in\mathscr P_2(\mathbb R^d)$, where  $\Pi(\mu,\nu)$ denotes the set of couplings of  $\mu$ and $\nu$.
	Clearly, $\mathscr P_2(\mathbb R^d)$ is a Polish space under this metric.
	We use 	$\mathscr L^{ p}(\Omega)$ to denote the Banach space of all $\mathbb{R}^d$-valued random  variables $Y$ defined on a probability space $(\Omega, \mathscr{F}, \mathbb{P})$ and satisfies $\|Y\|_{\mathscr L^{ p}(\Omega)}:=\big[\mathbb{E}|Y|^p\big]^{1/p}<\infty$, {where $\mathbb{E}$ stands for expectation with respect to $\mathbb{P}$}. 
	Similarly, we use $\mathscr{L}^{p}([0,T]\times \Omega)$ to denote the Banach space of processes $Y:[0,T]\times\Omega\mapsto \mathbb R^d$ having $\|Y\|_{\mathscr{L}^{p}([0,T]\times \Omega)}=\big[\int_{0}^T\|Y(s)\|_{\mathscr{L}^{p}(\Omega)}^pds\big]^{1/p}<\infty$. 
	Also, $\mathscr{C}^{\alpha}([0,T],\,\mathscr{L}^{p}(\Omega))$ stands for the space of all $\alpha$-H\"older continuous functions $Y:[0,T]\mapsto\mathscr{L}^p(\Omega)$ with the following norm,
	\begin{align}
		\|Y\|_{\mathscr{C}^{\alpha}([0,T],\,\mathscr{L}^{p}(\Omega))}
		=\sup_{t\in[0,T]}\|Y(t)\|_{\mathscr{L}^{p}(\Omega)}+\sup_{\underset{t\neq t'}{t,t'\in[0, T]}}\frac{\|Y(t)-Y(t')\|_{\mathscr{L}^{p}(\Omega)}}{|t-t'|^\alpha}. \label{eq:norm:hold}
	\end{align}
	For a function $f:[0,T]\times\mathbb R^d\times\mathscr{P}_2(\mathbb{R}^d)\rightarrow\mathbb R$, $\partial_x f:[0,T]\times\mathbb R^d\times\mathscr{P}_2(\mathbb{R}^d)\rightarrow\mathbb R^d$ is the derivative of $f$ with respect to the space variable and $\partial_\mu f:[0,T]\times\mathbb R^d\times\mathscr{P}_2(\mathbb{R}^d)\times\mathbb R^d\rightarrow\mathbb R^d$ is the Lions derivative with respect to the measure variable. 
	{$I_A$ stands for the indicator function of a set $A$ \textcolor{black}{and $\bar{\mathbb{N}}=\mathbb{N}\cup \{0\}$. } 
The constants that appear in the paper vary from line to line, will depend on the problems data, for instance $T$, $m_0, m_1$, etc., but critically are independent of the number of particle $N$ and the schemes timestep $h$ (given below).}

	\section{McKean--Vlasov Stochastic Differential Equations and the interacting particle systems}
	\label{sec:preliminaries}
	Let $T>0$ be a   fixed constant. 	
	Consider    probability spaces   $(\Omega^1, \mathscr{F}^1, \mathbb{P}^1)$ and $(\Omega^0, \mathscr{F}^0, \mathbb P^0)$ equipped with filtrations $\mathbb F^1:=\{\mathscr{F}_t^1\}_{t\in [0, T]}$ and $\mathbb F^0:=\{\mathscr{F}_t^0\}_{t\in [0, T]}$, respectively. 
	The filtrations $\mathbb F^1$ and $\mathbb F^0$ satisfy the usual conditions, i.e., they are complete and right continuous. 
	Assume that  $W:=\{W_t  {\in \mathbb{R}^{m_1}}\}_{t\in [0,T]}$ and $W^0:=\{W_t^0 { \in \mathbb{R}^{m_0}}\}_{t\in [0,T]}$ are independent Brownian motions defined on $(\Omega^1, \mathscr{F}^1, \mathbb F^1,\mathbb{P}^1)$ and $(\Omega^0, \mathscr{F}^0,\mathbb F^0, \mathbb P^0)$,  respectively. 
	In what follows, the interacting particles are governed by i.i.d.\ copies of $W$ and $W^0$ represents the noise common to all the particles. 
	Let us define  a   product probability   space $(\tilde\Omega, \tilde{\mathscr{F}},  \tilde{\mathbb{P}})$, where $\tilde\Omega:=\Omega^1\times\Omega^0$,  $(\tilde{\mathscr{F}},\tilde{\mathbb{P}})$ is the completion of $(\mathscr{F}^1\otimes\mathscr{F}^0,\mathbb{P}^1\otimes \mathbb P^0)$ and $\tilde{\mathbb{F}}:=\{\tilde{\mathscr{F}}_t\}_{t\in[0,T]}$ is the completion and right-continuous augmentation of $\{\mathscr{F}_t^1\otimes\mathscr{F}_t^{0}\}_{t\in[0,T]}$.  
	The expectation with respect to $\tilde{\mathbb{P}}$ is denoted by $\tilde{\mathbb{E}}$.

	Let $b:[0,T]\times \mathbb{R}^d\times \mathscr{P}_2(\mathbb{R}^d)\mapsto \mathbb{R}^d$, $\sigma_1:[0,T]\times \mathbb{R}^d\times \mathscr{P}_2(\mathbb{R}^d)\mapsto \mathbb{R}^{d\times m_1}$ and $\sigma_0: [0, T]\times\mathbb{R}^d\times\mathscr{P}_2(\mathbb{R}^d)\mapsto \mathbb{R}^{d\times m_0}$  be $\mathscr B([0, T])\otimes \mathscr{B}(\mathbb R^d)\otimes \mathscr{B}(\mathscr P_2(\mathbb R^d))$-measurable functions.

	In this article, we consider the following $\mathbb R^d$-valued McKean--Vlasov stochastic differential equations (SDEs) with common noise  defined on $(\tilde\Omega, \tilde{\mathscr{F}}, \tilde{\mathbb{F}}, \tilde{\mathbb{P}})$, 
	\begin{align} \label{eq:sde}
		X_t = X_0 + &\int_0^t b(s, X_s, \mathcal L^1(X_s))ds +  \sum_{\ell=1}^{m_1}  \int_0^t \sigma_1^{\ell} (s, X_s, \mathcal L^1(X_s))dW_s^\ell\notag		
		\\
		&+\sum_{\ell=1}^{m_0}\int_0^t \sigma_0^{\ell} (s, X_s,\mathcal L^1(X_s))dW_s^{0,\ell}, 
	\end{align}
	almost surely for all $t\in [0, T]$, where $\{\mathcal L^1(X_t)\}_{t\in[0,T]}$ is the  flow of   conditional laws of $X_t$  given  $W^0$ and $X_0$. 
	{A priori, it is not certain that the flow of conditional marginals $\{\mathcal L^1(X_t)\}_{t\in [0,T]}$ is an $\mathbb{F}^0$-adapted continuous process. However, due to Lemma 2.5 in \cite{Carmona2018-II}, when the McKean--Vlasov SDE \eqref{eq:sde} has a unique $\tilde{\mathbb{F}}$-adapted continuous solution with uniformly bounded second moment, then  $\{\mathcal L^1(X_t)\}_{t\in [0,T]}$ is an $\mathbb{F}^0$-adapted continuous process.}

	We make the following assumptions.
	\begin{asump}\label{asump:initial_cond}
		$X_0\in \mathscr L^{ \bar p}(\tilde\Omega)$ for some $\bar{p}\geq 2$.
	\end{asump}
	
	\begin{asump} \label{asump:lip}
		There exists  a constant $L>0$  such that
		\begin{align*}
			|b(t, x, \mu)-b(t, x', \mu')| + \sum_{u=0}^{1}|\sigma_u(t, x,\mu)-\sigma_u(t, x',\mu')| \leq  L\{|x-x'|+\mathcal W_2(\mu, \mu')\}, 
		\end{align*}
		for all  $t \in[0, T]$, $x, x'\in \mathbb R^d$ and $\mu, \mu'\in  \mathscr P_2(\mathbb R^d)$.
	\end{asump}
	
	\begin{asump} \label{asump:holder}
		There exists a constant   $L>0$  such that   
		\begin{align*} 
			|b(t, x, \mu)-b(t', x, \mu)|
			&\leq L\{1+|x|+ \mathcal{W}_2(\mu,\delta_0) \}|t-t'|^{1/2},
			\\
			\sum_{u=0}^{1} |\sigma_u(t, x,\mu)-\sigma_u(t', x,\mu)|
			&\leq  L\{1+|x| + \mathcal{W}_2(\mu,\delta_0)\}|t-t'|,
		\end{align*}
		for all $t, t'\in[0, T]$, $x\in \mathbb R^d$   and   $\mu\in  \mathscr P_2(\mathbb R^d)$.
	\end{asump}
	
	\begin{rem} \label{rem:linear}
		From Assumptions \mbox{\normalfont {H--\ref{asump:lip}}}  and \mbox{\normalfont{H--\ref{asump:holder}}}, for all $t\in[0, T]$, $x \in \mathbb R^d $  and  $\mu\in \mathscr P_2(\mathbb R^d)$,
		\begin{align*}
			|b(t, x, \mu)| + \sum_{u=0}^1 |\sigma_u(t, x, \mu)|
			&\leq \bar L\{1+|x|+\mathcal{W}_2(\mu,\delta_0)\},
		\end{align*}
		where $\bar L=\max\big\{L,LT, L\sqrt{T}, b(0,0,\delta_0), \sigma_0(0,0,\delta_0), \sigma_1(0,0,\delta_0)\big\}$.	
	\end{rem}
	
	The proof of the following proposition can be found in \cite{Carmona2018-II,C.Kumar2021} and Appendix \ref{appendix:prop:sde_bound}.
	
	\begin{prop}[\bf Well-posedness and Moment Bounds] \label{prop:sde_bound}
		
		If Assumptions \mbox{\normalfont{H--\ref{asump:initial_cond}}} with $\bar{p}\geq 2$,  \mbox{\normalfont{H--\ref{asump:lip}}} and  \mbox{\normalfont{H--\ref{asump:holder}}} are satisfied, then the  McKean--Vlasov SDE \eqref{eq:sde} has a  unique $\tilde{\mathbb{F}}$-adapted solution $\{X_t\}_{ t\in[0, T]}$  and    
		\begin{align*}
			\big\|\sup_{t\in[0,T]}|X_t|\big\|_{\mathscr{L}^{\bar p}(\tilde\Omega)}^{\bar{p}}
			\leq C_1(1+\|X_0\|_{_{\mathscr L^{\bar p}(\tilde \Omega)}}^{\bar{p}}),
		\end{align*}	
		for $C_1:=4^{\bar p-1} \max\big\{1, 3^{\bar{p}-1} \bar{L}^{\bar{p}} \big(T^{
		\bar{p}}+2\big(\frac{\bar{p}^3 T}{2(\bar{p}-1)}\big)^{\bar{p}/2}  \big)\big\} \exp\big(12^{\bar p-1}\bar{L}^{\bar{p}}  \big(2T^{
		\bar{p}}+4\big(\frac{\bar{p}^3T}{2(\bar{p}-1)}\big)^{\bar{p}/2} \big) \big)$. 
	\end{prop}

	To introduce the interacting particle system connected with the McKean--Vlasov SDEs \eqref{eq:sde}, let us consider $N\in \mathbb N$ i.i.d.~copies of $W$ and $X_0$, denoted by $W^i$ and $X_0^i$ for $i\in\{1,\ldots, N\}$, respectively. Define the following system of equations,     
	\begin{align} \label{eq:noninteract}
		X^i_t=X^i_0 &+\int_0^tb(s,X^i_s,\mathcal L^1(X_s^i))ds+ \sum_{\ell=1}^{m_1}\int_0^t\sigma_1^{\ell} (s,X^i_s,\mathcal 
		L^1(X_s^i))dW^{i,\ell}_s\notag
		\\
		&+ \sum_{\ell=1}^{m_0}\int_0^t\sigma_0^{\ell} (s,X^i_s,\mathcal L^1(X_s^i))dW^{0,\ell}_s,
	\end{align}
	almost surely for any $t\in [0, T]$ and $i\in\{1,\ldots,N\}$. 
	Notice that due to Proposition 2.11 in \cite{Carmona2018-II},  $\mathbb{P}^0\big(\mathcal L^1(X_t^1)=\mathcal L^1(X_t^i) \mbox{ for all } t\in [0,T] \big)=1$. 
	We remark that the proof of Proposition 2.11 in  \cite{Carmona2018-II}, which establishes the result in a more restrictive setting,  uses {(only)}  the well-posedness of the system \eqref{eq:sde} and Theorem 1.33 from  \cite{Carmona2018-II} (Yamada Watanabe Theorem) and thus particles have the same law under our settings {also}.
	The system \eqref{eq:noninteract} is popularly known as the \textit{conditional non-interacting particle system}. 
	\textcolor{black}{On  approximating $\mathcal{L}^1(X_t^1)$ by the empirical measure of the states $\{X_t^i\}_{i\in\{1,\ldots, N\}}$ of $N$ particles  at time $t$, and denoted $\mu_t^{X, N}$}, 
	one obtains the following \textit{system of  interacting particles},
	\begin{align} \label{eq:interact}
		X_t^{i, N}=X_0^{i} &+\int_0^tb(s,X_s^{i, N},\mu_s^{X, N})ds+ \sum_{\ell=1}^{m_1} \int_0^t\sigma_1^{\ell} (s,X_s^{i, 
			N},\mu_s^{X, N})dW_s^{i,\ell}\notag
		\\
		&+\sum_{\ell=1}^{m_0}\int_0^t\sigma_0^{\ell}(s,X^{i,N}_s,\mu_s^{X,N})dW_s^{0,\ell},  
	\end{align}
	almost surely for any $t\in [0, T]$ and $i\in\{1,\ldots,N\}$,  where 
	\begin{align*}
		\mu_t^{X, N}(\cdot):=\frac{1}{N}\sum_{i=1}^{N}\delta_{X_t^{i, N}}(\cdot),
	\end{align*}
	is the  empirical measure of $N$ particles.

		\begin{rem}\label{rem:interact_bound}
		{
			The system \eqref{eq:interact} can be understood as an $\mathbb{R}^{d\times N}$-dimensional SDE and thus its well-posedness and moment stability up to order $\bar{p}$ follow from \cite{gyongy1980} under Assumptions \mbox{\normalfont{H--\ref{asump:initial_cond}}} (with $\bar p\geq 2$),   \mbox{\normalfont{H--\ref{asump:lip}}} and \mbox{\normalfont{H--\ref{asump:holder}}}. 
			In other words, 		the interacting particle system \eqref{eq:interact} connected with the McKean--Vlasov SDE \eqref{eq:sde} has a   unique strong solution $\{X_t^{i,N}\}_{ t\in[0, T]}$ adapted to the filtration $\{\tilde{\mathscr F}_t\}_ { t\in[0, T]}$ and one can show that     
			\begin{align*}
				\max_{i\in\{1,\ldots,N\}}\big \|\sup_{t\in[0,T]}|X_t^{i,N}|\big \|_{\mathscr{L}^{\bar p}(\tilde\Omega)}^{\bar p}  \leq C_1\Big(1+\max_{i\in\{1,\ldots,N\}}\|X_0^i\|_{_{\mathscr L^{\bar p}(\tilde \Omega)}}^{\bar p}\Big),
			\end{align*}
			where $C_1$ is the same constant as in Proposition \ref{prop:sde_bound}. Note that the RHS of the moment estimate above is \emph{critically} independent of $N$ -- this follows from \cite{C.Kumar2021} but does not from \cite{gyongy1980}. 
			}
	\end{rem}
	The following lemma gives   the time-regularity of the interacting particle system \eqref{eq:interact} and its proof is given in Appendix \ref{lemma1}.
	\begin{lem} \label{lem:dif_sde}
		Let Assumptions \mbox{\normalfont{H--\ref{asump:initial_cond}}}  with $\bar p\geq 2$, \mbox{\normalfont{H--\ref{asump:lip}}} and \mbox{\normalfont{H--\ref{asump:holder}} } hold. Then,
		\begin{align*}
			\max_{i\in\{1,\ldots,N\}} 
			\big\|X_t^{i,N}-X_{t'}^{i,N}\big\|_{\mathscr{L}^{\bar p}(\tilde\Omega)}^{\bar p} &\leq C_2|t-t'|^{ \bar p/2}\big(1+\max_{i \in \{1,\ldots,N\}}\big\|X_0^{i,N}\big\|_{\mathscr{L}^{\bar p}(\tilde\Omega)}^{\bar p}\big),
		\end{align*}
		for all $t> t'\in[0, T]$  and $N\in\mathbb N$,    
		where  
		$C_2:=(1+2C_1) 9^{\bar p-1}L^{\bar p} \big(T^{\bar p/2}+2\big( \frac{\bar{p}(\bar{p}-1)}{2}\big)^{\bar{p}/2}\big)$ and the constant $C_1$ is defined in Proposition \ref{prop:sde_bound}. 
	\end{lem}
The convergence of the interacting particle system \eqref{eq:interact} to the non-interacting particle system \eqref{eq:noninteract} is popularly known in the literature as the \textit{propagation of chaos} and is stated in the following proposition (see Theorem 2.12 in \cite{Carmona2018-II} for details or Appendix \ref{appendix:prop:poc}). 	
	For this, let us define the empirical measure of the non-interacting particle system \eqref{eq:noninteract} as
	\begin{align*}
		\mu_t^{X}(\cdot):=\frac{1}{N}\sum_{i=1}^{N}\delta_{X_t^{i}}(\cdot)
	\end{align*}
	almost surely for any $t\in [0,T]$ and $N\in \mathbb{N}$. 	
	Further, use Theorem 5.8 in \cite{Carmona2018-I} and Proposition \ref{prop:sde_bound}  to obtain the following estimate, 
	\begin{align} \label{eq:meaure:rate}
		\big\| \mathcal W_2 (\mathcal L^1(X_t^1),\mu_t^{X})\big\|_{\mathscr{L}^{2}(\tilde \Omega)}^{2}
		& \leq C_3
		\begin{cases}
			N^{-1/2}, & \mbox{ if }  d <4,
			\\
			N^{-1/2} \log_2 N, & \mbox{ if } d=4,
			\\
			N^{-2/d},  &  \mbox{ if }d>4,
		\end{cases}
	\end{align}
	for any $t\in [0,T]$ and $N\in \mathbb{N}$, where  $C_3:=C_3(d,\bar{p}, \|X_0 \|_{\mathscr{L}^{\bar{p}}(\tilde{\Omega})}^{\bar{p}})$ is a positive constant.
	
	\begin{prop}[\bf Propagation of Chaos] 
	\label{prop:poc}
		Let Assumptions \mbox{\normalfont{H--\ref{asump:initial_cond}}} with $\bar{p}>4$,   \mbox{\normalfont{H--\ref{asump:lip}}}  and  \mbox{\normalfont{H--\ref{asump:holder}}}   hold. Then, 
		\begin{align*}
			\max_{i\in\{1,\ldots,N\}}\big\|\sup_{t\in[0,T]} &|X_t^{i}-X_{t}^{i,N}|\big\|_{\mathscr{L}^{2}(\tilde\Omega)}^2
			\leq C_4
			\begin{cases}
				N^{-1/2}, & \mbox{ if }  d <4,
				\\
				N^{-1/2}\ \log_2 N, & \mbox{ if } d=4,
				\\
				N^{-2/d},  &  \mbox{ if }d>4,
			\end{cases}
		\end{align*}
		where $ C_4:=4(3T+24)TL^2 e^{6(3T+24)TL^2} C_3(d,\bar{p}, \|X_0 \|_{\mathscr{L}^{\bar{p}}(\tilde{\Omega})}^{\bar{p}})$, where the constant $C_3$ appears in \eqref{eq:meaure:rate}.  
	\end{prop} 
	The proof of this result can be found in Appendix \ref{appendix:prop:poc}.
	
	All in all, the PoC rates here are in line with the PoC results found in the numerics for McKean Vlasov SDE literature \cite{bao2021-I, bao2021-II, bao2021-III, Chen2021, Chen2023, Salkeld2019, Kumar2021, C.Kumar2021, Li2022, Reis2019, Reisinger2022}. There are improved  results obtaining rate of $1/N$ and $1/N^2$ (e.g., \cite{DelarueLackerRamanan2019,Lacker2023}) that hold under stricter conditions that do not fit the scope of the work presented in the manuscript. We have added text to this effect in the main body of the paper.

	\section{Main Result} \label{sec:main}
	The drift-randomised Milstein scheme was originally proposed and analysed for standard SDEs with non-differentiable drift coefficient in \cite{Kruse2019}. 
	In the case of McKean--Vlasov SDEs, for which the coefficients  depend on the law of the solution process as well, one needs additional randomisation of the drift coefficient with respect to the measure component and a way to deal with the implications from interacting particles.  The associated difficulties are tackled in this paper.

	\subsection{The Scheme}
	\label{subsec:scheme}
	In order to propose the randomised Milstein scheme for the interacting particle system \eqref{eq:interact}
	connected with the McKean--Vlasov SDEs \eqref{eq:noninteract}, the map $(x,\mu) \mapsto (\sigma_1(t,x,\mu),\sigma_0(t,x,\mu))$ is assumed to be continuously differentiable for every $t\in [0,T]$.
	The notion of Lions derivative (see Appendix \ref{lions} for details) is used to differentiate these functions with respect to the measure component.  
	For  $u\in\{0,1\}$,	let us define $d \times d$ matrices $\Lambda_{\sigma_u \sigma_1}$ and $\Lambda_{\sigma_u \sigma_0}$ whose $\ell$-th column are given by, 
	\begin{align}
		\Lambda_{\sigma_u\sigma_1}^{\ell}(t,s, x^i, \mu):=&\sum_{\ell_1=1}^{m_1}\int_{s}^{t}\partial_{x}\sigma_u^{\ell}(s, x^i, \mu)\sigma_1^{\ell_1}(s, x^i, \mu)dW^{i,\ell_1}_r \notag
		\\
		&+\frac{1}{N}\sum_{k= 1}^{N}\sum_{\ell_1=1}^{m_1}\int_{s}^{t}\partial_{\mu}\sigma_u^{\ell}(s, x^i, \mu,x^k)
		\sigma_1^{\ell_1}(s, x^k, \mu)dW^{k,\ell_1}_r, 
		\label{eq:lamb}
		\\
		\Lambda^{\ell}_{\sigma_u\sigma_0}(t,s, x^i,\mu):= & \sum_{\ell_1=1}^{m_0}\int_{s}^{t}\partial_{x}\sigma_u^{\ell}(s, x^i, \mu)\sigma_0^{\ell_1}(s, x^i, \mu)dW_r^{0,\ell_1} \notag
		\\
		&+\frac{1}{N}\sum_{k= 1}^{N}\sum_{\ell_1=1}^{m_0}\int_{s}^{t}\partial_{\mu}\sigma_u^{\ell}(s, x^i, \mu,x^k)
		\sigma_0^{\ell_1}(s, x^k, \mu)dW_r^{0,\ell_1}, \label{eq:lamb*}
	\end{align}
	for all $s, t\in[0, T]$, $\ell \in\{1,\ldots,m_u\}$,  $x^i \in \mathbb{R}^d$,  $i\in\{1,\ldots,N\}$, $\mu\in\mathscr{P}_2(\mathbb{R}^d)$. Also,  define the $d$-dimensional vectors $\tilde{\sigma}_1^\ell$ and $\tilde{\sigma}_0^{\ell_1}$   for any $\ell \in\{1,\ldots,m_1\}$ and $\ell_1 \in\{1,\ldots,m_0\}$ as

	\begin{equation}
		\begin{aligned}
			&\left.\begin{aligned}
				\tilde{\sigma}_1^\ell(t,s, x^i,  \mu)
				:= & \sigma_1^\ell(s, x^i, \mu)+\Lambda^\ell_{\sigma_1 \sigma_1}(t,s, x^i, \mu)
				+\Lambda^\ell_{\sigma_1\sigma_0}(t,s, x^i, \mu)\\
				\tilde{\sigma}_0^{\ell_1}(t,s, x^i,  \mu)
				:= &\sigma^{\ell_1}_0(s, x^i, \mu)+\Lambda^{\ell_1}_{\sigma_0\sigma_1}(t,s, x^i, \mu)+\Lambda^{\ell_1}_{\sigma_0\sigma_0}
				(t,s, x^i, \mu)  
			\end{aligned}\right\}\\
		\end{aligned}\label{eq:sigma_tilde}
	\end{equation}
for all    $i\in\{1,\ldots,n_h\}$.

	Now, consider a sequence  
	$\eta:=\{\eta_j\}_{j\in \mathbb N}$
	of  i.i.d.~standard uniformly distributed random variables defined  on a   probability space $(\Omega^{\eta}, \mathscr{F}^{\eta}, \mathbb P^{\eta})$, equipped with the natural filtration $\mathbb{F}^\eta:=\{\mathscr{F}_j^{\eta}\}_{j\in \mathbb N}$ of $\{\eta_j\}_{j\in \mathbb N}$. 
	Let $\mathbb{E}^\eta$ stand for the expectation with respect to $\mathbb{P}^\eta$. 
	The random variables $\{\eta_j\}_{j\in \mathbb N}$ are assumed to be independent of $W$, $W^0$, $W^i$ and $X_0^i$ for all $i\in \{1,\ldots,N\}$. 
	Now consider the general non-equidistant temporal grid $\varrho_h$ of $[0,T]$ with $n_h$ subintervals,
	\begin{align}
		\label{eq:varho-timeGrid}
		\varrho_h=\{(t_0,t_1,\ldots,t_{n_h}): 0=t_0<t_1<\cdots<t_{n_h}=T\} ,
	\end{align}
	with $h_j:= (t_j-t_{j-1})>0$  for $j\in\{1,\ldots, n_h\}$, $n_h\in \mathbb{N}$ and $h:=\max_{ j\in\{1,\ldots, n_h\}} h_j\leq \min(1,T)$.
	Now, consider a new  probability space $(\Omega, \mathscr{F},   \mathbb P)=(\tilde\Omega\times\Omega^{\eta}, \tilde{\mathscr{F}}\otimes \mathscr{F}^\eta, \tilde{\mathbb{P}} \otimes \mathbb{P}^\eta)$ equipped with a filtration $\mathbb F:=\{\mathscr{F}_j^{h}\}_{j\in\{0,\ldots,n_h\}}$; where  $\mathscr{F}_{j}^h=\tilde{\mathscr{F}}_{t_j}\otimes\mathscr{F}_j^{\eta}$.
	We will denote the expectation with respect to $\mathbb{P}$ by $\mathbb{E}$.
	
	The drift-randomised Milstein scheme for the interacting particle system \eqref{eq:interact} of the McKean--Vlasov SDE \eqref{eq:noninteract} is given by 
	\begin{align} 
		X_{j, \eta}^{i, N, h}=& X_{j-1} ^{i, N, h} +\eta_jh_j b(t_{j-1}, X_{j-1} ^{i, N, h},  \mu_{j-1}^{X, N, h}) \notag
		\\
		&+\sum_{\ell=1}^{m_1}\sigma_1^{\ell}(t_{j-1},X_{j-1} ^{i,N,h},  \mu_{j-1}^{X, N, h})\int_{t_{j-1}}^{t_{j-1}+\eta_j h_j} dW_s^{i,\ell}  \notag
		\\
		&+\sum_{\ell=1}^{m_0} \sigma_0^{\ell}(t_{j-1}, X^{i,N,h}_{j-1},\mu_{j-1}^{X,N,h})\int_{t_{j-1}}^{t_{j-1}+\eta_j h_j} dW_s^{0,\ell}  ,\label{eq:euler} 
		\\
		X_{j}^{i, N, h}= & X_{j-1}^{i, N, h}
		+ h_jb(t_{j-1}+\eta_j h_j, X_{j,\eta}^{i, N, h},\mu_{j, \eta}^{X, N, h}) \notag
		\\ 
		&+\sum_{\ell=1}^{m_1}\int_{t_{j-1}}^{t_{j}}\tilde{\sigma}_1^{\ell}(s,t_{j-1}, X_{j-1}^{i, N, h},\mu_{j-1}^{X, N, h}) dW^{i,\ell}_s  \notag
		\\
		&+\sum_{\ell=1}^{m_0}\int_{t_{j-1}}^{t_{j}}\tilde\sigma_0^{\ell}(s,t_{j-1}, X_{j-1}^{i, N, h},\mu_{j-1}^{X, N, h})dW_s^{0,\ell}, \label{eq:scheme}
	\end{align}
	almost surely for all $j\in\{1,\ldots,n_h\}$ and $i \in \{1,\ldots,N\}$ with the initial value $X_0^{i,N,h}=X_0^i$, where the empirical measures $\mu_{j-1}^{X,N,h}$ and $\mu_{j, \eta}^{X, N, h}$ are defined as 
	\begin{align}
		\mu_{j-1}^{X,N,h}(\cdot)&:=\frac{1}{N}\sum_{i=1}^{N}\delta_{X_{j-1}^{i, N,h}} (\cdot) 
		\quad\mbox{ and }\quad
		\mu_{j, \eta}^{X, N, h}(\cdot):=\frac{1}{N}\sum_{i=1}^{N}\delta_{ X_{j, \eta}^{i,  N, h}}(\cdot).
	\label{eq:em}
	\end{align}
	\begin{rem}[Comparison to prior art: classical case]
		When $\sigma_0\equiv 0$, the conditional law $\mathcal{L}^1(\cdot)$ becomes the unconditional law of the solution process. 
		If either $\mathcal{L}^1(\cdot)$ is known or $b$ and $\sigma_1$ do not depend on $\mathcal{L}^1(\cdot)$, then the McKean--Vlasov SDE \eqref{eq:sde} becomes a standard SDE. 
		In such a case, the randomised Milstein scheme \eqref{eq:scheme} considered here reduces to the one considered in \cite{Kruse2019}. 
		Indeed, the terms involving measure derivatives in \eqref{eq:lamb} and hence in \eqref{eq:sigma_tilde}, \eqref{eq:euler} and \eqref{eq:scheme} vanish in this case. 
		Additional terms that appear in \eqref{eq:euler} and \eqref{eq:scheme} are due to the dependence of  the coefficients $b$, $\sigma_1$ and $\sigma_0$ on the law of the solution process (measure variable).
	\end{rem}

			\subsubsection{Practical Implementation}
			\label{sec:practicalImplementation}
			In this section, we comment on implementation issues, specifically the sampling of the random time mesh and simulation of the L{\'e}vy area. The latter results explicitly from the presence of common noise and is not specific to our randomisation method, the former is an essential aspect of the scheme.
			
		We first note that in Equations \eqref{eq:euler} and \eqref{eq:scheme} the same uniform random variables $\eta_1$, $\ldots$, $\eta_{n_h}$ are used for each particle in the system to identify random points in each sub-interval of the temporal grid $\rho_h$. 
		This approach is similar to the {adaptive time-stepping Euler scheme} of \cite{Reisinger2022} where the same random points (arising due to  adaptive step-sizes) are used for each particle of the system. A refined version uses different time meshes for individual particles, but common, uniform
 timesteps for the definition of the empirical measure.
 In our setting, for each realisation of the common noise, a single path of random timesteps is simulated, so that the computational effort of the randomisation is comparable to the simulation of a discrete Brownian path, but negligible compared to the simulation of  iterated stochastic integrals as required for the Milstein scheme, which we discuss further now.

If the following, commutative conditions are imposed on the diffusion coefficients,
\begin{align*}
 \partial_{x}\sigma_1^{\ell}(t,  x,\mu)\sigma_1^{\ell_1}(t, x,\mu) & = \partial_{x}\sigma_1^{\ell_1}(t,  x, \mu)\sigma_1^{\ell}(t,  x, \mu),
\end{align*}
for any $\ell, \ell_1=1,\ldots,m_1$, $t\in [0,T]$, $x\in \mathbb{R}^d$ and $\mu\in\mathcal{P}_2(\mathbb{R}^d)$, then
\begin{align}
& \sum_{\ell=1}^{m_1} \sum_{\ell_1=1}^{m_1}\int_{t_{j-1}}^{t_{j}} \int_{t_{j-1}}^{s} \partial_{x}\sigma_1^{\ell}(t_{j-1},  X_{j-1}^{i, N, h},\mu_{j-1}^{X, N, h})\sigma_1^{\ell_1}(t_{j-1},  X_{j-1}^{i, N, h},\mu_{j-1}^{X, N, h})dW^{i,\ell_1}_r dW^{i,\ell}_s \notag
\\
 = & \sum_{\ell=1}^{m_1} \int_{t_{j-1}}^{t_{j}} \int_{t_{j-1}}^{s} \partial_{x}\sigma_1^{\ell}(t_{j-1},  X_{j-1}^{i, N, h},\mu_{j-1}^{X, N, h})\sigma_1^{\ell}(t_{j-1},  X_{j-1}^{i, N, h},\mu_{j-1}^{X, N, h})dW^{i,\ell}_r dW^{i,\ell}_s \notag
				\\
				& + \sum_{\ell=1}^{m_1} \sum_{\ell_1>\ell}^{m_1}\int_{t_{j-1}}^{t_{j}} \int_{t_{j-1}}^{s} \Big( \partial_{x}\sigma_1^{\ell}(t_{j-1},  X_{j-1}^{i, N, h},\mu_{j-1}^{X, N, h})\sigma_1^{\ell_1}(t_{j-1},  X_{j-1}^{i, N, h},\mu_{j-1}^{X, N, h})dW^{i,\ell_1}_r dW^{i,\ell}_s \notag
				\\
				& \qquad+  \partial_{x}\sigma_1^{\ell_1}(t_{j-1},  X_{j-1}^{i, N, h},\mu_{j-1}^{X, N, h})\sigma_1^{\ell}(t_{j-1},  X_{j-1}^{i, N, h},\mu_{j-1}^{X, N, h})dW^{i,\ell}_r dW^{i,\ell_1}_s \Big) \notag
				\\
				&= \sum_{\ell=1}^{m_1}  \partial_{x}\sigma_1^{\ell}(t_{j-1},  X_{j-1}^{i, N, h},\mu_{j-1}^{X, N, h})\sigma_1^{\ell}(t_{j-1},  X_{j-1}^{i, N, h},\mu_{j-1}^{X, N, h}) \frac{1}{2}\big((\Delta W^{i,\ell})^2-h_j\big)  \notag
				\\
				& + \sum_{\ell=1}^{m_1} \sum_{\ell_1>\ell}^{m_1}  \partial_{x}\sigma_1^{\ell}(t_{j-1},  X_{j-1}^{i, N, h},\mu_{j-1}^{X, N, h})\sigma_1^{\ell_1}(t_{j-1},  X_{j-1}^{i, N, h},\mu_{j-1}^{X, N, h}) \Delta W^{i,\ell} \Delta W^{i,\ell_1}  \notag
				\\
				=& \frac{1}{2}\sum_{\ell=1}^{m_1} \sum_{\ell_1=1}^{m_1}  \partial_{x}\sigma_1^{\ell}(t_{j-1},  X_{j-1}^{i, N, h},\mu_{j-1}^{X, N, h})\sigma_1^{\ell_1}(t_{j-1},  X_{j-1}^{i, N, h},\mu_{j-1}^{X, N, h})( \Delta W^{i,\ell} \Delta W^{i,\ell_1}-h_j I_{\{\ell=\ell_1}\}), \notag
\end{align}
		almost surely for any $j\in\{1,\ldots,n_h\}$ and $i \in \{1,\ldots,N\}$,
and thus one can  write the  third term on the right-hand side of \eqref{eq:scheme}  as
 \begin{align*}
& \sum_{\ell=1}^{m_1}  \int_{t_{j-1}}^{t_{j}}\tilde{\sigma}_1^{\ell}(s,t_{j-1}, X_{j-1}^{i, N, h},\mu_{j-1}^{X, N, h}) dW^{i,\ell}_s 
 =  \sum_{\ell=1}^{m_1} \sigma_1^{\ell}(t_{j-1}, X_{j-1}^{i, N, h},\mu_{j-1}^{X, N, h}) \Delta W^{i,\ell}  \notag
 \\
 & + \frac{1}{2}\sum_{\ell=1}^{m_1} \sum_{\ell_1=1}^{m_1}  \partial_{x}\sigma_1^{\ell}(t_{j-1},  X_{j-1}^{i, N, h},\mu_{j-1}^{X, N, h})\sigma_1^{\ell_1}(t_{j-1},  X_{j-1}^{i, N, h},\mu_{j-1}^{X, N, h})( \Delta W^{i,\ell} \Delta W^{i,\ell_1}-h_j I_{\{\ell=\ell_1}\}) \notag
 \\
 & + \sum_{\ell=1}^{m_1} \int_{t_{j-1}}^{t_{j}}  \frac{1}{N}\sum_{k= 1}^{N}\sum_{\ell_1=1}^{m_1}\int_{t_{j-1}}^{s}\partial_{\mu}\sigma_1^{\ell}(t_{j-1}, X_{j-1}^{i, N, h}, \mu_{j-1}^{X, N, h}, X_{j-1}^{k, N, h}) \notag
		\\
		&\qquad \qquad \sigma_1^{\ell_1}(t_{j-1}, X_{j-1}^{k, N, h}, \mu_{j-1}^{X, N, h})dW^{k,\ell_1}_r dW^{i,\ell}_s \notag
		\\
	& + \sum_{\ell=1}^{m_1} \sum_{\ell_1=1}^{m_0}\int_{t_{j-1}}^{t_{j}} \int_{t_{j-1}}^{s} \partial_{x}\sigma_1^{\ell}(t_{j-1},  X_{j-1}^{i, N, h},\mu_{j-1}^{X, N, h})\sigma_0^{\ell_1}(t_{j-1},  X_{j-1}^{i, N, h},\mu_{j-1}^{X, N, h})dW^{0,\ell_1}_r dW^{i,\ell}_s	
	\\
	 & + \sum_{\ell=1}^{m_1} \int_{t_{j-1}}^{t_{j}}  \frac{1}{N}\sum_{k= 1}^{N}\sum_{\ell_1=1}^{m_0}\int_{t_{j-1}}^{s}\partial_{\mu}\sigma_1^{\ell}(t_{j-1}, X_{j-1}^{i, N, h}, \mu_{j-1}^{X, N, h}, X_{j-1}^{k, N, h}) \notag
		\\
		& \qquad \qquad \sigma_0^{\ell_1}(t_{j-1}, X_{j-1}^{k, N, h}, \mu_{j-1}^{X, N, h})dW^{0,\ell_1}_r dW^{i,\ell}_s. \notag
 \end{align*}
 In the above, the last three terms require the approximation of the L\'evy area, which can be done with the help of the techniques developed in \cite{wiktorsson2001}. 
 A similar conclusion holds for the fourth term on the right-hand side of \eqref{eq:scheme}. 
 Furthermore, the terms involving Lions derivatives are of order 
 {$\mathcal{O}(1/N)$, as shown in \cite[Proof of Proposition 2.3]{bao2021-II} for the regular case,} 
 and hence can be ignored when $N$ is large.  In addition, if $\sigma_0\equiv 0$, i.e.,  the common noise term is not present, the fourth term on the right-hand side of the above equation can also be dropped and thus we have
 \begin{align*}
 & \sum_{\ell=1}^{m_1}  \int_{t_{j-1}}^{t_{j}}\tilde{\sigma}_1^{\ell}(s,t_{j-1}, X_{j-1}^{i, N, h},\mu_{j-1}^{X, N, h}) dW^{i,\ell}_s 
 =  \sum_{\ell=1}^{m_1} \sigma_1^{\ell}(t_{j-1}, X_{j-1}^{i, N, h},\mu_{j-1}^{X, N, h}) \Delta W^{i,\ell}  \notag
 \\
 & + \frac{1}{2}\sum_{\ell=1}^{m_1} \sum_{\ell_1=1}^{m_1}  \partial_{x}\sigma_1^{\ell}(t_{j-1},  X_{j-1}^{i, N, h},\mu_{j-1}^{X, N, h})\sigma_1^{\ell_1}(t_{j-1},  X_{j-1}^{i, N, h},\mu_{j-1}^{X, N, h})( \Delta W^{i,\ell} \Delta W^{i,\ell_1}-h_j I_{\{\ell=\ell_1}\}), \notag
 \end{align*}
 which leads to a fully implementable randomised Milstein scheme. 
Summing up, when $N$ is large and $\sigma_0\equiv 0$, the randomised Milstein scheme can be reduced to
 \begin{align*}
 &X_{j, \eta}^{i, N, h}= X_{j-1} ^{i, N, h} +\eta_jh_j b(t_{j-1}, X_{j-1} ^{i, N, h},  
		\mu_{j-1}^{X, N, h})+\sum_{\ell=1}^{m_1}\sigma_1^{\ell}(t_{j-1},X_{j-1} ^{i,N,h},  \mu_{j-1}^{X, N, h})(W^{i,\ell}_{t_{j-1}+\eta_j h_j}  -W^{i,\ell}_{t_{j-1}}),
		\notag
		\\
		&X_{j}^{i, N, h}=  X_{j-1}^{i, N, h}+ h_jb(t_{j-1}+\eta_j h_j, X_{j,\eta}^{i, N, h},\mu_{j, \eta}^{X, N, h})  + \sum_{\ell=1}^{m_1} \sigma_1^{\ell}(t_{j-1}, X_{j-1}^{i, N, h},\mu_{j-1}^{X, N, h}) \Delta W^{i,\ell}  \notag
 \\
 & + \frac{1}{2}\sum_{\ell=1}^{m_1} \sum_{\ell_1=1}^{m_1}  \partial_{x}\sigma_1^{\ell}(t_{j-1},  X_{j-1}^{i, N, h},\mu_{j-1}^{X, N, h})\sigma_1^{\ell_1}(t_{j-1},  X_{j-1}^{i, N, h},\mu_{j-1}^{X, N, h})( \Delta W^{i,\ell} \Delta W^{i,\ell_1}-h_j I_{\{\ell=\ell_1}\}), \notag
\end{align*}  
almost surely for any $j\in\{1,\ldots,n_h\}$ and $i \in \{1,\ldots,N\}$.

			\subsection{The Main Convergence Result and Its Assumptions}
			In order to investigate the  rate of convergence of  the randomised Milstein scheme \eqref{eq:scheme}, we make the following additional assumptions.

			\begin{asump} \label{asump:derv_lip} 	There exists  a constant $L>0$   such that
				\begin{align*}
					|\partial_x\sigma_u^{\ell}(t, x,\mu)-\partial_x\sigma_u^{\ell}(t, x',\mu')|
					&\leq  L\big\{|x-x'|+\mathcal W_2(\mu, \mu')\big\},
					\\
					|\partial_\mu\sigma_u^{\ell}(t, x,\mu,y)-\partial_\mu\sigma_u^{\ell}(t, x',\mu',y')|
					&\leq  L\big\{|x-x'|+|y-y'|+\mathcal W_2(\mu, \mu')\big\}, 
				\end{align*}
				for all $u\in\{0,1\}$, $\ell\in\{1,\ldots,m_u\}$, $t\in[0, T]$,  $x, x', y, y'\in \mathbb R^d$  and  $ \mu, \mu'\in \mathcal P_2(\mathbb R^d)$.
			\end{asump}

			\begin{asump} \label{asump:derv_lip*}
				There exists  a constant $L>0$   such that
				\begin{align*}
					|\partial_{x}\sigma_u^{\ell}(t,x,\mu)\sigma_v^{\ell_1}(t,x,\mu)-\partial_{x}\sigma_u^{\ell}(t,x',\mu')
					\sigma_v^{\ell_1}(t, x',\mu')| \leq &  L \{|x-x'|+\mathcal W_2(\mu, \mu')\},
					\\
					|\partial_{\mu}\sigma_u^{\ell}(t,x,\mu,y)\sigma_v^{\ell_1}(t,y,\mu)-\partial_{\mu}\sigma_u^{\ell}(t,x',
					\mu',y')\sigma_v^{\ell_1}(t, y',\mu')| \leq  & L \{|x-x'|+|y-y'|+\mathcal W_2(\mu, \mu')
					\},
				\end{align*}
				for all $u,v\in\{0,1\}$, $\ell\in\{1,\ldots,m_u\}$, $\ell_1\in\{1,\ldots,m_v\}$, $t\in[0, T]$, $x, x',y,y'\in \mathbb R^d$ and $\mu, \mu'\in \mathcal P_2(\mathbb R^d)$.
			\end{asump} 	
			\begin{rem} \label{rem:derv_bound}
				Due to Assumption \mbox{\normalfont{H--\ref{asump:lip}}}, we have  
				\begin{align*}
					|\partial_x\sigma_u^{\ell}(t, x,\mu)|+|\partial_\mu\sigma_u^{\ell}(t, x,\mu,y)|&\leq L,
				\end{align*}
				for  all $u\in\{0,1\}$, $\ell\in\{1,\ldots,m_u\}$,  $t\in[0, T]$, $x,y \in \mathbb R^d $ and $\mu\in \mathcal P_2(\mathbb R^d)$.
			\end{rem}

			Below, we state  the main result of this paper containing the error rate for the approximation of the scheme to the interacting particle system \eqref{eq:interact}. 
	\textcolor{black}{For notational simplicity, we use  $X_{j}^{i,N}$ to represent  $X_{t_j}^{i,N}$ for any $i\in \{1,\ldots,N\}$ and $j \in\{1,\ldots,n_h\}$}. 	
	This result is proved in Section \ref{sec:rate}.		

			\begin{theorem}\label{thm:mainresult}
				Let Assumptions \mbox{\normalfont{H--\ref{asump:initial_cond}}}  with $\bar p\geq 4$,  \mbox{\normalfont{H--\ref{asump:lip}}}  to \mbox{\normalfont{H--\ref{asump:derv_lip*}}} hold.  Then, the drift-randomised Milstein scheme \eqref{eq:scheme} converges in the strong sense to the true solution of the interacting particle system \eqref{eq:interact} with order 1. Concretely, for $q=\bar{p}/2$ and $h\leq \min(1,T)$ we have 
				\begin{align*}
					\max_{i\in\{1,\ldots,N\}}\big\|\max_{j\in\{0,\ldots,n_h\}}|X_{j}^{i,N}-X_j^{i,N,h}|\big\|_{\mathscr{L}^{q}(\Omega)}\leq C_7 C_{10}h,
				\end{align*}
				where the positive constants $C_7$ and $C_{10}$ appear in Proposition \ref{prop:residual} and Proposition \ref{prop:consistency}, respectively.  
			\end{theorem}
A direct combination of Proposition \ref{prop:poc} and Theorem \ref{thm:mainresult} delivers the control on the error for the numerical approximation of the McKean--Vlasov Equation SDE \eqref{eq:sde} (and \eqref{eq:noninteract}). 
\textcolor{black}{For convenience of notation,   $X_{j}^{i}$ is used to denote  $X_{t_j}^{i}$ for any $i\in \{1,\ldots,N\}$ and $j \in\{1,\ldots,n_h\}$}.

\begin{cor}
\label{thm:mainresultCOROLLARY}
Let assumptions of  Theorem \ref{thm:mainresult} hold with $\bar p> 4$. 
Then, the drift randomized Milstein scheme \eqref{eq:scheme} converges in the strong sense to the true solution of the McKean--Vlasov SDE \eqref{eq:sde}, 
 \begin{align*}
					\max_{i\in\{1,\ldots,N\}}\big\|\max_{j\in\{0,\ldots,n_h\}}|X_{j}^{i}-X_j^{i,N,h}|\big\|_{\mathscr{L}^{2}(\Omega)}\leq 
\begin{cases}
\sqrt{C_4} N^{-1/4}+C_7 C_{10}h, & \mbox{ if } d<4 
\\
\sqrt{C_4} N^{-1/4} \sqrt{\log_2N} + C_7 C_{10}h, & \mbox{ if } d=4 
\\
\sqrt{C_4} N^{-1/d}  + C_7 C_{10}h, & \mbox{ if } d>4 
\end{cases}					
				\end{align*}
				with $h \leq \min\{1,T\}$ and $ C_4:=2(3T+24)TL^2 e^{4(3T+24)TL^2} C_3(d,\bar{p}, \|X_0 \|_{\mathscr{L}^{\bar{p}}(\tilde{\Omega})}^{\bar{p}})$ where the constant $C_3$ appears in \eqref{eq:meaure:rate}, and $C_7$ and $C_{10}$ come from Proposition \ref{prop:residual} and  Proposition \ref{prop:consistency} respectively.  
\end{cor}

\color{black}
\begin{rem}[Application to optimal control and machine learning]
In an $n$-player stochastic differential game with an $\left(\mathbb{R}^d\right)^n$-valued state process $\boldsymbol{X}=\left(X^1, \ldots, X^n\right)$, agent $i$ chooses a control process $(\alpha^i_t)$ with values in an action space $A$
so as to minimize some target functional (see, e.g., the framework considered in \cite[(1.1)]{DelarueLackerRamanan2019}).
For a 
\emph{Markovian control} $\alpha^i_t=\alpha^i\left(t, \boldsymbol{X}_t\right)$,
the resulting dynamics can be described by a (controlled) McKean--Vlasov SDE,
\begin{align}
\label{eq1.1fromDelarueLackerRamanan2019-REMARK}
d X_t^i=\widehat b\big(X_t^i, m_{\boldsymbol{X}_t}^n, \alpha^i\left(t, \boldsymbol{X}_t\right)\big) \, d t+\sigma \, d B_t^i+\sigma_0 \, d W_t.
\end{align}
Throughout, $W$ and $B^1, \ldots, B^n$ are independent Wiener processes, and we write
\begin{align*}
m_{\boldsymbol{x}}^n=\frac{1}{n} \sum_{k=1}^n \delta_{x_k}
\end{align*}
to denote the empirical measure of a vector $\boldsymbol{x}=\left(x_1, \ldots, x_n\right)$ in $\left(\mathbb{R}^d\right)^n$. 
A similar situation arises in mean-field control, where a central agent chooses the same feedback control $\alpha\left(t, \boldsymbol{X}_t\right)$ for each agent so as to minimise their (the central agent's) objective.

In both these situations, for a non-differentiable control of (Markovian) type, $\alpha(t,X_t)$, appearing in $\widehat b$ and not appearing in $\sigma$ or $\sigma_0$ (as Assumption H-\ref{asump:derv_lip} and H-\ref{asump:derv_lip*} require differentiability), then our approximation scheme will be applicable to the simulation of the controlled (mean-field) SDE and still produce an approximation of strong order $1$ as long as one can establish sufficient regularity of $\alpha^\cdot$ such that Assumption H-\ref{asump:lip} and H-\ref{asump:holder} holds for the modified drift $b$ 
\[
[0,T]\times \mathbb R^d \times \mathcal P_2(\mathbb R^d) \ni (t,x,\mu) \mapsto b(t,x,\mu):= \widehat b\big(t,x,\mu,\alpha (t,x)\big).
\]
This exact same argument would work for controlled SDE in classical settings where the control is non-differentiable, e.g., when (adapting from \eqref{eq1.1fromDelarueLackerRamanan2019-REMARK})
\begin{align*}
d X_t=\widehat b\big(t,X_t, \alpha \left(t, {X}_t\right)\big) \, d t+ \sigma(t,X_t) \, d W_t. 
\end{align*}
In situations when $\alpha$ is a random field our theory would not apply directly (Kruse et al.'s \cite{Beyn2010,Kruse2012,Kruse2019} or \cite{morkisz2021randomized,przybylowicz2022randomized} as well). It might be possible to address space-measure mean-field controls $\alpha(t,X_t,\mu_t)$ as in \cite{PossamaiTangpi2021} but \cite{PossamaiTangpi2021} also shows that proving regularity properties for $\alpha$ in its measure component is involved. 

Lastly, our method also fits into a situation where machine learning is applied, via tools like reinforcement learning or policy iteration, to solve the optimal control problem. The requirement is a suitable choice of control/policy iteration class that would ensure Assumption H-\ref{asump:lip} and H-\ref{asump:holder} holds. 
A popular choice, especially in the moderate- to high-dimensional context, are deep neural networks, and a commonly used activation function therein is a ReLU, which makes the resulting parametric ansatz function Lipschitz but not everywhere differentiable; see \cite{reisinger2021fast} for applications of such a policy gradient method to non-smooth mean-field control, and to \cite{reisinger2022linear} for a proof that the resulting feedback control remains uniformly Lipschitz over the iterations, in a setting with controlled drift but without mean-field interaction.
\end{rem}
\color{black}

			\section{Moment Bound}
			In this section, we assume throughout that the conditions of Theorem \ref{thm:mainresult} are in force. Here, 	
			we establish moment bounds for the scheme \eqref{eq:scheme}, but before proving it (Lemma \ref{lem:mb}), we state and prove the following auxiliary result. 
			\begin{lem}\label{lem:Gamma_bound}  
				Let Assumptions  \mbox{\normalfont{ H--\ref{asump:lip}}}  and \mbox{\normalfont{H--\ref{asump:holder}}}  be satisfied. 
				For some $p \geq 2$,  if 	 $X_{j-1}^{i, N,h}\in \mathscr{L}^{p}(\Omega) $  for any $i\in\{1,\ldots, N\}$ and $j\in \{1,\ldots,n_h\}$,  then
				\[ X_{j,\eta}^{i, N, h}\in \mathscr{L}^{p}(\Omega) \mbox{ and } X_{j}^{i, N, h}\in \mathscr{L}^{p}(\Omega),\]
				for all $j\in \{1,\ldots,n_h\}$  and $i\in\{1,\ldots,N\}$.
			\end{lem}
			\begin{proof}
				Notice that for every $i \in \{1,\ldots,N\}$, $X_{j-1}^{i, N,h}$ is $\mathscr F_{j-1}^h$-measurable  and Assumption  H--\ref{asump:lip} gives the continuity of $b$ and $\sigma_u$ for $u \in \{0,1\}$, which in turn implies that  $X_{j,\eta}^{i,N,h}$ defined in  \eqref{eq:euler} is  $\mathscr F_j^h$-measurable for  any $j\in \{1,\ldots,n_h\}$. 
				Also,  continuity of  $\partial_x\sigma_u^{\ell}$ and $\partial_\mu\sigma_u^{\ell}$ for $u\in \{0,1\}$ and $\ell\in\{1,2,\ldots,m_u\}$ implies $X_{j}^{i, N, h}$ in \eqref{eq:scheme}  is  $\mathscr F_j^h$-measurable for  any $j\in \{1,\ldots,n_h\}$  and $i\in\{1,\ldots,N\}$.
				
				As $X_{j-1}^{i, N,h}\in  \mathscr{L}^{p}(\Omega) $, we have from Remark \ref{rem:linear} and Minkowski's inequality that
				\begin{align} \label{eq:best}
					\big\|b(t_{j-1},X_{j-1}^{i, N,h}, \mu_{j-1}^{X, N,h})\big\|_{\mathscr{L}^{p}(\Omega)} 
					&\leq \bar{L} \big\{1+ \big\|X_{j-1}^{i, N,h}\big\|_{\mathscr{L}^{p}(\Omega)}+\big\|\mathcal{W}_2(\mu_{j-1}^{X,N,h},\delta_0)\big\|_{\mathscr{L}^{p}(\Omega)}\big\}\notag
					\\
					&\leq  \bar{L} \big\{1+ 2\max_{i\in\{1,\ldots,N\}} \big\|X_{j-1}^{i, N,h}\big\|_{\mathscr{L}^{p}(\Omega)}\big\} <\infty,
				\end{align}
				where the last inequality is obtained by using 
				\begin{align*}
					\big\|\mathcal{W}_2(\mu_{j-1}^{X,N,h},\delta_0)\big\|_{\mathscr{L}^{p}(\Omega)} \leq \frac{1}{N} \sum_{i=1}^N \big\|X_{j-1}^{i, N,h}\big\|_{\mathscr{L}^{p}(\Omega)} \leq \max_{i\in\{1,\ldots,N\}} \big\|X_{j-1}^{i, N,h}\big\|_{\mathscr{L}^{p}(\Omega)}.
				\end{align*}
			Similarly, we get for  $u\in\{0,1\}$, 
				\begin{align} \label{eq:sest}
					\big\|\sigma_u(t_{j-1},X_{j-1}^{i, N,h}, \mu_{j-1}^{X, N,h})\big\|_{\mathscr{L}^{p}(\Omega)} \leq  \bar{L} \big\{1+ 2\max_{i\in\{1,\ldots,N\}} \big\|X_{j-1}^{i, N,h}\big\|_{\mathscr{L}^{p}(\Omega)}\big\} <\infty,
				\end{align}
				which  along with \eqref{eq:euler} further implies
				\begin{align}
					\big\|X_{j,\eta}^{i, N, h}\big\|_{\mathscr{L}^{p}(\Omega)}\leq &\big\|X_{j-1}^{i, N,h}\big\|_{\mathscr{L}^{p}(\Omega)}+h_j\big\|b(t_{j-1}, X_{j-1}^{i, N,h}, \mu_{j-1}^{X, N,h})\big\|_{\mathscr{L}^{p}(\Omega)}\notag
					\\
					&+h_j^{1/2} \Big(\frac{p(p-1)}{2}\Big)^{1/2}\sum_{u=0}^{1}\big\|\sigma_u(t_{j-1}, X_{j-1}^{i, N,h}, \mu_{j-1}^{X, N,h})\big\|_{\mathscr{L}^{p}(\Omega)} <\infty, \notag
				\end{align}
				for any $j\in \{1,\ldots,n_h\}$  and $i\in\{1,\ldots,N\}$. 
				Thus, on using Remark \ref{rem:linear}, one also obtains
				\begin{align}\label{eq:estimate_b}
					\big\|b(t_{j-1}+\eta_jh_j, X_{j,\eta}^{i, N, h},\mu_{j, \eta}^{X, N, h})\big\|_{\mathscr{L}^{p}(\Omega)}
					&\leq \bar{L} \big\{1+\big\|X_{j,\eta}^{i, N, h}\big\|_{\mathscr{L}^{p}(\Omega)}+\big\|\mathcal{W}_2(\mu_{j, \eta}^{X, N, h},\delta_0)\big\|_{\mathscr{L}^{p}(\Omega)}\big\}\notag
					\\
					&\leq \bar{L}\big\{1+2\max_{i\in\{1,\ldots,N\}}\big\|X_{j,\eta}^{i, N, h}\big\|_{\mathscr{L}^{p}(\Omega)}\big\}<\infty,
				\end{align}
				for all $j\in \{1,\ldots,n_h\}$  and $i\in\{1,\ldots,N\}$. 
				Moreover, recall \eqref{eq:lamb} and use  Remarks \ref{rem:linear} and \ref{rem:derv_bound} along with \eqref{eq:best} and \eqref{eq:sest} to obtain the following, 
				\begin{align} 
					&\sum_{\ell=1}^{m_u}\big\|\Lambda_{\sigma_u\sigma_1}^{\ell}(s,t_{j-1}, X_{j-1}^{i, N,h}, \mu_{j-1}^{X, N,h})\big\|_{\mathscr{L}^{p}(\Omega)}\notag
					\\
					&\leq h_j^{1/2}\Big(\frac{p( p-1)}{2}\Big)^{1/2}\sum_{\ell=1}^{m_u}\sum_{\ell_1=1}^{m_1}\big\|\partial_{x}\sigma_u^{\ell}(t_{j-1}, X_{j-1}^{i, N, h}, \mu_{j-1}^{X, N, h})\sigma_1^{\ell_1}(t_{j-1}, X_{j-1}^{i, N, h}, \mu_{j-1}^{X, N, h})\big\|_{\mathscr{L}^{p}(\Omega)}\notag
					\\
					& +h_j^{1/2}\Big(\frac{ p( p-1)}{2}\Big)^{1/2}\sum_{\ell=1}^{m_u} \frac{1}{N}\sum_{k=1}^N\sum_{\ell_1=1}^{m_1}\big\|\partial_{\mu}\sigma_u^{\ell}(t_{j-1}, X_{j-1}^{i, N, h}, \mu_{j-1}^{X, N, h}, X_{j-1}^{k, N, h}) \notag 
					\\
					& \qquad\qquad\sigma_1^{\ell_1}(t_{j-1}, X_{j-1}^{k, N, h}, \mu_{j-1}^{X, N, h})\big\|_{\mathscr{L}^{p}(\Omega)}\notag
					\\
					&\leq 2h_j^{1/2}\Big(\frac{ p( p-1)}{2}\Big)^{1/2}m_u m_1\bar L^2 \big\{1+2\max_{i\in\{1,\ldots,N\}}\big\|X_{j-1}^{i, N,h}\big\|_{\mathscr{L}^{p}(\Omega)}\big\}<\infty, \label{eq:nn2}
				\end{align}	
				for all $u \in \{0,1\}$, $s\in [t_{j-1}, t_j]$, $i \in \{1,\ldots, N\}$ and  $j\in \{1,\ldots,n_h\}$.  		
				Similarly, 
				\begin{align} 
					&\sum_{\ell=1}^{m_u}\big\|\Lambda_{\sigma_u\sigma_0}^{\ell}(s,t_{j-1}, X_{j-1}^{i, N,h}, \mu_{j-1}^{X, N,h})\big\|_{\mathscr{L}^{p}(\Omega)}  \notag
					\\
					&\leq 2h_j^{1/2}\Big(\frac{ p( p-1)}{2}\Big)^{1/2}m_u m_0 \bar L^2 \big\{1+2\max_{i\in\{1,\ldots,N\}}\big\|X_{j-1}^{i, N,h}\big\|_{\mathscr{L}^{p}(\Omega)}\big\}<\infty , \label{eq:nn1}
				\end{align}
				for all $u \in \{0,1\}$, $s\in [t_{j-1}, t_j]$, $i \in \{1,\ldots, N\}$ and  $j\in \{1,\ldots,n_h\}$.

				Recalling the expressions of $\tilde{\sigma}_1$ and $\tilde{\sigma}_0$ from \eqref{eq:sigma_tilde} and then applying Remark \ref{rem:linear} along with Equations \eqref{eq:nn1} and \eqref{eq:nn2}, we have 	
				\begin{align}
					&\sum_{u=0}^1\sum_{\ell=1}^{m_u}\big\|\tilde{\sigma}_u^{\ell}(s,t_{j-1},X_{j-1}^{i, N,h}, \mu_{j-1}^{X, N,h})\big\|_{\mathscr{L}^{p}(\Omega)} \notag
					\\
					&\leq \sum_{u=0}^1\sum_{\ell=1}^{m_u}\big\|\sigma^\ell_u(t_{j-1},X_{j-1}^{i, N,h}, \mu_{j-1}^{X, N,h})\big\|_{\mathscr{L}^{p}(\Omega)}+\sum_{u=0}^1\sum_{\ell=1}^{m_1}\big\|\Lambda^\ell_{\sigma_1 \sigma_u}(s,t_{j-1},X_{j-1}^{i, N,h}, \mu_{j-1}^{X, N,h})\big\|_{\mathscr{L}^{p}(\Omega)} \notag
					\\
					&\hspace{1cm}+\sum_{u=0}^1\sum_{\ell=1}^{m_0}\big\|\Lambda^\ell_{\sigma_0 \sigma_u}(s,t_{j-1},X_{j-1}^{i, N,h}, \mu_{j-1}^{X, N,h})\big\|_{\mathscr{L}^{p}(\Omega)} \notag
					\\
					&\leq (\bar L(m_0+m_1)+2h_j^{1/2}\Big(\frac{ p( p-1)}{2}\Big)^{1/2}(m_0+m_1)^2 \bar L^2) \big\{1+2\max_{i\in\{1,\ldots,N\}}\big\|X_{j-1}^{i, N,h}\big\|_{\mathscr{L}^{p}(\Omega)}\big\}<\infty, \label{eq:mmk}
				\end{align}	
				for all $i\in\{1,\ldots, N\}$, $s\in [t_{j-1}, t_j]$, $i \in \{1, \ldots, N\}$ and $j\in \{1,\ldots,n_h\}$. 
				Thus, by \eqref{eq:scheme} and Theorem 7.1 in \cite{Mao2008}, 
				\begin{align}
					\big\| X_{j}^{i, N, h}\big\|_{ \mathscr{L}^{p}(\Omega)} \leq  & \big\| X_{j-1}^{i, N, h}\big\|_{ \mathscr{L}^{p}(\Omega)}+h_j\big\|b(t_{j-1}+\eta_jh_j, X_{j, \eta}^{i, N, h}, \mu_{j, \eta}^{X, N, h} )\big\|_{\mathscr{L}^{p}(\Omega)}\notag
					\\
					&+h_j^{1/2}  \Big(\frac{p(p-1)}{2}\Big)^{1/2} \sum_{u=0}^1\sum_{\ell=1}^{m_u} \Big(\int_{t_{j-1}}^{t_j}\big\|\tilde{\sigma}_u^{\ell}(s,t_{j-1},X_{j-1}^{i, N,h}, \mu_{j-1}^{X, N,h})\big\|_{\mathscr{L}^{p}(\Omega)}^{p}ds\Big)^{1/p}, \notag
				\end{align}
				for all $i\in\{1,\ldots, N\}$ and $j\in \{1,\ldots,n_h\}$,	which on using \eqref{eq:estimate_b} and \eqref{eq:mmk} completes the proof. 
			\end{proof}
			As a consequence of the above lemma, we obtain the following corollary. 
			\begin{cor} 
				Let Assumptions \mbox{\normalfont{ H--\ref{asump:initial_cond}}} with $\bar{p}\geq 2$, \mbox{\normalfont{H--\ref{asump:lip}}}  and \mbox{\normalfont{H--\ref{asump:holder}}}  hold.  
				Then, $X_{j}^{i, N,h}\in \mathscr{L}^{\bar p}(\Omega)$ for all $i\in\{1,\ldots, N\}$ and $j\in \{1,\ldots,n_h\}$.
			\end{cor}
			\begin{proof}
				First, let us recall the randomised Milstein scheme given in \eqref{eq:scheme}. 
				For $j=1$, the result holds due to   Assumption \mbox{\normalfont{H--\ref{asump:initial_cond}}} and Lemma \ref{lem:Gamma_bound}. 
				Assume now that the result is true for $j=k$ for some $k\in\{1,\ldots,n_h\}$, i.e., $X_k^{i,N,h}\in  \mathscr{L}^{\bar p}(\Omega)$ for all  $i\in\{1,\ldots, N\}$. Then,  Lemma \ref{lem:Gamma_bound} yields  $X_{k+1}^{i,N,h}\in  \mathscr{L}^{\bar p}(\Omega)$. An inductive argument complete the proof. 
			\end{proof} 
			The following lemma gives the moment bound of the randomised Milstein scheme \eqref{eq:scheme}. 
			\begin{lem} \label{lem:mb}
				Let Assumptions \mbox{\normalfont{ H--\ref{asump:initial_cond}}} with $\bar{p}\geq 2$,   \mbox{\normalfont{ H--\ref{asump:lip}}} and  \mbox{\normalfont{H--\ref{asump:holder}}}  hold. Then, for any time grid $\varrho_h$  \eqref{eq:varho-timeGrid} with $h\leq \min(1,T)$,
				\begin{align*}
				\sup_{i \in \{1,\ldots,N\}} 	\big\|\max_{j\in\{1,\ldots,n_h\}}\big|X_{j}^{i, N, h}\big|\big\|_{\mathscr{L}^{\bar p}(\Omega)}\leq C_5 (1+\sup_{i \in \{1,\ldots,N\}}  \big\|X_{0}^{i}\big\|_{\mathscr{L}^{\bar p}(\Omega)}),
				\end{align*}
				where  \[C_5:=\max\{2, 2\bar LT+4\bar L^2T+8\bar{L}^2T \Big(\frac{p(p-1)}{2}\Big)^{1/2}+2C_0\sqrt T\} e^{4T(\bar L+2\bar L^2+4\bar{L}^2(\frac{\bar{p}(\bar{p}-1)}{2})^{1/2}+C_0^2)}\] with $C_0:= \bar{C}_{\bar p} \Big(\frac{\bar p(\bar p-1)}{2}\Big)^{1/2}(\bar L(m_0+m_1)+2h_j^{1/2} \Big(\frac{\bar p(\bar p-1)}{2}\Big)^{1/2}(m_0+m_1)^2 \bar L^2)$ and $\bar{C}_{\bar{p}}$ defined in Lemma \ref{lem:Burkholder}. 
			\end{lem}
			
			\begin{proof} 
				Recall \eqref{eq:euler} and use Minkowski's inequality,  Remark \ref{rem:linear} and $h\leq 1$ to obtain the following,    
				\begin{align}
					\big\|&X_{j,\eta}^{i,N,h}\big\|_{\mathscr{L}^{\bar p}(\Omega)}\leq  \big\|X_{j-1}^{i,N,h}\big\|_{\mathscr{L}^{\bar p}(\Omega)}+h_j \big\| \eta_j b(t_{j-1}, X_{j-1}^{i,N,h}, \mu_{j-1}^{X,N,h})\big\|_{\mathscr{L}^{\bar p}(\Omega)}\notag
					\\
					&+h_j^{1/2}  \Big(\frac{p(p-1)}{2}\Big)^{1/2} \sum_{u=0}^1 \big\|\sigma_u(t_{j-1}, X_{j-1}^{i,N,h}, \mu_{j-1}^{X,N,h})\big\|_{\mathscr{L}^{\bar p}(\Omega)}
					\notag
					\\
					\leq & \big\|X_{j-1}^{i,N,h}\big\|_{\mathscr{L}^{\bar p}(\Omega)} 
					\notag
					\\ 
					&
					+ (\bar{L}+2\bar{L} \Big(\frac{p(p-1)}{2}\Big)^{1/2} \big(1+\big\|X_{j-1}^{i,N,h}\big\|_{\mathscr{L}^{\bar p}(\Omega)}+\big\|\mathcal W_2(\mu_{j-1}^{X,N,h}, \delta_0)\big\|_{\mathscr{L}^{\bar p}(\Omega)}\big)	 \notag
					\\
					\leq & \bar{L}+2\bar{L} \Big(\frac{p(p-1)}{2}\Big)^{1/2}  + (1+2\bar{L}+ 4\bar{L} \Big(\frac{p(p-1)}{2}\Big)^{1/2} )\max_{i\in\{1,\ldots,N\}}\big\|X_{j-1}^{i,N,h}\big\|_{\mathscr{L}^{\bar p}(\Omega)} , \label{eq:xeta}
				\end{align}
				for all $i\in\{1,\ldots, N\}$, $N\in\mathbb N$ and $j \in \{1, \ldots, n_h\}$.
				Due to  \eqref{eq:scheme} and Minkowski's inequality,  for any   $\bar k \in \{1,\ldots, n_h\}$ and $i\in \{1,\ldots,N\}$, 	
				\begin{align}
					\big\|\max_{\bar n\in\{0,\ldots, \bar k\}} & \big|X_{\bar n}^{i,N,h} \big|\big\|_{\mathscr{L}^{\bar p}(\Omega)}  \leq  \big\|X_{0}^{i}\big\|_{\mathscr{L}^{\bar p}(\Omega)}  \notag
					\\
					& +  \big\| \max_{\bar n \in \{1,\ldots,\bar k\}} \big| \sum_{j=1}^{\bar n}  h_j b(t_{j-1} + \eta_j h_j, X_{j,\eta}^{i,N,h},\mu_{j,\eta}^{X,N,h})\big| \big\|_{\mathscr{L}^{\bar p}(\Omega)}  \notag
					\\
					&+ \big\| \max_{\bar n \in \{1,\ldots,\bar k\}} \big| \sum_{j=1}^{\bar n}\int_{t_{j-1}}^{t_{j}}\tilde\sigma_1(s,t_{j-1}, X_{j-1}^{i,N,h},  \mu_{j-1}^{X,N,h})dW^{i}_s|\|_{\mathscr{L}^{\bar p}(\Omega)} \notag
					\\
					&+\big\| \max_{\bar n \in \{1,\ldots,\bar k\}} \big| \sum_{j=1}^{\bar n}\int_{t_{j-1}}^{t_{j}}\tilde\sigma_0(s,t_{j-1}, X_{j-1}^{i,N,h},\mu_{j-1}^{X,N,h})dW_s^{0}  \big|\big\|_{\mathscr{L}^{\bar p}(\Omega)}, \notag
				\end{align}			 		
				which on application of Lemma \ref{lem:Burkholder}	and Remark \ref{rem:linear} yields
				\begin{align}
					\big\|\max_{\bar n\in\{0,\ldots, \bar k\}} & \big|X_{\bar n}^{i,N,h} \big|\big\|_{\mathscr{L}^{\bar p}(\Omega)}  \leq  \big\|X_{0}^{i}\big\|_{\mathscr{L}^{\bar p}(\Omega)}   +   \sum_{j=1}^{\bar k}  h_j \big\| b(t_{j-1} + \eta_j h_j, X_{j,\eta}^{i,N,h},\mu_{j,\eta}^{X,N,h}) \big\|_{\mathscr{L}^{\bar p}(\Omega)}  \notag
					\\
					&+ \bar{C}_{\bar p} \big\|  \big(\sum_{j=1}^{\bar{k}} \big| \int_{t_{j-1}}^{t_{j}}\tilde\sigma_1(s,t_{j-1}, X_{j-1}^{i,N,h},  \mu_{j-1}^{X,N,h})dW^{i}_s \big|^2\big)^{1/2} \big\|_{\mathscr{L}^{\bar p}(\Omega)} \notag
					\\
					&+  \bar{C}_{\bar p} \big\|  \big( \sum_{j=1}^{\bar k} \big|\int_{t_{j-1}}^{t_{j}}\tilde\sigma_0(s,t_{j-1}, X_{j-1}^{i,N,h},\mu_{j-1}^{X,N,h})dW_s^{0}  \big|^2 \big)^{1/2}\big\|_{\mathscr{L}^{\bar p}(\Omega)} \notag
					\\
					\leq  & \big\|X_{0}^{i}\big\|_{\mathscr{L}^{\bar p}(\Omega)}   +   \sum_{j=1}^{\bar k}  h_j \bar L \big\{1+\big\|X_{j,\eta}^{i,N,h}\big\|_{\mathscr{L}^{\bar p}(\Omega)}+\big\|\mathcal W_2(\mu_{j,\eta}^{X,N,h},\delta_0)\big\|_{\mathscr{L}^{\bar p}(\Omega)}\big\} \notag	
					\\
					&+ \bar{C}_{\bar p}  \big\|\sum_{j=1}^{\bar{k}} |\int_{t_{j-1}}^{t_{j}}\tilde\sigma_1(s,t_{j-1}, X_{j-1}^{i,N,h},  \mu_{j-1}^{X,N,h})dW^{i}_s|^2  \big\|_{\mathscr{L}^{\bar p/2}(\Omega)}^{1/2} \notag
					\\
					&+  \bar{C}_{\bar p}  \big\|\sum_{j=1}^{\bar k}  |\int_{t_{j-1}}^{t_{j}}\tilde\sigma_0(s,t_{j-1}, X_{j-1}^{i,N,h},\mu_{j-1}^{X,N,h})dW_s^{0}|^2   \big\|_{\mathscr{L}^{\bar p/2}(\Omega)}^{1/2} \notag	
					\\
					\leq & \big\|X_{0}^{i}\big\|_{\mathscr{L}^{\bar p}(\Omega)}  +   \sum_{j=1}^{\bar k}  h_j \bar{L} \big\{1+2\max_{i\in\{1,\ldots,N\}}\big\|X_{j,\eta}^{i,N,h}\big\|_{\mathscr{L}^{\bar p}(\Omega)}\big\} \notag
					\\
					&+ \bar{C}_{\bar p}  \Big(\sum_{j=1}^{\bar{k}} \big\|\int_{t_{j-1}}^{t_{j}}\tilde\sigma_1(s,t_{j-1}, X_{j-1}^{i,N,h},  \mu_{j-1}^{X,N,h})dW^{i}_s \big\|_{\mathscr{L}^{\bar p}(\Omega)}^2\Big)^{1/2} \notag
					\\
					&+  \bar{C}_{\bar p}  \Big(\sum_{j=1}^{\bar k}  \big\|\int_{t_{j-1}}^{t_{j}}\tilde\sigma_0(s,t_{j-1}, X_{j-1}^{i,N,h},\mu_{j-1}^{X,N,h})dW_s^{0}  \big\|_{\mathscr{L}^{\bar p}(\Omega)}^2\Big)^{1/2}, \notag	
				\end{align}
				for any $i \in \{1, \ldots, N\}$ and $\bar{k} \in \{1, \ldots, n_h\}$. 
				Furthermore, the application of \eqref{eq:xeta} and Theorem 7.1 in \cite{Mao2008}   give
				\begin{align}
					\big\|&\max_{\bar n\in\{0,\ldots, \bar k\}}  \big|X_{\bar n}^{i,N,h} \big|\big\|_{\mathscr{L}^{\bar p}(\Omega)} \leq   \big\|X_{0}^{i}\big\|_{\mathscr{L}^{\bar p}(\Omega)}   \notag
					\\
					& +   \sum_{j=1}^{\bar k}  h_j \bar L \big\{1+2\bar{L}+4\bar{L}\Big(\frac{\bar{p}(\bar{p}-1)}{2}\Big)^{1/2}   \notag
					\\
					& + 2(1+2\bar{L}+4\bar{L}\Big(\frac{\bar{p}(\bar{p}-1)}{2}\Big)^{1/2} ) \max_{i\in\{1,\ldots,N\}}\big\|X_{j-1}^{i,N,h}\big\|_{\mathscr{L}^{\bar p}(\Omega)}\big\}  \notag
					\\
					&+  \bar{C}_{\bar p} \Big(\frac{\bar p(\bar p-1)}{2}\Big)^{1/2}  \sum_{u=0}^{1}   \Big(\sum_{j=1}^{\bar k} h_{j}^{(\bar{p}-2)/\bar{p}} \Big( \int_{t_{j-1}}^{t_{j}} \big\|\tilde\sigma_u(s,t_{j-1}, X_{j-1}^{i,N,h},\mu_{j-1}^{X,N,h})   \big\|_{\mathscr{L}^{\bar p}(\Omega)}^{\bar{p}} ds\Big)^{2/\bar{p}}\Big)^{1/2} \label{eq:kkk}
				\end{align}
				for any $\bar{k}\in\{1,\ldots,n_h\}$ and  $i\in \{1,\ldots, N\}$. 
				
				Also, by adapting an  argument similar to the one used in \eqref{eq:mmk}, we have
				\begin{align*}
					&\sum_{u=0}^1\big\|\tilde\sigma_u(s,t_{j-1}, X_{j-1}^{i,N,h},\mu_{j-1}^{X,N,h}) \big\|_{\mathscr{L}^{\bar p}(\Omega)}
					\\
					&\leq 	(\bar L(m_0+m_1)+2h_j^{1/2} \Big(\frac{\bar p(\bar p-1)}{2}\Big)^{1/2}(m_0+m_1)^2 \bar L^2) \big\{1+2\max_{i\in\{1,\ldots,N\}}\big\|X_{j-1}^{i, N,h}\big\|_{\mathscr{L}^{\bar p}(\Omega)}\big\},
				\end{align*}
				for any $s\in [t_{j-1}, t_j]$,   $i\in \{1,\ldots, N\}$  and $j\in\{1,\ldots,n_h\}$.
				On substituting the above  in Equation \eqref{eq:kkk}, one obtains	
				\begin{align*}
				 \big\|\max_{\bar n\in\{0,\ldots, \bar k\}} & \big|X_{\bar n}^{i,N,h} \big|\big\|_{\mathscr{L}^{\bar p}(\Omega)} \leq  \big\|X_{0}^{i}\big\|_{\mathscr{L}^{\bar p}(\Omega)}  
				 \\
				 & +   \bar L(1+2\bar L+4 \bar{L} \Big(\frac{\bar{p}(\bar{p}-1)}{2}\Big)^{1/2})\sum_{j=1}^{\bar k}  h_j  \big\{1+2\max_{i\in\{1,\ldots,N\}}\big\|X_{j-1}^{i,N,h}\big\|_{\mathscr{L}^{\bar p}(\Omega)}\big\} 
					\\ 
					& +C_0\Big(\sum_{j=1}^{\bar k}  h_j  \big\{1+2\max_{i\in\{1,\ldots,N\}}\big\|X_{j-1}^{i,N,h}\big\|_{\mathscr{L}^{\bar p}(\Omega)}\big\}^2\Big)^{1/2},
				\end{align*}
where 	$C_0:= \bar{C}_{\bar p} \Big(\frac{\bar p(\bar p-1)}{2}\Big)^{1/2}(\bar L(m_0+m_1)+2h_j^{1/2} \Big(\frac{\bar p(\bar p-1)}{2}\Big)^{1/2}(m_0+m_1)^2 \bar L^2)$,		which on using Young's inequality	yields 
					\begin{align*}
				  \max_{i\in\{1,\ldots,N\}} & \big\|\max_{\bar n\in\{0,\ldots, \bar k\}}  \big|X_{\bar n}^{i,N,h} \big|\big\|_{\mathscr{L}^{\bar p}(\Omega)} \leq \max_{i\in\{1,\ldots,N\}} \big\|X_{0}^{i}\big\|_{\mathscr{L}^{\bar p}(\Omega)} 
				 \\
				 & +   \bar LT(1+2\bar L+4 \bar{L}\Big(\frac{\bar{p}(\bar{p}-1)}{2}\Big)^{1/2})+C_0\sqrt T
				\\
				&  +2\bar{L}(1+2\bar L+4\bar{L} \Big(\frac{\bar{p}(\bar{p}-1)}{2}\Big)^{1/2})\sum_{j=1}^{\bar k} h_j \max_{i\in\{1,\ldots,N\}}\big\|X_{j-1}^{i,N,h}\big\|_{\mathscr{L}^{\bar p}(\Omega)}
					\\ 
					&  +2 C_0 \max_{i\in\{1,\ldots,N\}}\big\|\max_{\bar{n}\in\{0,\ldots,\bar{k}\}} |X_{\bar{n}}^{i,N,h}|\big\|_{\mathscr{L}^{\bar p}(\Omega)}^{1/2} \Big(\sum_{j=1}^{\bar k}  h_j  \max_{i\in\{1,\ldots,N\}}\big\|X_{j-1}^{i,N,h}\big\|_{\mathscr{L}^{\bar p}(\Omega)}\Big)^{1/2}
					\\
					 \leq &  \max_{i\in\{1,\ldots,N\}} \big\|X_{0}^{i}\big\|_{\mathscr{L}^{\bar p}(\Omega)}  +   \bar LT(1+2\bar L+4\bar{L} \Big(\frac{\bar{p}(\bar{p}-1)}{2}\Big)^{1/2})+C_0\sqrt T
					\\
					& +2\bar L(1+2\bar L+4\bar{L}\Big(\frac{\bar{p}(\bar{p}-1)}{2}\Big)^{1/2})\sum_{j=1}^{\bar k} h_j \max_{i\in\{1,\ldots,N\}}\big\|X_{j-1}^{i,N,h}\big\|_{\mathscr{L}^{\bar p}(\Omega)}
					\\ 
					& +\frac{1}{2}\max_{i\in\{1,\ldots,N\}}\big\|\max_{\bar n\in\{0,\ldots, \bar k\}}\big|X_{\bar n}^{i,N,h}\big|\big\|_{\mathscr{L}^{\bar p}(\Omega)}+2C_0^2 \sum_{j=1}^{\bar k}  h_j  \max_{i\in\{1,\ldots,N\}}\big\|X_{j-1}^{i,N,h}\big\|_{\mathscr{L}^{\bar p}(\Omega)}. 
				\end{align*}
				 Thus, we have 
				\begin{align*}
					\max_{i\in\{1,\ldots,N\}} & \big\|\max_{\bar n\in\{0,\ldots, \bar k\}}   \big|X_{\bar n}^{i,N,h} \big|\big\|_{\mathscr{L}^{\bar p}(\Omega)} \leq 2\max_{i\in\{1,\ldots,N\}}\big\|X_{0}^{i}\big\|_{\mathscr{L}^{\bar p}(\Omega)} 
					\\
					&+ 2\big(\bar LT+2\bar L^2T+4\bar{L}^2T\Big(\frac{\bar{p}(\bar{p}-1}{2}\Big)^{1/2})+C_0\sqrt T\big)
					\\
					& +4\big(\bar L+2\bar L^2++4\bar{L}^2 \Big(\frac{\bar{p}(\bar{p}-1)}{2}\Big)^{1/2}+C_0^2\big)\sum_{j=1}^{\bar k} h_j \max_{i\in\{1,\ldots,N\}}\big\|X_{j-1}^{i,N,h}\big\|_{\mathscr{L}^{\bar p}(\Omega)} < \infty , \notag
				\end{align*}
				for any $\bar{k} \in \{1,\ldots,n_h\}$. 
				The application of the discrete Gr{\"o}nwall inequality (see Lemma \ref{lem:Gronwall}) yields
				\begin{align*}
					\max_{i\in\{1,\ldots,N\}} &\big\|\max_{\bar n\in\{0,\ldots,n_h\}}\big|X_{\bar n}^{i, N, h}\big|\big\|_{\mathscr{L}^{\bar p}(\Omega)} \leq  \big(2\max_{i\in\{1,\ldots,N\}}\big\|X_{0}^{i}\big\|_{\mathscr{L}^{\bar p}(\Omega)}  + 2\bar LT+4\bar L^2T
					\\
					&+8\bar{L}^2 T \Big(\frac{\bar{p}(\bar{p}-1)}{2}\Big)^{1/2}+2C_0\sqrt T\big) e^{4T(\bar L+2\bar L^2+4 \bar{L}^2 (\frac{\bar{p}(\bar{p}-1)}{2})^{1/2}+ C_0^2)}.
				\end{align*}
				This completes the proof. 
			\end{proof}

			\section{Proof of Main Result} \label{sec:main_proof}
	
			The proof mechanism we use builds from the notions of ``bistability'' and ``consistency'' introduced in \cite{Beyn2010} (and further explored in \cite{Kruse2012,Kruse2019}). We introduce a notion of bistability and consistency for the numerical scheme \eqref{eq:scheme} of the interacting particle system \eqref{eq:interact} (and associated with the McKean--Vlasov SDE \eqref{eq:sde}) by choosing suitable norms and spaces. This choice of Banach spaces and norms is designed to capture the underlying key feature of the systems being analysed, namely, that we deal with interacting particle systems -- and to the best of our knowledge, are new. 
			
			Throughout this section and in line with the statement of Theorem \ref{thm:mainresult}, we take $q=\bar{p}/2$, where $\bar{p}$ comes from Assumption H--\ref{asump:initial_cond} and assumed to satisfy $\bar{p}\geq 4$. We next introduce the required notation and definitions to prove our main result Theorem \ref{thm:mainresult}.

			\subsection{Quantities of Interest, Norms,  Banach Spaces  and Residuals}		
			For the time grid $\varrho_h$ given in  \eqref{eq:varho-timeGrid}, define the Banach spaces $(\mathscr G^{h}_{q},\|\cdot~\|_{\mathscr G^{h}_{q}})$, $(\mathscr G^{h}_{S,q},\| \cdot\|_{\mathscr G^{h}_{S,q}})$ of stochastic  grid  processes $Y^h:=\big\{Y_j^{i,N, h} \in \mathscr L^{q}(\Omega,\mathscr F_j^h,\mathbb P;\,\mathbb R^d); j\in\{0,1,\ldots,n_h\} \mbox{ and } i \in \{1,\ldots,N\}\big\}$   as
			\begin{align*}
				\|Y^h\|_{\mathscr G^{h}_{q}}:=\max_{i\in\{1,\ldots,N\}}\big\|\max_{\: j\in\{0,\ldots,n_h\}}|
				Y_j^{i,N, h}|\big\|_{\mathscr{L}^{q}(\Omega)}<\infty,
			\end{align*}
			and 
			\begin{align} \label{eq:Spnorm}
				\|Y^h\|_{\mathscr G^{h}_{S,q}}:=&\max_{i\in\{1,\ldots,N\}}\big\|Y_0^{i,N, h}\big\|_{ \mathscr{L}^{q}(\Omega)}+\max_{i\in\{1,\ldots,N\}}\big\|\max_{j\in\{1,\ldots,n_h\}}|\sum_{k=1}^{j}
				Y_k^{i,N, h}|\big\|_{\mathscr{L}^{q}(\Omega)}<\infty,
			\end{align}
			respectively. 
		Also,  define 
			\begin{align}
				\label{eq:Gamma}
				&\Gamma^h_j(Y_{j-1}^{i,N,h},\mu_{j-1}^{Y,N,h},Y_{j,\eta}^{i,N,h},\mu_{j,\eta}^{Y,N,h},\eta_j):=h_jb(t_{j-1}+\eta_jh_j, Y_{j,\eta}^{i,N,h},\mu_{j,\eta}^{Y,N,h})  \notag
				\\
				&+\sum_{\ell=1}^{m_1}\int_{t_{j-1}}^{t_{j}}\tilde\sigma_1^{\ell}(s,t_{j-1}, Y_{j-1}^{i,N,h},  \mu_{j-1}^{Y,N,h})dW^{i,\ell}_s
				+\sum_{\ell=1}^{m_0}\int_{t_{j-1}}^{t_{j}}\tilde\sigma_0^{\ell}(s,t_{j-1}, Y_{j-1}^{i,N,h},\mu_{j-1}^{Y,N,h})dW_s^{0,\ell} ,
			\end{align}
			almost surely for any $j\in\{1,\ldots,n_h\}$ and $i\in\{1,\ldots,N\}$, where $Y_{j,\eta}^{i,N,h}$ is defined using \eqref{eq:euler}, and the empirical measures $\mu_{j-1}^{Y,N,h}$ and $\mu_{j,\eta}^{Y,N,h}$ are defined using  \eqref{eq:em}. 
				
				Define $\mathcal R[Y^h]\in \mathscr G^{h}_{S, q}$ for $q\geq 2$ as the collection of the pointwise residuals $\mathcal R^{i,N}_j[Y^h]$ (associated to executing the scheme with $Y^h$) by
				\begin{align} \label{eq:residual}
					\mathcal R^{i,N}_0[Y^h]&=Y^{i,N, h}_0-X_0^{i, N, h} , \notag 
					\\
					\mathcal R^{i,N}_j[Y^h]&=Y^{i,N, h}_j-Y^{i,N, h}_{j-1}-\Gamma^h_j(Y^{i,N, h}_{j-1}, \mu_{j-1}^{Y,N, h},Y_{j,\eta}^{i,N, h},\mu_{j,\eta}^{Y,N, h},\eta_j), \, j\in\{1,\ldots,n_h\} 
				\end{align}
		almost surely	for any    $i\in\{1,\ldots,N\}$.

			\subsection*{Randomised Quadrature Rule for Stochastic Processes.} 
			
				In this small section we discuss the randomised quadrature rule for stochastic processes developed in \cite{Kruse2019}. For this, let us recall the sequence of i.i.d.\ uniform random variables  $\eta:=\{\eta_j\}_{j\in \mathbb N}$  and the temporal grid $\varrho_h$ from Section \ref{subsec:scheme}. 
		Consider a stochastic process $V:[0,T]\times\tilde\Omega\mapsto\mathbb{R}^d$ on $(\tilde{\Omega}, \tilde{\mathscr{F}}, \tilde{\mathbb{P}})$ satisfying $V\in \mathscr{L}^{p}([0,T]\times \tilde{\Omega})$ for $p \geq 2$. 
			For each $\bar n\in\{1,\ldots,n_h\}$,  
			$\displaystyle \int_0^{t_{\bar n}}V(s)ds$ is approximated by the randomised Riemann sum 
		\begin{align}
		\label{eq:Theata1}
			\Theta_{\bar n, \eta}^{ h}\big[V\big]=\displaystyle\sum_{j=1}^{\bar n}h_jV(t_{j-1}+\eta_jh_j), 
		\end{align}
			which is a random variable on $(\Omega, \mathscr{F}, \mathbb{P})$  and  an unbiased estimator of $\displaystyle \int_0^{t_{\bar n}}V(s)ds$, i.e.,
$\mathbb{E}^\eta\Theta_{\bar n, \eta}^{ h}\big[V\big]=\displaystyle \int_0^{t_{\bar n}}V(s)ds \in \mathscr{L}^{p}(\tilde{\Omega})$. 
			Moreover,  due to \textcolor{black}{Theorem 4.1} in \cite{Kruse2019}, we have
			\begin{align*}
				\big\| \max_{\bar{n}\in \{1,\ldots,n_h\}}|\Theta_{\bar n, \eta}^{ h}\big[V\big] - \int_0^{t_{\bar n}}V(s)ds | \big\|_{\mathscr{L}^{ p}(\Omega)} \leq 2 \textcolor{black}{\bar{C}_p}
				 T^{( p-2)/(2 p)}\|V\|_{\mathscr{L}^{ p}([0,T]\times\tilde\Omega)}h^{1/2}.
			\end{align*}
			Additionally, if $V\in\mathscr{C}^{\alpha}([0,T],\,\mathscr{L}^{p}(\tilde\Omega))$  for some $\alpha\in(0,1]$, then
			\begin{align} 
				\big\|\max_{\bar n\in\{1,\ldots,n_{h}\}}\big|\Theta_{\bar n, \eta}^{h}[V]-\int_0^{t_{\bar n}}V(s)ds\big|\big\|_{\mathscr{L}^{ p}(\Omega)}\leq 
				\textcolor{black}{\bar{C}_p} \sqrt T\|V\|_{\mathscr{C}^{\alpha}([0,T],\,\mathscr{L}^{ p}(\tilde\Omega))}h^{\alpha+1/2} \label{eq:quad}
			\end{align}
	\textcolor{black}{where $\bar{C}_p$ is defined in Lemma \ref{lem:Burkholder}.} 		
			\subsection{Bistability of the  Scheme}
			We now specify the notion of a scheme's bistability and show that the proposed randomised Milstein scheme \eqref{eq:scheme} of the interacting particle system \eqref{eq:interact} is bistable (see Proposition \ref{prop:residual} below).
			For this, recall the scheme from \eqref{eq:scheme} and  define $X^h:=\{ X_j^{i, N, h} : j \in \{0,\ldots,n_h\}, i\in \{1,\ldots,N\}\}$. Clearly, $X^h \in\mathscr G_{q}^h$ due to Lemma \ref{lem:mb}. 
			Also, recall the definition of residuals  $\mathcal R[Y^h]$ from \eqref{eq:residual} for some set $Y^h\in\mathscr G_{q}^h $.
			
			\begin{definition}[Stochastic Bistability] 
				\label{def:StochBistability}
				The randomised Milstein scheme \eqref{eq:scheme} associated to the interacting particle system \eqref{eq:interact} is called stochastically bistable if there exist constants $C_6, C_7>0$, independent of   $h$ and $N\in \mathbb{N}$,  such that for any arbitrary time grid $\varrho_h$ as given in \eqref{eq:varho-timeGrid}  and for any $Y^h\in\mathscr G_{q}^h$, the following holds,
				\begin{align*}
					C_6\|\mathcal R[Y^h]\|_{\mathscr G_{S,q}^{h}}\leq\|Y^h-X^{h}\|_{\mathscr G_{q}^h}
					\leq                   
					C_7\|\mathcal R[Y^h]\|_{\mathscr G_{S,q}^{h}},
				\end{align*}
for $q\geq 2$. 
			\end{definition}
		Let us first establish some useful lemmas. 	
		
			\begin{lem} \label{lem:b}
				Let Assumptions \mbox{\normalfont{H--\ref{asump:lip}}} and \mbox{\normalfont{H--\ref{asump:holder}}}   hold. 
				Then, for any $q\geq 2$ 
				and  $Y^h,Z^h\in\mathscr G_{q}^h$, 
				\begin{align*}
					\big\|\max_{\bar n\in\{1,\ldots,\bar{k}\}}&\big|\sum_{j=1}^{\bar n}h_j\big[b(t_{j-1}+\eta_jh_j, Y_{j,\eta}^{i,N,h},\mu_{j,\eta}^{Y,N,h})-b(t_{j-1}+\eta_jh_j, Z_{j,\eta}^{i,N,h},\mu_{j,\eta}^{Z,N,h})\big]\big|\big\|_{ \mathscr{L}^{q}(\Omega)}
					\\
					&\leq C_8 \sum_{j=1}^{\bar k}h_j \max_{i\in\{1,\ldots,N\}}\big\|\max_{\bar{n} \in\{0,\ldots,j-1\}}\big|Y^{i,N,h}_{\bar{n}}-Z^{i, N, h}_{\bar{n}} \big|\big\|_{\mathscr{L}^{q}(\Omega)},
				\end{align*}
				for any $\bar{k}\in\{1,\ldots,n_h\}$ and $i \in \{1,\ldots,N\}$ where  $C_8:=2L+4L^2T  + 8L^2\sqrt{ Tq(q-1)/2}$.
			\end{lem}
			
			\begin{proof} 
				Notice that all the terms in the statement of this lemma are well-defined  	due to Remark \ref{rem:linear}.
				By using \eqref{eq:euler} with $Y^h,Z^h\in\mathscr G_{q}^h$  along with 	Minkowski's inequality and Assumption H--\ref{asump:lip}, we have 
				\begin{align}
					&\big\|Y_{j,\eta}^{i,N,h}-Z_{j,\eta}^{i,N,h}\big\|_{\mathscr{L}^{q}(\Omega)}\leq \big\|Y_{j-1}^{i,N,h}-Z_{j-1}^{i,N,h}\big\|_{\mathscr{L}^{q}(\Omega)}\notag
					\\
					&+h_j  \big\|\eta_j (b(t_{j-1}, Y_{j-1}^{i,N,h}, \mu_{j-1}^{Y,N,h})-b(t_{j-1}, Z_{j-1}^{i,N,h}, \mu_{j-1}^{Z, N, h}))\big\|_{\mathscr{L}^{q}(\Omega)}\notag
					\\ \nonumber
					&+\sqrt{h_j} \Big(\frac{q(q-1)}{2}\Big)^{1/2} 
					\sum_{u=0}^1\big\|\sqrt{\eta_j}(\sigma_u(t_{j-1}, Y_{j-1}^{i,N,h}, \mu_{j-1}^{Y,N,h})
					\\
					&\hspace{6cm}
					-\sigma_u(t_{j-1}, Z_{j-1}^{i,N,h}, \mu_{j-1}^{Z, N, h}))\big\|_{\mathscr{L}^{q}(\Omega)}
					\notag
					\\
					\leq & \big\|Y_{j-1}^{i,N,h}-Z_{j-1}^{i,N,h}\big\|_{\mathscr{L}^{q}(\Omega)} + \Big(LT  + 2L\sqrt T\Big(\frac{q(q-1)}{2}\Big)^{1/2}\Big) 
					\notag
					\\
					& \qquad \times \big\{\big\|Y_{j-1}^{i,N,h}-Z_{j-1}^{i,N,h}\big\|_{\mathscr{L}^{q}(\Omega)}+\big\|\mathcal W_2(\mu_{j-1}^{Y,N,h}, \mu_{j-1}^{Z, N, h})\big\|_{\mathscr{L}^{q}(\Omega)}\big\}\notag	
					\\
					&\leq \Big(1 +2L T + 4L \sqrt{T}\Big(\frac{q(q-1)}{2}\Big)^{1/2}\Big) \max_{i\in\{1,\ldots,N\}}\big\|\max_{\bar n\in\{0,\ldots,j-1\}}\big|Y_{\bar n}^{i,N,h}-Z_{\bar n}^{i, N, h}\big|\big\|_{\mathscr{L}^{q}(\Omega)}, \label{eq:htf}	
				\end{align}
				where the last inequality is obtained by using
				\begin{align*}
					\mathcal W_2(\mu_{j-1}^{Y,N,h}, \mu_{j-1}^{Z, N, h}) \leq \Big(\frac{1}{N} \sum_{i=1}^{N} \big|Y_{j-1}^{i,N,h}-Z_{j-1}^{i,N,h}\big|^2 \Big)^{1/2},
				\end{align*}		
				for all $i\in\{1,\ldots, N\}$, $N\in\mathbb N$ and $j\in\{1,\ldots,n_h\}$. 
				
				Now, the application Assumption H--\ref{asump:lip}  	 yields
				\begin{align*}
					\big\|\max_{\bar n\in\{1,\ldots,\bar{k}\}}\big|\sum_{j=1}^{\bar n}&h_j\big[b(t_{j-1}+\eta_jh_j, Y_{j,\eta}^{i,N,h},\mu_{j,\eta}^{Y,N,h})-b(t_{j-1}+\eta_jh_j, Z_{j,\eta}^{i,N,h},\mu_{j,\eta}^{Z,N,h})\big]\big|\big\|_{\mathscr{L}^{q}(\Omega)}\notag
					\\
					&\leq  L \sum_{j=1}^{\bar k}h_j\big\{\big\|Y_{j,\eta}^{i,N,h}-Z_{j,\eta}^{i,N,h}\big\|_{\mathscr{L}^{q}(\Omega)}+\big\|\mathcal W_2(\mu_{j,\eta}^{Y,N,h},\mu_{j,\eta}^{Z,N,h})\big\|_{\mathscr{L}^{q}(\Omega)}\big\}\notag
					\\
					&\leq 2 L \sum_{j=1}^{\bar k}h_j\max_{i\in\{1,\ldots,N\}}\big\|Y_{j,\eta}^{i,N,h}-Z_{j,\eta}^{i,N,h}\big\|_{\mathscr{L}^{q}(\Omega)},
				\end{align*}	
				and then one uses \eqref{eq:htf} to complete the proof. 
			\end{proof}
			\begin{lem} \label{lem:sigma}
			Let Assumptions \mbox{\normalfont{H--\ref{asump:lip}}}, \mbox{\normalfont{H--\ref{asump:holder}}}, \mbox{\normalfont{H--\ref{asump:derv_lip*}}}   hold. Then, for any $q\geq 2$
			and $Y^h,Z^h\in\mathscr G_{q}^h$,  
			\begin{align*}
				&\big\|\max_{\bar n\in\{1,\ldots,\bar{k}\}}\big|\sum_{j=1}^{\bar n}\sum_{\ell=1}^{m_1}\int_{t_{j-1}}^{t_j}\big[\tilde\sigma_1^{\ell}(s,t_{j-1}, Y_{j-1}^{i,N,h},  \mu_{j-1}^{Y,N,h})
				\\
				&\qquad \qquad\qquad \qquad
				-\tilde\sigma_1^{\ell}(s,t_{j-1}, Z_{j-1}^{i,N,h},  \mu_{j-1}^{Z, N, h})\big]dW_s^{i,\ell}\big|\big\|_{\mathscr{L}^{q}(\Omega)}
				\\
				+& \big\|\max_{\bar n\in\{1,\ldots,\bar{k}\}}\big|\sum_{j=1}^{\bar n}\sum_{\ell=1}^{m_0}\int_{t_{j-1}}^{t_j}\big[\tilde\sigma_0^{\ell}(s,t_{j-1}, Y_{j-1}^{i,N,h},  \mu_{j-1}^{Y,N,h})
								\\
				&\qquad \qquad\qquad \qquad
				-\tilde\sigma_0^{\ell}(s,t_{j-1}, Z_{j-1}^{i,N,h},  \mu_{j-1}^{Z, N, h})\big]dW_s^{0,\ell}\big|\big\|_{\mathscr{L}^{q}(\Omega)}
				\\
				& \qquad \leq C_9 \Big(\sum_{j=1}^{\bar k}h_j\max_{i\in\{1,\ldots,N\}}\big\|\max_{\bar n\in\{0,\ldots,j-1\}}\big|Y^{i,N,h}_{\bar n}-Z^{i, N, h}_{\bar n}\big|\big\|_{\mathscr{L}^{q}(\Omega)}^2\Big)^{1/2},
			\end{align*}
			for all $i\in\{1,\ldots,N\}$ and $\bar{k}\in\{1,\ldots,n_h\}$, where $C_9:= (m_0 + m_1) \bar{C}_{q} \sqrt {q(q-1)/2} \big(2 L+5(m_0+m_1)L\sqrt {Tq(q-1)/2}\big)$ and $\bar{C}_q$ is given in Lemma \ref{lem:Burkholder}.    
		\end{lem}
		\begin{proof}
		Notice that 
		\begin{align*}
	\Big\{ \sum_{j=1}^{\bar n} \sum_{\ell=1}^{m_1}\int_{t_{j-1}}^{t_j}\big[\tilde\sigma_1^{\ell}(s,t_{j-1}, Y_{j-1}^{i,N,h},  \mu_{j-1}^{Y,N,h})-\tilde\sigma_1^{\ell}(s,t_{j-1}, Z_{j-1}^{i,N,h},  \mu_{j-1}^{Z, N, h})\big]dW_s^{i,\ell} \Big\}_{\bar{n}},
		\end{align*}
		with ${\bar{n} \in \{0,\ldots,n_h\}}$
		is  an $\{\mathscr F_{\bar{n}}^h\}_{\bar{n}\in\{0,\ldots,n_h\}}$-adapted standard martingale. 
Thus, on using 		 Lemma \ref{lem:Burkholder} and Minkowski's inequality, one obtains, 		for all $i\in\{1,\ldots, N\}$ and $\bar{k}\in\{1,\ldots,n_h\}$,
			\begin{align}
			\nonumber
			\label{eq:tilde_sigma}
	\big\|&\max_{\bar n\in\{1,\ldots,\bar{k}\}}\big|\sum_{j=1}^{\bar n}\sum_{\ell=1}^{m_1}\int_{t_{j-1}}^{t_j}\big[\tilde\sigma_1^{\ell}(s,t_{j-1}, Y_{j-1}^{i,N,h},  \mu_{j-1}^{Y,N,h})
					\\ \nonumber
				&\hspace{5cm}
				-\tilde\sigma_1^{\ell}(s,t_{j-1}, Z_{j-1}^{i,N,h},  \mu_{j-1}^{Z, N, h})\big]dW_s^{i,\ell}\big|\big\|_{\mathscr{L}^{q}(\Omega)} \notag
				\\ \nonumber
				&\leq  \bar{C}_{q}\big\|\sum_{j=1}^{\bar{k}}\big|\sum_{\ell=1}^{m_1}\int_{t_{j-1}}^{t_j}\big[\tilde\sigma_1^{\ell}(s,t_{j-1}, Y_{j-1}^{i,N,h},  \mu_{j-1}^{Y,N,h})
								\\ \nonumber
				&\hspace{5cm}
				-\tilde\sigma_1^{\ell}(s,t_{j-1}, Z_{j-1}^{i,N,h},  \mu_{j-1}^{Z, N, h})\big]dW_s^{i,\ell}\big|^2\big\|_{\mathscr{L}^{q/2}(\Omega)}^{1/2}\notag
				\\  \nonumber
				&\leq \bar{C}_{q}\Big(\sum_{j=1}^{\bar{k}}\big\|\sum_{\ell=1}^{m_1}\int_{t_{j-1}}^{t_j}\tilde\sigma_1^{\ell}(s,t_{j-1}, Y_{j-1}^{i,N,h},  \mu_{j-1}^{Y,N,h})
								\\ \nonumber
				&\hspace{5cm}-\tilde\sigma_1^{\ell}(s,t_{j-1}, Z_{j-1}^{i,N,h},  \mu_{j-1}^{Z, N, h})dW_s^{i,\ell}\big\|^2_{\mathscr{L}^{q}(\Omega)}\Big)^{1/2}\notag
				\\ \nonumber
				&\leq \sqrt m_1 \bar{C}_{q} \Big(\frac{q(q-1)}{2}\Big)^{1/2} \Big(\sum_{j=1}^{\bar{k}}h_j^{(q-2)/q}\sum_{\ell=1}^{m_1}\big[\int_{t_{j-1}}^{t_j}\big\|\tilde\sigma_1^{\ell}(s,t_{j-1}, Y_{j-1}^{i,N,h},  \mu_{j-1}^{Y,N,h})\notag
				\\
				&\hspace{5cm}-\tilde\sigma_1^{\ell}(s,t_{j-1}, Z_{j-1}^{i,N,h},  \mu_{j-1}^{Z, N, h})\big\|^{q}_{\mathscr{L}^{q}(\Omega)}ds\big]^{2/q}\Big)^{1/2},
			\end{align}
		where the last inequality is obtained due to  Theorem 7.1 in \cite{Mao2008}.
			
			Now, recall \eqref{eq:lamb}, \eqref{eq:lamb*} and \eqref{eq:sigma_tilde} and use Assumption H--\ref{asump:lip} to get the following, 
			\begin{align*}
				\big\|\tilde\sigma_1^{\ell}(s&,t_{j-1}, Y_{j-1}^{i,N,h},  \mu_{j-1}^{Y,N,h})-\tilde\sigma_1^{\ell}(s,t_{j-1}, Z_{j-1}^{i,N,h},  \mu_{j-1}^{Z, N, h})\|_{\mathscr{L}^{q}(\Omega)}
				\\
				\leq &  \|\sigma_1^{\ell}(t_{j-1}, Y_{j-1}^{i,N,h},  \mu_{j-1}^{Y,N,h})-\sigma_1^{\ell}(t_{j-1}, Z_{j-1}^{i,N,h},  \mu_{j-1}^{Z, N, h})\|_{\mathscr{L}^{q}(\Omega)}
				\\
				&+\sum_{u=0}^{1}\big\|\Lambda_{\sigma_1\sigma_u}^\ell(s,t_{j-1}, Y_{j-1}^{i,N,h},  \mu_{j-1}^{Y,N,h})-\Lambda_{\sigma_1\sigma_u}^\ell(s,t_{j-1}, Z_{j-1}^{i,N,h}, \mu_{j-1}^{Z, N, h})\big\|_{\mathscr{L}^{q}(\Omega)}
				\\
				\leq & L \big\{\big\|Y_{j-1}^{i,N,h}- Z_{j-1}^{i,N,h}\big\|_{\mathscr{L}^{q}(\Omega)}+\big\|\mathcal W_2(\mu_{j-1}^{Y,N,h}, \mu_{j-1}^{Z, N, h})\big\|_{\mathscr{L}^{q}(\Omega)}\big\}
				\\
				&+h_j^{1/2}\Big(\frac{q(q-1)}{2}\Big)^{1/2} \sum_{u=0}^{1}\sum_{\ell_1=1}^{m_u}\big\|\partial_{x}\sigma_1^{\ell}(t_{j-1}, Y_{j-1}^{i,N,h}, \mu_{j-1}^{Y,N,h})\sigma_u^{\ell_1}(t_{j-1},Y_{j-1}^{i,N,h}, \mu_{j-1}^{Y,N,h})\notag
				\\
				&-\partial_{x}\sigma_1^{\ell}(t_{j-1}, Z_{j-1}^{i,N,h}, \mu_{j-1}^{Z, N, h})\sigma_u^{\ell_1}
				(t_{j-1}, Z_{j-1}^{i,N,h}, \mu_{j-1}^{Z, N, h})\big\|_{\mathscr{L}^{q}(\Omega)}\notag
				\\
				&+h_j^{1/2}\Big(\frac{q(q-1)}{2}\Big)^{1/2} \frac{1}{N}\sum_{u=0}^{1}\sum_{\ell_1=1}^{m_u}\sum_{k=1}^{N}\big\|\partial_{\mu}\sigma_1^{\ell}(t_{j-1}, Y_{j-1}^{i,N,h}, \mu_{j-1}^{Y,N,h}, Y_{j-1}^{k, N,h})
				\notag
				\\
				&\times \sigma_u^{\ell_1}(t_{j-1},Y_{j-1}^{k, N,h}, \mu_{j-1}^{Y,N,h})-\partial_{\mu}\sigma_1^{\ell}(t_{j-1}, Z_{j-1}^{i,N,h}, \mu_{j-1}^{Z, N, h}, Z_{j-1}^{k, N, 
					h})
					\\
					& \times \sigma_u^{\ell_1}(t_{j-1}, Z_{j-1}^{k, N, h}, \mu_{j-1}^{Z, N, h})\big\|_{\mathscr{L}^{q}(\Omega)},\notag
			\end{align*}
which due to Assumption H--\ref{asump:derv_lip*} yields
\begin{align*}
\big\|\tilde\sigma_1^{\ell}(s&,t_{j-1}, Y_{j-1}^{i,N,h},  \mu_{j-1}^{Y,N,h})-\tilde\sigma_1^{\ell}(s,t_{j-1}, Z_{j-1}^{i,N,h},  \mu_{j-1}^{Z, N, h})\|_{\mathscr{L}^{q}(\Omega)}
				\\
				 \leq & \Big(L+2(m_0+m_1)L\Big(\frac{Tq(q-1)}{2}\Big)^{1/2} \Big) 
				\\
				& \times \big\{\big\|Y_{j-1}^{i,N,h}- Z_{j-1}^{i,N,h}\big\|_{\mathscr{L}^{q}(\Omega)}+\big\|W_2(\mu_{j-1}^{Y,N,h}, \mu_{j-1}^{Z, N, h})\big\|_{\mathscr{L}^{q}(\Omega)}\big\}
				\\
				&+(m_0+m_1)L\Big(\frac{Tq(q-1)}{2}\Big)^{1/2} \frac{1}{N}\sum_{k=1}^N \big\|Y_{j-1}^{k,N,h}- Z_{j-1}^{k,N,h} \big\|_{\mathscr{L}^{q}(\Omega)} \notag
				\\
				\leq & \Big(2L+5(m_0+m_1)L\Big(\frac{Tq(q-1)}{2}\Big)^{1/2} \Big)  \max_{i\in \{1,\ldots,N\}} \big\|\max_{\bar n\in\{0,\ldots,j-1\}} \big|Y_{\bar n}^{i,N,h}-Z_{\bar n}^{i, N, h}\big|\big\|_{\mathscr{L}^{q}(\Omega)} ,
\end{align*}
for any $s\in [t_{j-1}, t_j]$, $j\in \{1,\ldots,n_h\}$ and $i\in\{1,\ldots,N\}$. 
	The  proof of the first part of the lemma is completed by substituting the above in \eqref{eq:tilde_sigma}.  
	A bound for the terms involving $\tilde{\sigma}_0$ follows by similar arguments. 
		\end{proof}
			
			As a consequence of Lemmas \ref{lem:b} and  \ref{lem:sigma}, we obtain the following corollary. 
			\begin{cor} \label{cor:Gam1}
			If Assumptions \mbox{\normalfont{H--\ref{asump:lip}}}, \mbox{\normalfont{H--\ref{asump:holder}}} and \mbox{\normalfont{H--\ref{asump:derv_lip*}}} are satisfied. Then, for any $q\geq 2$
			and $Y^h,Z^h\in\mathscr G_{q}^h$,  
			\begin{align*}
				&\big\|\max_{\bar n\in\{1,\ldots,\bar{k}\}}\big|\sum_{j=1}^{\bar n}\big[\Gamma^h_j(Y_{j-1}^{i,N,h},\mu_{j-1}^{Y,N,h},Y_{j,\eta}^{i,N,h},\mu_{j,\eta}^{Y,N,h},\eta_j)
				\\
				&\hspace{5cm}-\Gamma^h_j(Z_{j-1}^{i,N,h},\mu_{j-1}^{Z,N,h},Z_{j,\eta}^{i,N,h},\mu_{j,\eta}^{Z,N,h},\eta_j)\big]\big|\big\|_{\mathscr{L}^{q}(\Omega)}
				\\
				\leq &   C_8\sum_{j=1}^{\bar{k}}h_j\max_{i\in\{1,\ldots,N\}} \big\|\max_{\bar n\in\{0,\ldots,j-1\}}\big|Y^{i,N,h}_{\bar n}-Z^{i, N, h}_{\bar n}\big|\big\|_{\mathscr{L}^{q}(\Omega)}
				\\
				&+C_9\Big(\sum_{j=1}^{\bar{k}}h_j\max_{i\in\{1,\ldots,N\}} \big\|\max_{\bar n\in\{0,\ldots,j-1\}}\big|Y^{i,N,h}_{\bar n}-Z^{i, N, h}_{\bar n}\big|\big\|_{\mathscr{L}^{q}(\Omega)}^2\Big)^{1/2} 
				\\
				\leq&  (C_8 T+C_9\sqrt T) \big\|Y^h-Z^h\big\|_{\mathscr G_{q}^h}, 
			\end{align*}
			for all  $\bar{k}\in\{1,\ldots,n_h\}$ and $i\in\{1,\ldots,N\}$, where the  constants  $C_8$ and $C_9$ appear in Lemmas \ref{lem:b} and  \ref{lem:sigma}, respectively.
					\end{cor}
		\begin{proof}
			Recall $\Gamma^h_j$ from \eqref{eq:Gamma} and write
			\begin{align*}
				\Gamma^h_j(Y_{j-1}^{i,N,h},& \mu_{j-1}^{Y,N,h},Y_{j,\eta}^{i,N,h},\mu_{j,\eta}^{Y,N,h},\eta_j)-\Gamma^h_j(Z_{j-1}^{i,N,h},\mu_{j-1}^{Z,N,h},Z_{j,\eta}^{i,N,h},\mu_{j,\eta}^{Z,N,h},\eta_j)
				\\
				&= h_j\big[b(t_{j-1}+\eta_jh_j, Y_{j,\eta}^{i,N,h},\mu_{j,\eta}^{Y,N,h})-b(t_{j-1}+\eta_jh_j, Z_{j,\eta}^{i,N,h},\mu_{j,\eta}^{Z,N,h})\big]
				\\
				&\quad+\sum_{\ell=1}^{m_1}\int_{t_{j-1}}^{t_{j}}\big[\tilde\sigma_1^{\ell}(s, t_{j-1}, Y_{j-1}^{i,N,h},  \mu_{j-1}^{Y,N,h})-\tilde\sigma_1^{\ell}(s, t_{j-1}, Z_{j-1}^{i,N,h},  \mu_{j-1}^{Z, N, h})\big]dW^{i,\ell}_s
				\\
				&\quad+\sum_{\ell=1}^{m_0}\int_{t_{j-1}}^{t_{j}}\big[\tilde\sigma_0^{\ell}(s,t_{j-1}, Y_{j-1}^{i,N,h},\mu_{j-1}^{Y,N,h})-\tilde\sigma_0^{\ell}(s,t_{j-1}, Z_{j-1}^{i,N,h},\mu_{j-1}^{Z,N,h})\big]dW_s^{0,\ell}.
			\end{align*}
			Then, on using Lemmas \ref{lem:b} and  \ref{lem:sigma}, we get the required result.
		\end{proof}
		We now prove the bistability of the scheme \eqref{eq:scheme} in the following proposition. 
		\begin{prop} \label{prop:residual}
		Let Assumptions \mbox{\normalfont{H--\ref{asump:initial_cond}}} with $\bar{p}\geq 4$ and set $q=\bar{p}/2\geq 2$, \mbox{\normalfont{H--\ref{asump:lip}}},  \mbox{\normalfont{H--\ref{asump:holder}}} and \mbox{\normalfont{H--\ref{asump:derv_lip*}}}  hold. Then, 
			for any $Y^h,X^{h}\in\mathscr G_{q}^h$, we have 
			\begin{align*}
				C_6\|\mathcal R[Y^h]\|_{\mathscr G_{S,q}^{h}}\leq\|Y^h-X^{h}\|_{\mathscr G_{q}^h}\leq  C_7\|\mathcal R[Y^h]\|_{\mathscr G_{S,q}^{h}},
			\end{align*}
			in other words, the randomised Milstein scheme given in \eqref{eq:scheme}  is  stochastically bistable in the sense of Definition  \ref{def:StochBistability} with  $C_6:=\frac{1}{3+ C_8 T+C_9\sqrt T}$ and $C_7:=2e^{(2C_8+C_9^2)T}$, where  the constants $C_8$ and $C_9$ appear in Lemmas \ref{lem:b} and \ref{lem:sigma}, respectively. 		
		\end{prop}
		\begin{proof}
			Recall Equations \eqref{eq:scheme}, \eqref{eq:Gamma} and  \eqref{eq:residual}  to write
			\begin{align}
				X_{\bar n}^{i, N,h}& -X_0^{i, N, h}-\sum_{j=1}^{\bar n}\Gamma ^h_j(X_{j-1}^{i, N, h},\mu_{j-1}^{X, N, h},X_{j,\eta}^{i,N,h},\mu_{j, \eta}^{X, N, h},\eta_j)=0, \notag
			\\
				 \sum_{j=1}^{\bar n} \mathcal R_j^{i,N}[Y^h]&=Y_{\bar n}^{i,N,h}-Y_0^{i,N,h}-\sum_{j=1}^{\bar n}\Gamma ^h_j(Y_{j-1}^{i,N,h},\mu_{j-1}^{Y,N,h},Y_{j,\eta}^{i,N,h},\mu_{j,\eta}^{Y,N,h},\eta_j) \notag
				\\
				&=\big(Y_{\bar n}^{i,N,h}-X_{\bar n}^{i, N, h}\big)-\big(Y_0^{i,N,h}-X_0^{i,N,h}\big)  \notag
				\\
				-\sum_{j=1}^{\bar n}\big(\Gamma ^h_j(Y_{j-1}^{i,N,h},\mu_{j-1}^{Y,N,h},&Y_{j,\eta}^{i,N,h},\mu_{j,\eta}^{Y,N,h}, \eta_j)-\Gamma ^h_j(X_{j-1}^{i,N,h},\mu_{j-1}^{X, N, h},X_{j,\eta}^{i,N,h},\mu_{j, \eta}^{X, N, h},\eta_j)\big), \label{eq:rearrange}
			\end{align}
			for any $\bar{n}\in\{1,\ldots,n_h\}$,  $i\in\{1,\ldots,N\}$ and then use the Spijker norm \eqref{eq:Spnorm} to get the following, 
			\begin{align}
				\big\|\mathcal R[Y&^h]\big\|_{\mathscr G_{S,q}^{h}}  =  \max_{i\in\{1\ldots,N\}}\big\|\big|\mathcal R_0^{i,N}[Y^h]\big|\big\|_{\mathscr{L}^{q}(\Omega)}+\max_{i\in\{1,\ldots,N\}}\big\|\max_{\bar n\in\{1,\ldots,n_h\}}\big|\sum_{j=1}^{\bar n}\mathcal R_j^{i,N}[Y^h]\big|\big\|_{\mathscr{L}^{q}(\Omega)}\notag
				\\
				\leq & 2\max_{i\in\{1,\ldots,N\}}\big\|Y_0^{i,N,h}-X_0^{i,N,h}\big\|_{\mathscr{L}^{q}(\Omega)}+\max_{i\in\{1,\ldots,N\}}\big\|\max_{\bar n\in\{1,\ldots,n_h\}}\big|Y_{\bar n}^{i,N,h}-X_{\bar n}^{i, N, h} \big|\big\|_{\mathscr{L}^{q}(\Omega)}\notag
				\\
				&+\max_{i\in\{1,\ldots,N\}}\big\|\max_{\bar n\in\{1,\ldots,n_h\}}\big|\sum_{j=1}^{\bar n}\big(\Gamma ^h_j(Y_{j-1}^{i,N,h},\mu_{j-1}^{Y,N,h},Y_{j,\eta}^{i,N,h},\mu_{j,\eta}^{Y,N,h},\eta_j)\notag
				\\
				&\qquad-\Gamma ^h_j(X_{j-1}^{i, N,h},\mu_{j-1}^{X, N, h},X_{j,\eta}^{i,N,h},\mu_{j, \eta}^{X, N, h},\eta_j)\big)\big|\big\|_{\mathscr{L}^{q}(\Omega)}, \notag
			\end{align}			
		which on using Corollary \ref{cor:Gam1}  yields
			\begin{align}\label{eq:stable_residual}
				\big\|\mathcal R[Y^h]\big\|_{\mathscr G_{S,q}^{h}} \leq (3+ C_8 T+C_9\sqrt T)\|Y^{h}-X^{h}\|_{\mathscr G_{q}^h}. 
			\end{align}							
			Further, by  rearranging the terms of \eqref{eq:rearrange} and using Minkowski's inequality, one obtains
			\begin{align*}
				\max_{i\in\{1,\ldots,N\}}&\big\|\max_{\bar n\in\{0,\ldots, \bar k\}}\big|Y_{\bar n} 
				^{i,N,h}-X_{\bar n}^{i,N,h} \big|\big\|_{\mathscr{L}^{q}(\Omega)}
				\\
				\leq & \max_{i\in\{1,\ldots,N\}}\big\|\max_{\bar n\in\{1,\ldots,\bar k\}}\big|
				\sum_{j=1}^{\bar n}\big(\Gamma^h_j(Y_{j-1}^{i, N,h},\mu_{j-1}^{Y,N,h},Y_{j,\eta}^{i,N,h},\mu_{j,\eta}^{Y,N,h},\eta_j)
				\\
				&\qquad-\Gamma ^h_j(X_{j-1}^{i, N,h},\mu_{j-1}^{X, N, h},X_{j,\eta}^{i,N,h},\mu_{j, \eta}^{X, N, 
					h},\eta_j)\big)\big|\big\|_{\mathscr{L}^{q}(\Omega)}
				\\
				&+\max_{i\in\{1,\ldots,N\}}\big\||\mathcal R_0^{i,N}[Y^h]|
				\big\|_{\mathscr{L}^{q}(\Omega)}+\max_{i\in\{1,\ldots,N\}} \big\|\max_{\bar n\in\{1,\ldots,\bar k\}}\big|\sum_{j=1}^{\bar n}\mathcal R_j^{i,N}[Y^h]\big|\big\|_{\mathscr{L}^{q}(\Omega)},
							\end{align*}		
and then application of Corollary \ref{cor:Gam1} gives the following, 
\begin{align*}
\max_{i\in\{1,\ldots,N\}}& \big\|\max_{\bar n\in\{0,\ldots, \bar k\}}\big|Y_{\bar n} 
				^{i,N,h}-X_{\bar n}^{i,N,h} \big|\big\|_{\mathscr{L}^{q}(\Omega)}
\\
				\leq & C_8\sum_{j=1}^{\bar k}h_j\max_{i\in\{1,\ldots,N\}}\big\|\max_{\bar n\in\{0,
					\ldots,j-1\}}\big|Y_{\bar n}^{i, N, h}-X_{\bar n}^{i, N, h} \big|\big\|_{\mathscr{L}^{q}(\Omega)}
				\\
				&+C_9\Big(\sum_{j=1}^{\bar k}h_j\max_{i\in\{1,\ldots,N\}}\big\|\max_{\bar n\in\{0,
					\ldots,j-1\}}\big|Y_{\bar n}^{i, N, h}-X_{\bar n}^{i, N, h} \big|\big\|_{\mathscr{L}^{q}(\Omega)}^2\Big)^{1/2}+\big\|\mathcal R[Y^h]\big\|_{\mathscr G_{S,q}^{h}}
				\\
				\leq & C_8\sum_{j=1}^{\bar k}h_j\max_{i\in\{1,\ldots,N\}}\big\|\max_{\bar n\in\{0,
					\ldots,j-1\}}\big|Y_{\bar n}^{i, N, h}-X_{\bar n}^{i, N, h} \big|\big\|_{\mathscr{L}^{q}(\Omega)}
				\\
				&+\max_{i\in\{1,\ldots,N\}}\big\|\max_{\bar n\in\{0,\ldots,\bar k\}}\big|Y_{\bar n}^{i, N, h}-X_{\bar n}^{i, N, h} \big|\big\|_{\mathscr{L}^{q}(\Omega)}^{1/2}
				\\
				&\qquad \times C_9\Big(\sum_{j=1}^{\bar k}h_j\max_{i\in\{1,\ldots,N\}}\big\|\max_{\bar n\in\{0,\ldots,j-1\}}\big|Y_{\bar n}^{i, N, h}-X_{\bar n}^{i, N, h} \big|\big\|_{\mathscr{L}^{q}(\Omega)}\Big)^{1/2}+\big\|\mathcal R[Y^h]\big\|_{\mathscr G_{S,q}^{h}},
\end{align*}			
which due to Young's inequality yields
\begin{align*}
\max_{i\in\{1,\ldots,N\}}& \big\|\max_{\bar n\in\{0,\ldots, \bar k\}}\big|Y_{\bar n} 
				^{i,N,h}-X_{\bar n}^{i,N,h} \big|\big\|_{\mathscr{L}^{q}(\Omega)} 
				\leq \frac{1}{2}\max_{i\in\{1,\ldots,N\}}\big\|\max_{\bar n\in\{0,\ldots,\bar k\}}\big|Y_{\bar n}^{i, N, h}-X_{\bar n}^{i, N, h} \big|\big\|_{\mathscr{L}^{q}(\Omega)}
				\\
				& + \big(C_8+\frac{C_9^2}{2}\big)\sum_{j=1}^{\bar k}h_j\max_{i\in\{1,\ldots,N\}}\big\|\max_{\bar n\in\{0,\ldots,j-1\}}\big|Y_{\bar n}^{i, N, h}-X_{\bar n}^{i, N, h} \big|\big\|_{\mathscr{L}^{q}(\Omega)}+\big\|\mathcal R[Y^h]\big\|_{\mathscr G_{S,q}^{h}}. 
\end{align*}
This further implies
			\begin{align*}
				\max_{i\in\{1,\ldots,N\}}& \big\|\max_{\bar n\in\{0,\ldots,\bar k\}}\big|Y_{\bar n}^{i, N, h}-X_{\bar n}^{i, N, h} \big|\big\|_{\mathscr{L}^{q}(\Omega)} \leq  2\big\|\mathcal R[Y^h]\big\|_{\mathscr G_{S,q}^{h}} 
				\\
				& +(2C_8+C_9^2)\sum_{j=1}^{\bar k}h_j\max_{i\in\{1,\ldots,N\}}\big\|\max_{\bar n\in\{0,\ldots,j-1\}}\big|Y_{\bar n}^{i, N, h}-X_{\bar n}^{i, N, h} \big|\big\|_{\mathscr{L}^{q}(\Omega)}.
			\end{align*}		
			Due to Lemma \ref{lem:Gronwall}, we get 
			\begin{align*}
			 \max_{i\in\{1,\ldots,N\}}& \big\|\max_{\bar n\in\{0,\ldots,\bar k\}}\big|Y_{\bar n}^{i, N, h}-X_{\bar n}^{i, N, h} \big|\big\|_{\mathscr{L}^{q}(\Omega)}  \leq   2e^{(2C_8+C_9^2)T}\big\|\mathcal R[Y^h]\big\|_{\mathscr G_{S,q}^h},
			\end{align*}
for any $\bar{k}\in \{0,\ldots,n_h\}$, which further implies
\begin{align*}
\|Y^{h}-X^{h}\|_{\mathscr G_{q}^h} \leq   2e^{(2C_8+C_9^2)T}\big\|\mathcal R[Y^h]\big\|_{\mathscr G_{S,q}^h},
\end{align*}
and the proof is completed by combining the above with \eqref{eq:stable_residual}.
		\end{proof}		
			\subsection{Consistency of the  Scheme}
			\label{section:consistency}
Recall the temporal grid $\varrho_h$ from \eqref{eq:varho-timeGrid} and set the values of the interacting particle system \eqref{eq:interact} over the grid points  of $\varrho_h$ as	the set  $X^{\varrho_h}:=\{X^{i, N}_{j}\}_{j\in\{0,\ldots,n_h\}}$ where $X_j^{i,N}=X_{t_j}^{i,N}$ for any  $j\in \{0, \ldots, n_h\}$ and $i \in \{1,\ldots,N\}$. 
Notice that $X^{\varrho_h}$ is different from $X^h$, where the later stands for the randomised Milstein scheme \eqref{eq:scheme}.  
Further, $\mathcal R[X^{\varrho_h}]$ is defined using $X^{\varrho_h}$ in \eqref{eq:residual}. 
Next, we introduce the notion of a scheme's consistency and establish that the randomised Milstein scheme \eqref{eq:scheme} is consistent (see below Proposition \ref{prop:consistency}). 
 
		\begin{definition} [Consistency]
			The randomised Milstein scheme \eqref{eq:scheme} for the interacting particle system \eqref{eq:interact} is called consistent of order $\gamma>0$ if there exists a constant $C_{10}>0$, independent of $h$ and $N\in \mathbb{N}$, such that  
			\begin{align*}
				\|\mathcal R[X^{\varrho_h}]\|_{\mathscr G_{S,q}^h}\leq C_{10}h^\gamma,
			\end{align*}
			for $q\geq 2$. 
		\end{definition}

		In the context of the measure dependent drift coefficient, we obtain the following randomised quadrature rule.
		\begin{cor} \label{cor:randomised_quad}
		Let Assumptions \mbox{\normalfont H--\ref{asump:initial_cond}} with $\bar{p}\geq 4$ and set $q=\bar{p}/2\geq 2$, \mbox{\normalfont H--\ref{asump:lip}},   \mbox{\normalfont H-- \ref{asump:holder}}  hold.  Then, for all $i\in\{1,\ldots, N\}$,  the operator $\Theta_{\bar n, \eta}^{h}[\cdot]$ defined in \eqref{eq:Theata1} and applied to $b$ satisfies 
			\begin{align} \notag
				\big\|\max_{\bar n\in\{1,\ldots,n_{h}\}}\big|\Theta_{\bar n, \eta}^{h}[b]-\int_0^{t_{\bar n}}b(s,X_s^{i,N},\mu_s^{X,N})ds\big|\big\|_{\mathscr{L}^{q}(\Omega)}\leq C_{11} h,
			\end{align}
		  where $C_{11}:= \sqrt{T} \textcolor{black}{\bar{C}_q}  \big(2\bar L+4\bar LC_1^{1/\bar p}+2\textcolor{black}{\bar L}C_2^{1/\bar p}\big)\big(1+\max_{i\in\{1,\ldots,N\}}\|X_0^i\|^{\bar{p}}_{{\mathscr L^{\bar{p}}(\tilde \Omega)}}\big)^{1/\bar{p}}$ and the positive constants $C_1$, $C_2$, \textcolor{black}{$\bar{C}_q$} appear in  Proposition \ref{prop:sde_bound}, Lemma \ref{lem:dif_sde} and \textcolor{black}{Lemma \ref{lem:Burkholder}}, respectively.
		\end{cor}
		\begin{proof}  
			We take $V(s)=b\big(s,X_s^{i,N},\mu_s^{X,N}\big)$ for any $s\in [0,T]$ and $i \in \{1,\ldots,N\}$, which is defined on $(\tilde{\Omega}, \tilde{\mathscr{F}}, \tilde{\mathbb{P}})$.
			Notice that Remarks \ref{rem:linear}  and  \ref{rem:interact_bound} along with H\"older's inequality (as $q=\bar{p}/2$) imply
			\begin{align*}
				\|V\|_{\mathscr{L}^{q}([0,T]\times \tilde{\Omega})} &=\Big( \int_0^T \|b\big(s,X_s^{i,N},\mu_s^{X,N}\big)\|_{\mathscr{L}^{q}( \tilde{\Omega})}^{q}ds \Big)^{1/q}
				\\
				&\leq \bar L  \Big(\int_0^T \big(1+\|X_s^{i,N}\|_{\mathscr{L}^{q}( \tilde{\Omega})}+\|\mathcal{W}_2(\mu_s^{X,N}, \delta_0)\|_{\mathscr{L}^{q}( \tilde{\Omega})}\big)^{q} ds\Big)^{1/q}
				\\
				& \leq \bar L \Big(\int_0^T (1+2\max_{i \in \{1,\ldots,N\}}\|X_s^{i,N}\|_{\mathscr{L}^{q}(\tilde\Omega)})^{q} ds\Big)^{1/q}
				\\
				& \leq \bar L \Big(\int_0^T (1+2\max_{i \in \{1,\ldots,N\}}\|X_s^{i,N}\|_{\mathscr{L}^{\bar{p}}(\tilde\Omega)})^{q} ds\Big)^{1/q}
				\\
				& \leq \bar L \Big(\int_0^T (1+2C_1^{1/\bar{p}}(1+\max_{i\in\{1,\ldots,N \}}\|X_0^{i}\|^{\bar{p}}_{\mathscr{L}^{\bar{p}}(\tilde\Omega)})^{1/\bar{p}})^{q} ds\Big)^{1/q}
				\\
				&\leq \bar{L} ( T)^{1/q} (1+2C_1^{1/\bar{p}}(1+\max_{i\in\{1,\ldots,N \}}\|X_0^{i}\|^{\bar{p}}_{\mathscr{L}^{\bar{p}}(\tilde\Omega)})^{1/\bar{p}}) <\infty. 
			\end{align*}
			Also, due to Assumptions H--\ref{asump:lip}, H--\ref{asump:holder} and Lemma \ref{lem:dif_sde}, for $t, t'\in[0, T]$,
			\begin{align}
				&\big\|V(t)-V(t')\big\|_{\mathscr{L}^{q}(\tilde\Omega)}
				=\big\|b(t,X_t^{i,N},\mu_t^{X,N})-b(t',X_{t'}^{i,N},\mu_{t'}^{X,N})\big\|_{\mathscr{L}^{q}(\tilde\Omega)}
				\leq L \big\{\big\|X_t^{i, N}-X_{t'}^{i, N}\big\|_{\mathscr{L}^{q}(\tilde\Omega)}\notag
				\\
				&\quad +\big\|\mathcal{W}_2(\mu_t^{X, N}, \mu_{t'}^{X, N})\big\|_{\mathscr{L}^{q}(\tilde\Omega)}+|t-t'|^{1/2}\big(1+\big\|X^{i,N}_t\big\|_{\mathscr{L}^{q}(\tilde\Omega)}+ \big\|\mathcal{W}_2(\mu_t^{X, N},\delta_0)\big\|_{\mathscr{L}^{q}(\tilde\Omega)}\big) \}\notag
				\\
				&\leq 2L\max_{i\in\{1,\ldots,N\}}\big\|X_t^{i, N}-X_{t'}^{i, N}\big\|_{\mathscr{L}^{q}(\tilde\Omega)} +  |t-t'|^{1/2} L\{1+2C_1^{1/\bar{p}}(1+\max_{i\in\{1,\ldots,N\}}\|X_0^i\|^{\bar{p}}_{_{\mathscr L^{\bar{p}}(\tilde \Omega)}})^{1/\bar{p}}\} \notag
				\\
				&\leq \{2L(C_1^{1/\bar{p}}+C_2^{1/\bar{p}})(1+\max_{i\in\{1,\ldots,N\}}\|X_0^i\|^{\bar{p}}_{{\mathscr L^{\bar{p}}(\tilde \Omega)}})^{1/\bar{p}}+L\}|t-t'|^{1/2}, \label{eq:b11}
			\end{align}
			for all $i\in\{1,\ldots, N\}$. Due to \eqref{eq:norm:hold},  
			$\|V\|_{\mathscr{C}^{1/2}([0,T],\,\mathscr{L}^{q}(\tilde\Omega))}\leq 2\bar L+\bar L (4C_1^{1/\bar{p}}+2C_2^{1/\bar{p}}\big)\big(1+\max_{i\in\{1,\ldots,N\}}\|X_0^i\|^{\bar{p}}_{{\mathscr L^{\bar{p}}(\tilde \Omega)}}\big)^{1/\bar{p}}$, as $L\leq \bar L$. 
			Then, \eqref{eq:quad} completes the proof. 
		\end{proof}
			\begin{lem}\label{lem:Zeta^1}
	Let Assumptions \mbox{\normalfont H--\ref{asump:initial_cond}} with $\bar{p}\geq 4$ and set $q=\bar{p}/2\geq 2$, \mbox{\normalfont{H--\ref{asump:lip}}}   to \mbox{\normalfont{H--\ref{asump:derv_lip}}}  hold. Let $h\leq \min(1,T)$. Then,
	\begin{align*}
		&\big\|\max_{\bar n\in\{1,\ldots, n_h\}}\big|\sum_{ j=1}^{\bar n}\sum_{\ell=1}^{m_1}\int_{t_{j-1}}^{t_j}\big[\sigma_1^{\ell}(s, X_s^{i, N},\mu_s^{X, N})-\tilde{\sigma}_1^{\ell}(s,t_{j-1}, X_{j-1}^{i, N},\mu_{j-1}^{X, N})\big]dW^{i,\ell}_s\big\|_{\mathscr{L}^{q}(\tilde\Omega)}
		\\
		&+\big\|\max_{\bar n\in\{1,\ldots, n_h\}}\big|\sum_{ j=1}^{\bar n}\sum_{\ell=1}^{m_0}\int_{t_{j-1}}^{t_j}\big[\sigma_0^{\ell}(s, X_s^{i, N},\mu_s^{X, N})-\tilde{\sigma}_0^{\ell}(s,t_{j-1}, X_{j-1}^{i, N},\mu_{j-1}^{X, N})\big]dW^{0,\ell}_s\big|\big\|_{\mathscr{L}^{q}(\tilde\Omega)}
		\\
		&\hspace{4cm}\leq  C_{12} h,
	\end{align*}
	for all $i\in\{1,\ldots, N\}$, where 
	\begin{align*}
C_{12}&:=(m_0+m_1) \bar C_{q} \sqrt{{Tq(q-1)}/{2}}\Big\{L+2 \bar L^2+4L^2\sqrt{{q(q-1)}/{2}}
	+\Big(2(L+2\bar L^2)C_1^{1/\bar p}
	\\
	&
	\qquad +8L^2
   \sqrt{{q(q-1)}{2}}\big(C_1^{1/\bar p}+C_2^{1/\bar p}\big)
	+4LC_2^{2/\bar p}\Big)
	\big(1+\max_{i\in\{1,\ldots,N\}}\|X_0^i\|_{_{\mathscr L^{\bar p}(\tilde \Omega)}}^{\bar p}\big)^{2/\bar p} \Big\}
	\end{align*}
	and  the positive constants $C_1$, $C_2$ and $\bar C_q$ appear in  Proposition \ref{prop:sde_bound}, Lemma \ref{lem:dif_sde}  and Lemma \ref{lem:Burkholder}, respectively.
\end{lem}
\begin{proof}
	For notational simplicity, define \[\zeta^{\ell}(s):=\sigma_1^{\ell}(s, X_s^{i, N},\mu_s^{X, N})-\tilde{\sigma}_1^{\ell}(s,t_{j-1}, X_{j-1}^{i, N},\mu_{j-1}^{X, N})\] for all $\ell\in\{1,\ldots,m_1\}$, $i\in\{1,\ldots, N\}$, $j \in \{1,\ldots,n_h\}$ and $s\in[t_{j-1},t_j]$. 
	Notice  that 
	$$\Big\{\displaystyle \sum_{j=1}^{\bar n}\sum_{\ell=1}^{m_1} \int_{t_j-1}^{t_j}\zeta^{\ell}(s)dW^{i,\ell}_s \Big\}_{\bar{n}\in \{0,\ldots,n_h\}}$$
	 	is  an $\{\tilde{\mathscr F}_{t_{\bar n}}\}_{\bar n\in\{0,\ldots,n_h\}}$-adapted standard martingale.    
	Thus, by using Lemma \ref{lem:Burkholder} and Theorem 7.1 in \cite{Mao2008}, we get
	\begin{align}
		\big\|&\max_{\bar n\in\{0,\ldots, n_h\}}\big|\sum_{ j=1}^{\bar n}\sum_{\ell=1}^{m_1}\int_{t_{j-1}}^{t_j}\zeta^{\ell}(s)dW^{i,\ell}_s\big|\big\|_{\mathscr{L}^{q}(\tilde\Omega)}
		\leq \bar C_{q}\big\|\big(\sum_{ j=1}^{ n_h}\big|\sum_{\ell=1}^{m_1}\int_{t_{j-1}}^{t_j}\zeta^{\ell}(s)dW^{i,\ell}_s\big|^2\big)^{1/2}\big\|_{\mathscr{L}^{q}(\tilde\Omega)} \notag
		\\
		&= \bar C_{q} \big\|\sum_{ j=1}^{ n_h}\big|\sum_{\ell=1}^{m_1}\int_{t_{j-1}}^{t_j}\zeta^{\ell}(s)dW^{i,\ell}_s\big|^2\big\|_{\mathscr{L}^{q/2}(\tilde\Omega)}^{1/2}  
		\leq \sqrt m_1 \bar C_{q} \Big(\sum_{ j=1}^{ n_h}\sum_{\ell=1}^{m_1}\big\|\int_{t_{j-1}}^{t_j}\zeta^{\ell}(s)dW^{i,\ell}_s\big\|^2_{\mathscr{L}^{q}(\tilde\Omega)}\Big)^{1/2} \notag
		\\
		&\leq \sqrt m_1 \bar C_{q} \Big(\frac{q(q-1)}{2}\Big)^{1/2} \Big(\sum_{ j=1}^{ n_h}\sum_{\ell=1}^{m_1} h_j^{(q-2)/{q}}\big(\int_{t_{j-1}}^{t_j}\big\| \zeta^{\ell}(s)\big\|^{q}_{\mathscr{L}^{q}(\tilde\Omega)} ds\big)^{2/q}\Big)^{1/2} , \label{eq:fin}
	\end{align}
	for all $i\in\{1,\ldots, N\}$. 
Now, recall $\tilde{\sigma}_1^{\ell}$ from \eqref{eq:sigma_tilde} to write the following for any $s\in [t_{j-1}, t_j]$, 
	\begin{align*} 
		\zeta^{\ell}(s)=&\big[\sigma_1^{\ell}(s, X_s^{i, N},\mu_s^{X, N})-\sigma_1^{\ell}(t_{j-1}, X_s^{i, N},\mu_s^{X, N})\big] +\big[\sigma_1^{\ell}(t_{j-1}, X_s^{i, N},\mu_s^{X, N}) \notag
		\\
		&-\sigma_1^{\ell}(t_{j-1}, X_{j-1}^{i, N},\mu_{j-1}^{X, N})-\partial_x\sigma_1^{\ell}(t_{j-1}, X_{j-1}^{i, N},\mu_{j-1}^{X, N})(X_s^{i,N}-X_{j-1}^{i,N})\notag
		\\
		&-\frac{1}{N}\sum_{k= 1}^{N}\partial_\mu\sigma_1^{\ell}(t_{j-1}, X_{j-1}^{i, N},\mu_{j-1}^{X, N}, X_{j-1}^{k, N})(X_s^{k,N}-X_{j-1}^{k,N})\big] +\big[\partial_x\sigma_1^{\ell}(t_{j-1}, X_{j-1}^{i, N},\mu_{j-1}^{X, N})\notag
		\\
		&\times(X_s^{i,N}-X_{j-1}^{i,N})+\frac{1}{N}\sum_{k= 1}^{N}\partial_\mu\sigma_1^{\ell}(t_{j-1}, X_{j-1}^{i, N},\mu_{j-1}^{X, N}, X_{j-1}^{k, N})(X_s^{k,N}-X_{j-1}^{k,N})\notag
		\\
		&\hspace{1cm}-\Lambda_{\sigma_1\sigma_1}^\ell(s,t_{j-1}, X_{j-1}^{i, N},\mu_{j-1}^{X, N})-\Lambda_{\sigma_1\sigma_0}^\ell(s,t_{j-1}, X_{j-1}^{i, N},\mu_{j-1}^{X, N})\big], \notag
	\end{align*}
	which on the application of Assumption H--\ref{asump:holder}, Lemma \ref{lem:mvt} along with  \eqref{eq:interact} and the values of $\Lambda_{\sigma_1\sigma_1}^\ell$, $\Lambda_{\sigma_1\sigma_0}^\ell$ from \eqref{eq:lamb}, \eqref{eq:lamb*} yields 
	\begin{align*}
		\zeta^{\ell}&(s)\leq   L (1+|X_s^{i, N}|+\mathcal{W}_2(\mu_s^{X, N}, \delta_0))|s-t_{j-1}|+\frac{3L}{2}|X_s^{i, N}-X_{j-1}^{i, N}|^2
		\\
		&+ \frac{5L}{2}\frac{1}{N} \sum_{k=1}^N  |X_s^{k, N}-X_{j-1}^{k, N}|^2 + \int_{t_{j-1}}^s\partial_x\sigma_1^{\ell}(t_{j-1}, X_{j-1}^{i, N},\mu_{j-1}^{X, N})b(r,X_r^{i,N},\mu_r^{X,N})dr\notag
		\\
		&+\int_{t_{j-1}}^s\partial_x\sigma_1^{\ell}(t_{j-1}, X_{j-1}^{i, N},\mu_{j-1}^{X, N})(\sigma_1(r, X_{r}^{i, N},\mu_{r}^{X, N})-\sigma_1(t_{j-1}, X_{j-1}^{i, N},\mu_{j-1}^{X, N}))dW^i_r\notag
		\\
		&+\int_{t_{j-1}}^s\partial_x\sigma_1^{\ell}(t_{j-1}, X_{j-1}^{i, N},\mu_{j-1}^{X, N})(\sigma_0(r, X_{r}^{i, N},\mu_{r}^{X, N})-\sigma_0(t_{j-1}, X_{j-1}^{i, N},\mu_{j-1}^{X, N}))dW^{0}_r\notag
		\\
		&+\frac{1}{N}\sum_{k= 1}^{N}\int_{t_{j-1}}^s\partial_\mu\sigma_1^{\ell}(t_{j-1}, X_{j-1}^{i, N},\mu_{j-1}^{X, N}, X_{j-1}^{k, N})b(r,X_r^{k,N},\mu_r^{X,N})dr\notag
		\\
		&+\frac{1}{N}\sum_{k= 1}^{N}\int_{t_{j-1}}^s\partial_\mu\sigma_1^{\ell}(t_{j-1}, X_{j-1}^{i, N},\mu_{j-1}^{X, N}, X_{j-1}^{k, N})\notag
		\\
		&\hspace{3cm}\times(\sigma_1(r, X_{r}^{k, N},\mu_{r}^{X, N})-\sigma_1(t_{j-1}, X_{j-1}^{k, N},\mu_{j-1}^{X, N}))dW^{k}_r\notag
		\\
		&+\frac{1}{N}\sum_{k= 1}^{N}\int_{t_{j-1}}^s\partial_\mu\sigma_1^{\ell}(t_{j-1}, X_{j-1}^{i, N},\mu_{j-1}^{X, N}, X_{j-1}^{k, N})\notag
		\\
		&\hspace{3cm}\times(\sigma_0(r, X_{r}^{k, N},\mu_{r}^{X, N})-\sigma_0(t_{j-1}, X_{j-1}^{k, N},\mu_{j-1}^{X, N}))dW^{0}_r\notag,
	\end{align*}
	for all $i \in\{1,\ldots,N\}$, $j\in\{1,\ldots,n_h\}$ and $\ell\in\{1,\ldots,m_1\}$. This further  implies, due to Remarks \ref{rem:linear} and \ref{rem:derv_bound},
	\begin{align}
		\|&\zeta^{\ell}(s) \|_{\mathscr{L}^{q}(\tilde\Omega)}\leq   L (1+2\max_{i\in\{1,\ldots,N\}}\|X_s^{i, N}\|_{\mathscr{L}^{q}(\tilde\Omega)})h_{j}+4L\max_{i\in\{1,\ldots,N\}}\|X_s^{i, N}-X_{j-1}^{i, N}\|_{\mathscr{L}^{2q}(\tilde\Omega)}^2 \notag
		\\
		&+ h_{j}^{(q-1)/{q}}2\bar L^2\big( \int_{t_{j-1}}^s (1+2\max_{i\in\{1,\ldots,N\}}\|X_r^{i,N}\|_{\mathscr{L}^{q}(\tilde\Omega)})^{q}dr \big)^{1/q}\notag
		\\
		&+\big\|\int_{t_{j-1}}^s\partial_x\sigma_1^{\ell}(t_{j-1}, X_{j-1}^{i, N},\mu_{j-1}^{X, N})(\sigma_1(r, X_{r}^{i, N},\mu_{r}^{X, N})-\sigma_1(t_{j-1}, X_{j-1}^{i, N},\mu_{j-1}^{X, N}))dW^{i}_r  \big\|_{\mathscr{L}^{q}(\tilde\Omega)}\notag
		\\
		&+\big\|\int_{t_{j-1}}^s\partial_x\sigma_1^{\ell}(t_{j-1}, X_{j-1}^{i, N},\mu_{j-1}^{X, N})(\sigma_0(r, X_{r}^{i, N},\mu_{r}^{X, N})-\sigma_0(t_{j-1}, X_{j-1}^{i, N},\mu_{j-1}^{X, N}))dW^{0}_r \big\|_{\mathscr{L}^{q}(\tilde\Omega)}\notag
		\\
		&+\big\|\frac{1}{N}\sum_{k= 1}^{N}\int_{t_{j-1}}^s\partial_\mu\sigma_1^{\ell}(t_{j-1}, X_{j-1}^{i, N},\mu_{j-1}^{X, N}, X_{j-1}^{k, N})\notag
		\\
		&\hspace{3cm}\times(\sigma_1(r, X_{r}^{k, N},\mu_{r}^{X, N})-\sigma_1(t_{j-1}, X_{j-1}^{k, N},\mu_{j-1}^{X, N}))dW^{k}_r \big\|_{\mathscr{L}^{q}(\tilde\Omega)}\notag
		\\
		&+\big\|\frac{1}{N}\sum_{k= 1}^{N}\int_{t_{j-1}}^s\partial_\mu\sigma_1^{\ell}(t_{j-1}, X_{j-1}^{i, N},\mu_{j-1}^{X, N}, X_{j-1}^{k, N})\notag
		\\
		&\hspace{3cm}\times(\sigma_0(r, X_{r}^{k, N},\mu_{r}^{X, N})-\sigma_0(t_{j-1}, X_{j-1}^{k, N},\mu_{j-1}^{X, N}))dW^{0}_r \big\|_{\mathscr{L}^{q}(\tilde\Omega)}. \notag
	\end{align}
	Moreover, the application of H\"older's inequality (as $q=\bar{p}/2)$,  Remarks \ref{rem:interact_bound} and \ref{rem:derv_bound}, Lemma \ref{lem:dif_sde}  and Theorem 7.1 in \cite{Mao2008}    yield
	\begin{align}
		\|&\zeta^{\ell}(s) \|_{\mathscr{L}^{q}(\tilde\Omega)}\leq   \big\{(L+2\bar L^2)\big(1+2C_1^{1/\bar p}(1+\max_{i\in\{1,\ldots,N\}}\|X_0^{i, N}\|_{\mathscr{L}^{\bar p}(\tilde\Omega)}^{\bar p})^{1/\bar p}\big)\notag
		\\
		&+4LC_2^{2/\bar p}\big(1+\max_{i\in\{1,\ldots,N\}}\|X_0^{i, N}\|_{\mathscr{L}^{\bar p}(\tilde\Omega)}^{\bar p}\big)^{2/\bar p}\big\}h_j \notag
		\\
		&+\Big(\frac{q(q-1)}{2}\Big)^{1/2} Lh_j^{(q-2)/{2q}}\sum_{u=0}^{1}\big(\int_{t_{j-1}}^s\big\|\sigma_u(r, X_{r}^{i, N},\mu_{r}^{X, N})-\sigma_u(t_{j-1}, X_{j-1}^{i, N},\mu_{j-1}^{X, N})\big\|_{\mathscr{L}^{q}(\tilde\Omega)}^{q}dr\big)^{1/{q}}  \notag
		\\
		&+\Big(\frac{q(q-1)}{2}\Big)^{1/2} Lh_j^{(q-2)/{2q}}\sum_{u=0}^{1}\frac{1}{N}\sum_{k= 1}^{N}\big(\int_{t_{j-1}}^s \big\|\sigma_u(r, X_{r}^{k, N},\mu_{r}^{X, N})-\sigma_u(t_{j-1}, X_{j-1}^{k, N},\mu_{j-1}^{X, N})\big\|_{\mathscr{L}^{q}(\tilde\Omega)}^{q}dr\big)^{1/{q}}, \label{eq:zeta1}
	\end{align}
	for any $s\in[t_{j-1}, t_j]$, $i\in\{1,\ldots,N\}$, $j\in\{1,\ldots,n_h\}$ and  $l\in\{1,\ldots, m_1\}$.
	Further, by using Assumptions  H--\ref{asump:lip} and H--\ref{asump:holder},  Remark \ref{rem:interact_bound}, H\"older's inequality, Lemma \ref{lem:dif_sde} and $h\leq 1$, one obtains
	\begin{align}
		&\big\|\sigma_u(r, X_{r}^{i, N},\mu_{r}^{X, N})-\sigma_u(t_{j-1}, X_{j-1}^{i, N},\mu_{j-1}^{X, N})\big\|_{\mathscr{L}^{q}(\tilde\Omega)}^{q} \leq  L \big\{\big\|X_r^{i, N}-X_{j-1}^{i, N}\big\|_{\mathscr{L}^{q}(\tilde\Omega)}\notag
		\\
		&\quad +\big\|\mathcal{W}_2(\mu_r^{X, N}, \mu_{j-1}^{X, N})\big\|_{\mathscr{L}^{q}(\tilde\Omega)}+|r-t_{j-1}|\big(1+\big\|X^{i,N}_{j-1}\big\|_{\mathscr{L}^{q}(\tilde\Omega)}+ \big\|\mathcal{W}_2(\mu_{j-1}^{X, N},\delta_0)\big\|_{\mathscr{L}^{q}(\tilde\Omega)}\big) \}\notag
		\\
		&\leq  2L\max_{i\in\{1,\ldots,N\}}\big\|X_r^{i, N}-X_{j-1}^{i, N}\big\|_{\mathscr{L}^{q}(\tilde\Omega)}+|r-t_{j-1}| L\{1+2C_1^{1/\bar p}(1+\max_{i\in\{1,\ldots,N\}}\|X_0^i\|_{_{\mathscr L^{\bar p}(\tilde \Omega)}}^{\bar p})^{1/\bar p}\} \notag
		\\
		&\leq  \{2L(C_1^{1/\bar p}+C_2^{1/\bar p})(1+\max_{i\in\{1,\ldots,N\}}\|X_0^i\|_{_{\mathscr L^{\bar p}(\tilde \Omega)}}^{\bar p})^{1/\bar p}+L\}h_j
		^{1/2} \label{eq:sigma11},
	\end{align}
		for all $i \in\{1,\ldots,N\}$, $j\in\{1,\ldots,n_h\}$, $u\in\{0,1\}$ and $r\in[t_{j-1}, t_{j}]$. 
		On substituting  \eqref{eq:sigma11} in  \eqref{eq:zeta1}, one obtains
	\begin{align*}
		&\|\zeta^{\ell}(s) \|_{\mathscr{L}^{q}(\tilde\Omega)} \leq   \Big\{L+2\bar L^2+4L^2\Big(\frac{q(q-1)}{2}\Big)^{1/2}
		+\Big(2(L+2\bar L^2)C_1^{1/\bar p}
	   +8L^2\big(C_1^{1/\bar p}+C_2^{1/\bar p}\big)
	   \\
	   &\times\Big(\frac{q(q-1)}{2}\Big)^{1/2}
		+4LC_2^{2/\bar p}\Big)
	  \big(1+\max_{i\in\{1,\ldots,N\}}\|X_0^i\|_{_{\mathscr L^{\bar p}(\tilde \Omega)}}^{\bar p}\big)^{2/\bar p} \Big\}h_{j},
	\end{align*} 
	which, substituting in \eqref{eq:fin}, gives the first estimate. The second estimate can be proved by similar arguments. 
\end{proof}
		\begin{prop} \label{prop:consistency}
			Let Assumptions \mbox{\normalfont H--\ref{asump:initial_cond}} with $\bar{p}\geq 4$ and set $q=\bar{p}/2\geq 2$, \mbox{\normalfont{H--\ref{asump:lip}}}  to \mbox{\normalfont{H--\ref{asump:derv_lip}}} \it hold. Then, the drift-randomised Milstein scheme  \eqref{eq:scheme} is consistent of order 1, i.e., 
			\begin{align*}
				\big\|\mathcal R[ X^{\varrho_h}]\big\|_{\mathscr G_{S,q}^{h}}
				\leq C_{10}h,
			\end{align*}	
			   where 
				\begin{align*}
C_{10} :=2L\big(1+2\sqrt{{q(q-1)}/{2}}\big)\big\{L+& 2L(C_1^{1/\bar{p}}+C_2^{1/\bar{p}})
\\
& \times 
			\big(1+\max_{i\in\{1,\ldots,N\}}\|X_0^i\|^{\bar{p}}_{{\mathscr L^{\bar{p}}(\tilde \Omega)}}\big)^{1/\bar{p}}\big\}+C_{11}+C_{12}
				\end{align*} and the positive constants $C_1$, $C_2$, $C_{11}$ and  $C_{12}$ appear in Proposition \ref{prop:sde_bound},  Lemma \ref{lem:dif_sde},  Corollary \ref{cor:randomised_quad} and  Lemma \ref{lem:Zeta^1}, respectively.
		\end{prop}
		\begin{proof}
			Let us recall the residual $\mathcal R_j^{i, N}[X^{\varrho_h}]$ from \eqref{eq:residual}, the interacting particle system from \eqref{eq:interact}  and $\Gamma^h_j$ from  \eqref{eq:Gamma} to write $\mathcal R_0^{i, N}[X^{\varrho_h}]=0$ and 
			\begin{align}
				\mathcal R_j^{i, N}[X^{\varrho_h}]&=X_{j}^{i, N}-X_{j-1}^{i, N}-\Gamma^h_j (X_{j-1}^{i, N},\mu_{j-1}^{X,N},X_{j,\eta}^{i, N},\mu_{j,\eta}^{X, N},\eta_j) \notag
				\\
				=&\int_{t_{j-1}}^{t_j}\big[b(s, X_s^{i, N},\mu_s^{X, N})-b(t_{j-1}+\eta_jh_j, X_{j,\eta}^{i,N}, \mu_{j,\eta}^{X,N})\big]ds \notag
				\\
				&+\sum_{\ell=1}^{m_1}\int_{t_{j-1}}^{t_j}\big[\sigma_1^{\ell}(s, X_s^{i, N},\mu_s^{X, N})-\tilde{\sigma}_1^{\ell}(s,t_{j-1}, X_{j-1}^{i, N},\mu_{j-1}^{X, N})\big]dW^{i,\ell}_s \notag
				\\
				&+\sum_{\ell=1}^{m_0}\int_{t_{j-1}}^{t_j}\big[\sigma_0^{\ell}(s, X_s^{i, N},\mu_s^{X, N})-\tilde{\sigma}_0^{\ell}(s,t_{j-1}, X_{j-1}^{i, N},\mu_{j-1}^{X, N})\big]dW_s^{0,\ell}, \label{eq:resi}
			\end{align}
			for all $i\in\{1,\ldots,N\}$ and $j\in\{1,\ldots,n_h\}$. 	
			For any $\bar{n}\in\{1,\ldots,n_h\}$, from \eqref{eq:Theata1}, $$\Theta_{\bar n, \eta}^{ h}\big[b\big]=\displaystyle\sum_{j=1}^{\bar n}h_jb(t_{j-1}+\eta_jh_j, X^{i, N}_{t_{j-1}+\eta_jh_j},\mu_{t_{j-1}+\eta_jh_j}^{X, N}),$$
			which along with Assumption H--\ref{asump:lip} yields
			\begin{align*}
				\sum_{j=1}^{\bar{n}}&\int_{t_{j-1}}^{t_j}\big[b(s, X_s^{i, N},\mu_s^{X, N})-b(t_{j-1}+\eta_jh_j, X_{j,\eta}^{i,N}, \mu_{j,\eta}^{X,N})\big]ds
				\\
				=  & \sum_{j=1}^{\bar{n}}\int_{t_{j-1}}^{t_j}\big[b(s, X_s^{i, N},\mu_s^{X, N})-b(t_{j-1}+\eta_jh_j, X^{i, N}_{t_{j-1}+\eta_jh_j},\mu_{t_{j-1}+\eta_jh_j}^{X, N})\big]ds
				\\
				&+\sum_{j=1}^{\bar{n}}\int_{t_{j-1}}^{t_j}\big[b(t_{j-1}+\eta_jh_j, X^{i, N}_{t_{j-1}+\eta_jh_j},\mu_{t_{j-1}+\eta_jh_j}^{X, N})-b(t_{j-1}+\eta_jh_j, X_{j,\eta}^{i,N}, \mu_{j,\eta}^{X,N})\big]ds
				\\
				\leq & \int_{0}^{t_{\bar{n}}}b(s, X_s^{i, N},\mu_s^{X, N}) ds- \Theta_{\bar n, \eta}^{ h}\big[b\big] 
				\\
				&+L\sum_{j=1}^{\bar{n}}h_j\big[\big|X^{i, N}_{t_{j-1}+\eta_jh_j}-X_{j,\eta}^{i,N}\big|+\big|\mathcal{W}_2(\mu_{t_{j-1}+\eta_jh_j}^{X, N}, \mu_{j,\eta}^{X,N})\big|\big],
			\end{align*}
			and thus from  \eqref{eq:resi} one obtains
			\begin{align} \label{eq:resid}
				\big\|&\max_{\bar n\in\{1,\ldots,n_h\}}\big|\sum_{j=1}^{\bar n}\mathcal R^{i, N}_j[X^{\varrho_h}]\big|\big\|_{\mathscr{L}^{q}(\Omega)}
				\leq \big\|\max_{\bar n\in\{1,\ldots,n_{h}\}}\big|\Theta^{h}_{\bar n, \eta}[b]-\int_0^{t_{\bar n}}b(s, X_s^{i, 
					N},\mu_s^{X, N})ds\big|\big\|_{\mathscr{L}^{q}(\Omega)}\notag
				\\
				&+L\sum_{j=1}^{n_h}h_j\big[\big\|X^{i, N}_{t_{j-1}+\eta_jh_j}-X_{j,\eta}^{i,N}\big\|_{\mathscr{L}^{q}(\Omega)}
				+\big\|\mathcal{W}_2(\mu_{t_{j-1}+\eta_jh_j}^{X, N}, \mu_{j,\eta}^{X,N})\big\|_{\mathscr{L}^{q}(\Omega)}
				\big]\notag
				\\
				&+\big\|\max_{\bar n\in\{1,\ldots,n_h\}}\big|\sum_{j=1}^{\bar{n}} \sum_{\ell=1}^{m_1}\int_{t_{j-1}}^{t_j}\big[\sigma_1^{\ell}(s, X_s^{i, N},\mu_s^{X, N})-\tilde{\sigma}_1^{\ell}(s,t_{j-1}, X_{j-1}^{i, N},\mu_{j-1}^{X, N})\big]dW^{i,\ell}_s\big|\big\|_{\mathscr{L}^{q}(\Omega)} \notag
				\\
				&+\big\|\max_{\bar n\in\{1,\ldots,n_h\}}\big|\sum_{j=1}^{\bar{n}} \sum_{\ell=1}^{m_0}\int_{t_{j-1}}^{t_j}\big[\sigma_0^{\ell}(s, X_s^{i, N},\mu_s^{X, N})-\tilde{\sigma}_0^{\ell}(s,t_{j-1}, X_{j-1}^{i, N},\mu_{j-1}^{X, N})\big]dW_s^{0,\ell}\big|\big\|_{\mathscr{L}^{q}(\Omega)} ,
			\end{align}
			for all $i\in\{1,\ldots,N\}$. 
			Furthermore,   Theorem 7.1 in \cite{Mao2008} yields
			\begin{align*}
				&\big\|X^{i, N}_{t_{j-1}+\eta_jh_j}-X_{j,\eta}^{i,N}\big\|_{\mathscr{L}^{q}(\Omega)}
				\leq  h_j^{(q-1)/q}\Big(\int_{t_{j-1}}^{t_{j}} \big\|b(s, X_s^{i, N},\mu_s^{X, N})-b(t_{j-1}, X_{j-1}^{i, N},\mu_{j-1}^{X, N})\big\|_{\mathscr{L}^{q}(\Omega)}^{q}ds\Big)^{1/q}\notag
				\\
				&\quad+\Big(\frac{q(q-1)}{2}\Big)^{1/2} h_j^{(q-2)/2q}\sum_{u=0}^1 \Big(\int_{t_{j-1}}^{t_{j}}\big\|\sigma_u(s, X_s^{i, N},\mu_s^{X, N})-\sigma_u(t_{j-1}, X_{j-1}^{i, N},\mu_{j-1}^{X, N})\big\|_{\mathscr{L}^{q}(\Omega)}^{q}ds\Big)^{1/q}, \notag
			\end{align*}
		which on using Equations \eqref{eq:b11} and \eqref{eq:sigma11},
			\begin{align} \label{eq:dif_randomised_X}
				\big\|X^{i, N}_{t_{j-1}+\eta_jh_j}-X_{j,\eta}^{i,N}\big\|_{\mathscr{L}^{q}(\Omega)}\leq & \Big(1+2\Big(\frac{q(q-1)}{2}\Big)^{1/2}\Big)\Big\{L+2L(C_1^{1/\bar{p}}+C_2^{1/\bar{p}})\notag
				\\
				&\times\big(1+\max_{i\in\{1,\ldots,N\}}\|X_0^i\|^{\bar{p}}_{{\mathscr L^{\bar{p}}(\tilde \Omega)}}\big)^{1/\bar{p}}
				\Big\}h_j,
			\end{align}
			for all $i\in\{1,\ldots,N\}$ and $j\in\{1,\ldots,n_h\}$. Also, 
			\begin{align} \label{eq:uu}
				\big\|\mathcal{W}_2(\mu_{t_{j-1}+\eta_jh_j}^{X, N}, \mu_{j,\eta}^{X,N})\big\|_{\mathscr{L}^{q}(\Omega)}\leq \max_{i \in\{1,\ldots,N\}}\big\|X^{i, N}_{t_{j-1}+\eta_jh_j}-X_{j,\eta}^{i,N}\big\|_{\mathscr{L}^{q}(\Omega)},
			\end{align}
			for all $j\in\{1,\ldots,n_h\}$. 
			The proof is completed by substituting \eqref{eq:dif_randomised_X} and \eqref{eq:uu} in \eqref{eq:resid} and using  Corollary \ref{cor:randomised_quad} and Lemma~\ref{lem:Zeta^1}.
		\end{proof}
		
		\subsection{Rate of Convergence of the Scheme} \label{sec:rate}
		After proving bistability and consistency of the scheme \eqref{eq:scheme}, the  main result,  Theorem \ref{thm:mainresult}, follows immediately.  
		
		\begin{proof}[Proof of Theorem \ref{thm:mainresult}]
			Set $Y^h$ in Proposition \ref{prop:residual} as $Y^h=X^{\varrho_h}=\{X^{i, N}_{j}\}_{j\in\{0,\ldots,n_h\}}\in {\mathscr G_{q}^h}$ (as introduced in Section \ref{section:consistency}) to obtain,  
			\begin{align*}
				\|X^{\varrho_h}-X^{h}\|_{\mathscr G_{q}^h}&=\max_{i\in\{1,\ldots,N\}}\big\|\max_{j\in\{0,1,
					\ldots,n_h\}}\big|X_j^{i, N}-X_j^{i, N,h}\big|\big\|_{\mathscr{L}^{q}(\Omega)}\leq C_7\big\|\mathcal R[X^{\varrho_h}]\big\|_{\mathscr G_{S,q}^{h}},
			\end{align*}
			 which  on using Proposition \ref{prop:consistency} completes the proof.
		\end{proof}

			\appendix
			\section{Auxiliary results}

			The following lemma is the discrete Burkholder--Davis--Gundy inequality, see \cite{Burkholder1966}.
			\begin{lem}[\bf Discrete Burkholder--Davis--Gundy Inequality] \label{lem:Burkholder}
				Let $\{M_n\}_{n\in \bar{\mathbb N}}$ be a discrete-time martingale on a probability space $(\Omega, \mathscr{F}, \mathbb{P})$ with respect to the filtration $\{\mathscr{F}_n\}_{n\in \bar{\mathbb{N}}}$. Then,  there exist constants $\underbar{C}_{p}, \bar{C}_{p}>0$  such that,
				\begin{align*}
					\underbar{C}_{p}\|[M]_n^{1/2}\|_{\mathscr{L}^{p}(\Omega)}
					\leq \|\max_{j\in\{0,1,\cdots,n\}}|M_j|\|_{\mathscr{L}^{p}(\Omega)}
					\leq \bar{C}_{p}\|[M]_n^{1/2}\|_{\mathscr{L}^{p}(\Omega)},
				\end{align*}
				for any $p\geq 1$  where $\{[M]_n\}_{n\in \bar{\mathbb{N}}}$ is the quadratic variation  process of $\{M_n\}_{n\in \bar{\mathbb{N}}}$.
			\end{lem}

			The discrete version of  the Gr\"onwall's inequality is given below, see Proposition 4.1~in~ \cite{Emmrich1999}.
			\begin{lem}[\bf Discrete Gr\"onwall's Inequality] \label{lem:Gronwall}
				Let $\{y_n\}_{n\in \mathbb{N}}$ and $\{z_n\}_{n\in\mathbb N}$ be sequences  of non-negative  real numbers satisfying $y_n\leq a+\displaystyle \sum_{j=1}^{n-1}z_jy_j$
				for all $n\in\mathbb{N}$ where $a>0$ is a constant.
				Then, $y_n\leq a\exp\big(\displaystyle \sum_{j=1}^{n-1}z_j\big)$ for all $n\in\mathbb{N}$.
			\end{lem}

		\subsection{Proof of Proposition \ref{prop:sde_bound}}
			\label{appendix:prop:sde_bound}
		Let $t'\in[0,T]$. Using  \eqref{eq:sde} and  H\"older's inequality, we get 
		\begin{align*} 
			\sup_{t\in[0,t']}|X_t|^{\bar p}
			&\leq 4^{\bar p-1}\Big\{|X_0|^{\bar p} + T^{\bar p-1}\int_0^{t'} |b(s, X_s, \mathcal L^1(X_s))|^{\bar p}ds 
			\\
			& +    \sup_{t\in[0,t']}\big|\int_0^t \sigma_1 (s, X_s, \mathcal L^1(X_s))dW_s \big|^{ \bar p}  +    \sup_{t\in[0,t']}\big|\int_0^t \sigma_0 (s, X_s,\mathcal L^1(X_s))dW_s^{0}\big|^{\bar p}\Big\},	
		\end{align*}
		which	on using Burkholder--Gundy--Davis inequality (Theorem 7.2 in \cite{Mao2008}) and  Remark \ref{rem:linear} yields
		\begin{align*}
			&\big\|\sup_{t\in[0,t']}|X_t|\big\|^{\bar p}_{\mathscr{L}^{\bar p}(\tilde\Omega)}
			\leq  4^{\bar p-1} \Big \{\big\|X_0\big\|_{\mathscr{L}^{\bar p}(\tilde\Omega)}^{\bar p} + T^{\bar p-1}\int_{0}^{t'} \big\|b(s, X_s, \mathcal L^1(X_s))\big\|_{\mathscr{L}^{\bar p}(\tilde\Omega)}^{\bar p}ds
			\\
			& \quad +\Big(\frac{\bar{p}^3}{2(\bar{p}-1)}\Big)^{\bar{p}/2} T^{(\bar{p}-2)/2}\sum_{u=0}^1\int_{0}^{t'}\big\|\sigma_u(s, X_s,\mathcal L^1(X_s))\big\|_{\mathscr{L}^{\bar p}(\tilde\Omega)}^{\bar p}ds \Big \}
			\\
			& \leq 4^{\bar p-1}\big\|X_0\big\|_{\mathscr{L}^{\bar p}(\tilde\Omega)}^{\bar p} +4^{\bar p-1} 3^{\bar{p}-1}\bar{L}^{\bar{p}}\Big(T^{
			\bar{p}-1}+2\Big(\frac{\bar{p}^3}{2(\bar{p}-1)}\Big)^{\bar{p}/2} T^{(\bar{p}-2)/2} \Big)
			\\
			& \qquad \times \int_{0}^{t'}\big\{1+\big\|X_s\big\|_{\mathscr{L}^{\bar p}(\tilde\Omega)}^{\bar p}+\big\|\mathcal{W}_2(\mathcal L^1(X_s),\delta_0)\big\|_{\mathscr{L}^{\bar p}(\tilde\Omega)}^{\bar p}\big\}ds
			\\
			& \leq  4^{\bar p-1}\big\|X_0\big\|_{\mathscr{L}^{\bar p}(\tilde\Omega)}^{\bar p}+12^{\bar p-1}\bar{L}^{\bar{p}}\Big(T^{
 			\bar{p}}+2\Big(\frac{\bar{p}^3}{2(\bar{p}-1)}\Big)^{\bar{p}/2} T^{\bar{p}/2} \Big)
			\\
			&\qquad+ 12^{\bar p-1}\bar{L}^{\bar{p}}\Big(T^{
			\bar{p}-1}+2\Big(\frac{\bar{p}^3}{2(\bar{p}-1)}\Big)^{\bar{p}/2} T^{(\bar{p}-2)/2} \Big) 2 \int_{0}^{t'} \big\|\sup_{r\in[0,s]}|X_r|\big\|^{\bar p}_{\mathscr{L}^{\bar p}(\tilde\Omega)} ds
		\end{align*}
		for all $t'\in[0,T]$. Then, Gr\"onwall's inequality gives
			\begin{align*}
			\big\|\sup_{t\in[0,T]}|X_t|\big\|^{\bar p}_{\mathscr{L}^{\bar p}(\tilde\Omega)} \leq  
			&\Big(4^{\bar p-1}\big\|X_0\big\|_{\mathscr{L}^{\bar p}(\tilde\Omega)}^{\bar p}+12^{\bar p-1}\bar{L}^{\bar{p}}
			\Big(T^{\bar{p} }
					+2\Big(\frac{\bar{p}^3}{2(\bar{p}-1)}\Big)^{\bar{p}/2} T^{\bar{p}/2} \Big)\Big) 
			\\
			& \times  \exp\Big(12^{\bar p-1}\bar{L}^{\bar{p}}\Big(T^{
			\bar{p}}+2\Big(\frac{\bar{p}^3}{2(\bar{p}-1)}\Big)^{\bar{p}/2} T^{\bar{p}/2} \Big) 2\Big),
		\end{align*}		
		which completes the proof. \hfill $\square$
	
			\subsection{Proof of Lemma \ref{lem:dif_sde}}
			\label{lemma1}
			Recall Equation \eqref{eq:interact} and use Theorem 7.1 in \cite{Mao2008} along with Remark \ref{rem:linear}  to get the following, 
			\begin{align*}
				&\big\|X_t^{i,N}-X_{t'}^{i,N}\big\|^{\bar p}_{\mathscr{L}^{\bar p}(\tilde\Omega)} \leq  3^{\bar p-1}(t-t')^{\bar p-1}\int_{t'}^t \big\|b(s, X_s^{i,N}, \mu_s^{X,N})\big\|_{\mathscr{L}^{\bar p}(\tilde\Omega)}^{\bar p}ds
				\\
				&+3^{\bar p-1} \Big( \frac{\bar{p}(\bar{p}-1)}{2}\Big)^{\bar{p}/2}(t-t')^{(\bar{p}-2)/2} \sum_{u=0}^1 \int_{t'}^t\big\|\sigma_u(s, X_s^{i,N},\mu_s^{X,N})\big\|_{\mathscr{L}^{\bar p}(\tilde\Omega)}^{\bar p}ds
				\\
				\leq & 9^{\bar p-1}L^{\bar p}\Big((t-t')^{\bar p-1}+2\Big( \frac{\bar{p}(\bar{p}-1)}{2}\Big)^{\bar{p}/2}(t-t')^{(\bar p-2)/2}\Big)
				\\ 
				&
				\qquad \times\int_{t'}^t\big(1+\big\|X_s^{i,N}\big\|_{\mathscr{L}^{\bar p}(\tilde\Omega)}^{\bar p}+\big\|\mathcal{W}_2(\mu_s^{X,N},\delta_0)\big\|_{\mathscr{L}^{\bar p}(\tilde\Omega)}^{\bar p}\big)ds,
			\end{align*}
		for all $t> t'\in[0, T]$ and $i\in\{1,\ldots,N\}$,	which due to  Remark \ref{rem:interact_bound} yields, 
			\begin{align*}
				\big\|&X_t^{i,N}- X_{t'}^{i,N}\big\|^{\bar p}_{\mathscr{L}^{\bar p}(\tilde\Omega)}	\leq  9^{\bar p-1}L^{\bar p}\Big((t-t')^{\bar p-1}+2\Big( \frac{\bar{p}(\bar{p}-1)}{2}\Big)^{\bar{p}/2}(t-t')^{(\bar p-2)/2}\Big)
				\\
				& \times \int_{t'}^t\big\{1+2\max_{i \in \{1,\ldots,N\}}\big\|X_s^{i,N}\big\|_{\mathscr{L}^{\bar p}(\tilde\Omega)}^{\bar p}\big\}ds
				\\
				\leq & 9^{\bar p-1}L^{\bar p} \Big((t-t')^{\bar p/2}+2\Big( \frac{\bar{p}(\bar{p}-1)}{2}\Big)^{\bar{p}/2}\Big) (t-t')^{\bar p/2}\Big( 1+2C_1 \Big(1+\max_{i \in \{1,\ldots,N\}}\big\|X_0^{i,N}\big\|_{\mathscr{L}^{\bar p}(\tilde\Omega)}^{\bar p}\Big)\Big)
		\end{align*}			
			and thus the proof is completed. \hfill $\square$
			
		\subsection{Proof of Proposition \ref{prop:poc}}
			\label{appendix:prop:poc}
			
		By using \eqref{eq:noninteract} and \eqref{eq:interact} along with H\"older's inequality, 		one obtains
		\begin{align*} 
			\sup_{t\in[0,T]}&|X_t^{i}-X_{t}^{i,N}|^{2}\leq 3\Big\{ T\int_0^T |b(s, X_s^i, \mathcal L^1(X_s^1))-b(s, X_s^{i,N}, \mu_s^{X,N})|^{2}ds 
			\\
			&+   \sup_{t\in[0,T]}\big|\int_0^t \big(\sigma_1 (s, X_s^i, \mathcal L^1(X_s^1))-\sigma_1(s, X_s^{i,N}, \mu_s^{X,N})\big)dW_s \big|^{2} \notag
			\\
			&+ \sup_{t\in[0,T]}\big|\int_0^t \big ( \sigma_0 (s, X_s^i,\mathcal L^1(X_s))-\sigma_0(s, X_s^{i,N}, \mu_s^{X,N})\big)dW_s^{0}\big|^{2}\Big\}, 
		\end{align*}
		which on using the martingale inequality (Theorem 7.2 in \cite{Mao2008}) and Assumption H--\ref{asump:lip} yields 
		\begin{align*}
			\big\|&\sup_{t\in[0,t']} |X_t^{i}-X_{t}^{i,N}|\big\|_{\mathscr{L}^{2}(\tilde \Omega)}^2 \leq  3T\int_{0}^{t'} \big\|b(s, X_s^i, \mathcal L^1(X_s^1))-b(s, X_s^{i,N}, \mu_s^{X,N})\big\|_{\mathscr{L}^{2}(\tilde \Omega)}^{2}ds
			\\
			&\quad+12 \sum_{u=0}^1 \int_{0}^{t'}\big\|\sigma_u (s, X_s^i, \mathcal L^1(X_s^1))-\sigma_u(s, X_s^{i,N},\mu_s^{X,N})\big\|_{\mathscr{L}^{2}(\tilde \Omega)}^{2}ds
			\\
			&  \leq 2(3T+24) L^2\int_0^{t'} \big \{ \big\| X_s^i - X_s^{i,N}\big\|_{\mathscr{L}^{2}(\tilde \Omega)}^{2}+\big\| \mathcal W_2 (\mathcal L^1(X_s^1),\mu_s^{X,N}\big\|_{\mathscr{L}^{2}(\tilde \Omega)}^{2}\big \}ds
			\\
			& \leq 2(3T+24) L^2 \int_{0}^{t'} \big \{ \big\| X_s^i - X_s^{i,N}\big\|_{\mathscr{L}^{2}(\tilde \Omega)}^{2}+2\big\| \mathcal W_2 (\mathcal L^1(X_s^1),\mu_s^{X}\big\|_{\mathscr{L}^{2}(\tilde \Omega)}^{2}
			+2\big\| \mathcal W_2 (\mathcal \mu_s^{X},\mu_s^{X,N}\big\|_{\mathscr{L}^{2}(\tilde \Omega)}^{2}\big \}ds
			\\
			& \leq 2(3T+24) L^2 \int_{0}^{t'} \big \{ 3 \max_{i\in\{1,\ldots,N\}}\big\| X_s^i - X_s^{i,N}\big\|_{\mathscr{L}^{2}(\tilde \Omega)}^{2}+2\big\| \mathcal W_2 (\mathcal L^1(X_s^1),\mu_s^{X}\big\|_{\mathscr{L}^{2}(\tilde \Omega)}^{2}\big \}ds
		\end{align*}
		for any $t'\in [0,T]$ and $N\in \mathbb{N}$. 	
		An application of Gr\"onwall's inequality yields
		\begin{align*}
			\max_{i\in\{1,\ldots,N\}}\big\|\sup_{t\in[0,T]} |X_t^{i}-X_{t}^{i,N}|\big\|_{\mathscr{L}^{2}(\tilde \Omega)}^2 \leq  4(3T+24)L^2 e^{6(3T+24)TL^2}\int_{0}^T \big\| \mathcal W_2 (\mathcal L^1(X_s^1),\mu_s^{X}\big\|_{\mathscr{L}^{2}(\tilde \Omega)}^{2} ds,
		\end{align*}
		for any $N\in \mathbb{N}$.  
		The proof is completed by using \eqref{eq:meaure:rate}.
\hfill $\square$		

			\subsection{Lions Derivative and Particle Projection Function Inequalities} 
			\label{lions}
			Given a function $f:\mathscr{P}_2(\mathbb{R}^d)\to\mathbb{R}$ and $\nu_0\in\mathscr{P}_2(\mathbb{R}^d)$, we say that $f$ is  Lions differentiable at $\nu_0$ if we can find an atomless, Polish probability space $(\tilde{\Omega}, \tilde{\mathcal{F}}, \tilde{P})$ and a random variable $Y_0\in\mathscr{L}^2(\tilde{\Omega})$ with  law $\mathcal{L}^1({Y_0}):=\tilde{P}\circ Y_0^{-1}=\nu_0$ such that 
			the ``lift'' function $F:\mathscr{L}^2(\tilde{\Omega})\to \mathbb{R}$ defined by $F(Z):=f(\mathcal{L}^1(Z))$ has  Fr\'echet derivative $F'[Y_0]$ at $Y_0\in\mathscr{L}^2(\tilde{\Omega})$.   By Riesz representation theorem, we can find  $DF(Y_0)\in \mathscr{L}^2(\tilde{\Omega})$ such that $F'[Y_0](Z)=\tilde{E} \langle DF[Y_0], Z\rangle$  for all $Z\in\mathscr{L}^2 (\tilde{\Omega};\mathbb{R}^d)$.  
			Further, Theorem 6.5 (structure of the gradient) in \cite{Cardaliaquet2013} guarantees the existence of  a function $\partial_\mu f(\nu_0): \mathbb{R}^d\to\mathbb{R}^d$, independent of the choice of the probability space $(\tilde{\Omega}, \tilde{\mathcal{F}}, \tilde{P})$ and the random variable $Y_0$,  satisfying $
			\displaystyle\int_{\mathbb{R}^d} |\partial_\mu f(\nu_0)(x)|^2 \nu_0(dx) <\infty
			$ such that $DF(Y_0)=\partial_\mu f(\nu_0)(Y_0)$. 
			Then, we call $\partial_\mu f(\nu_0)$  as  \textit{Lions' derivative} of $f$ at $\nu_0=\mathcal{L}^1({Y_0})$.	
%

By using similar arguments used in Lemma 4.2 in \cite{Kumar2021}, one can prove the following lemma. 
			\begin{lem}\label{lem:mvt}
			Let $g$ be a real valued function defined on $[0,T]\times\mathbb R^d\times\mathscr{P}_2(\mathbb{R}^d)$ such that derivative  with respect to the state variable  $\partial_x g(\cdot,\cdot,\cdot)$ and  with respect to the measure variable  $\partial_\mu g(\cdot,\cdot,\cdot, \cdot)$ satisfy Lipschitz condition uniformly in time, i.e., there exists a constant $L>0$ such that, 
			\begin{align*}
				|\partial_xg(t,x,\mu)-\partial_xg(t,\bar x,\bar \mu)|
				&\leq L\{|x-\bar x|+\mathcal W_2(\mu,\bar \mu)\},
				\\
				|\partial_\mu g(t,x,\mu,y)-\partial_\mu g(t,\bar x,\bar \mu,\bar y)|
				&\leq L\{|x-\bar x|+\mathcal W_2(\mu,\bar \mu)+|y-\bar y|\},
			\end{align*}
			for all $t\in[0,T]$, $x,\bar x,y,\bar y\in \mathbb R^d$	 and $\mu,\bar \mu\in\mathscr{P}_2(\mathbb{R}^d)$. Then, for all $t\in[0,T]$, $x^i$, $\bar x^i\in\mathbb R^d$  and $i\in\{1,\ldots,N\}$, the following holds, 
			\begin{align*}
				&\big|g\big(t, x^i,\frac{1}{N}\sum_{k=1}^N\delta_{x^k}\big) - g\big(t, \bar x^i,\frac{1}{N}\sum_{k=1}^N\delta_{\bar x^k}\big) - \partial_xg\big(t,\bar x^{i},\frac{1}{N}\sum_{k=1}^N\delta_{\bar x^k}\big)(x^i-\bar x^i)
				\\
				&-\frac{1}{N}\sum_{k= 1}^{N}\partial_\mu g\big(t,\bar x^{i},\frac{1}{N}\sum_{k=1}^N\delta_{\bar x^k},\bar x^k\big)(x^k-\bar x^k)\big|\leq  \frac{3L}{2}|x^i-\bar x^i|^2+\frac{5L}{2}\frac{1}{N}\sum_{ k=1}^N|x^k-\bar x^k|^2.
			\end{align*}  
		\end{lem}

\begin{acks}[Acknowledgments]
The authors would like to thank the anonymous referees whose comments improved the quality of this paper.
\end{acks}

\begin{funding}
G.~dos Reis acknowledges support from the \emph{Funda{\c c}\~ao para a Ci$\hat{e}$ncia e a Tecnologia} (Portuguese Foundation for Science and Technology) through the project UIDB/00297/2020 and UIDP/00297/2020 (Center for Mathematics and Applications).
\end{funding}

			\subsection*{Authors' Addresses}\hfill\\
		
				\noindent Sani Biswas, Department of Mathematics, Indian Institute of Technology Roorkee,  Roorkee, 247 667, India.  \\{\tt 
				sbiswas2@ma.iitr.ac.in}\\ 
				
	\noindent Chaman Kumar, Department of Mathematics, Indian Institute of Technology Roorkee,  Roorkee, 247 667, India.  \\{\tt 
				chaman.kumar@ma.iitr.ac.in}\\				
				
			\noindent Neelima, Department of Mathematics, Ramjas College, University of Delhi, Delhi, 110 007, India. \\{\tt 
				neelima\_maths@ramjas.du.ac.in} \\

			\noindent Gon\c calo dos Reis, School of Mathematics, University of Edinburgh, Peter Guthrie Tait Road, Edinburgh, EH9 3FD, United Kingdom, and Centro de Matem\'atica e Aplica\c c \~oes (CMA), FCT, UNL, Portugal. \\{\tt
				G.dosReis@ed.ac.uk}\\

			\noindent Christoph Reisinger, Mathematical Institute, University of Oxford.
			Andrew Wiles Building, Woodstock Road, Oxford, OX2 6GG, UK. \\{\tt christoph.reisinger@maths.ox.ac.uk}\\


		\end{document}